\documentclass{elsarticle}

\usepackage{amsmath}
\usepackage{amsthm}
\usepackage{amssymb}
\usepackage{mathtools}
\usepackage{cancel}

\usepackage{booktabs}
\usepackage{algorithmicx}
\usepackage{array}
\usepackage{multirow}
\usepackage{caption}
\usepackage{xfrac}
\usepackage{enumitem}
\usepackage{longtable}
\usepackage{geometry}
\usepackage{graphicx}
\usepackage{xcolor}
\usepackage[T1]{fontenc}
\usepackage[utf8]{inputenc}

\usepackage{algpseudocode}
\usepackage{setspace} 
\usepackage{float}

\usepackage{pgfplots}
\usepackage{subcaption}
\pgfplotsset{compat=1.16}
\usetikzlibrary{arrows.meta,calc,positioning}
\usepgfplotslibrary{colorbrewer}
\pgfplotsset{
table/row sep=\\,
cycle list/Set1-5,
ignore legend/.style={
    every axis legend/.code={\renewcommand\addlegendentry[2][]{}}
}
}
\usepgfplotslibrary{external}
\usepackage{hyperref} %
\usepackage[capitalise]{cleveref} %

\newtheorem{thm}{Theorem}
\newtheorem{prop}[thm]{Proposition}
\newtheorem{rem}[thm]{Remark}
\newtheorem{defn}[thm]{Definition}
\newtheorem{lem}[thm]{Lemma}
\newtheorem{cor}[thm]{Corollary}
\newtheorem{ex}[thm]{Example}
\newtheorem{assn}[thm]{Assumption}

\Crefname{thm}{Theorem}{Theorems}
\Crefname{defn}{Definition}{Definitions}
\Crefname{rem}{Remark}{Remarks}
\Crefname{prop}{Proposition}{Propositions}
\Crefname{lem}{Lemma}{Lemmas}
\Crefname{section}{Section}{Sections}

\renewcommand{\namecref}{\lcnamecref}
\renewcommand{\namecrefs}{\lcnamecrefs}

\makeatletter
\newcommand{\leqnomode}{\tagsleft@true\let\veqno\@@leqno}
\newcommand{\reqnomode}{\tagsleft@false\let\veqno\@@eqno}
\makeatother 

\newfloat{algorithm}{h}{loa}
\floatname{algorithm}{Algorithm}

\usepackage{enumitem}
\usepackage{marginnote}
\definecolor{Green}{rgb}{0,0.9,0.5}
\definecolor{Red}{rgb}{0.8,0,0}
\definecolor{Blue}{rgb}{0,0,0.8}
\newif\ifShowCommentsAndColors
\ShowCommentsAndColorsfalse
\ifShowCommentsAndColors
  \newcommand{\reply}[1]{{\itshape\color{Blue}{#1}}}
  \newcommand{\added}[1]{{\phantomsection\color{Red}#1}} %
  \newcommand{\lookUp}[1]{\marginnote[{\color{Blue}\normalfont\scriptsize cf.\,\textbf{#1}}]{{\color{Blue}\normalfont\scriptsize cf.\,\textbf{#1}}}}
  \newcommand{\showCommentForReviewers}[1]{#1}
  \newcommand{\removed}[1]{}
\else
  \newcommand{\reply}[1]{}
  \newcommand{\added}[1]{#1}
  \newcommand{\lookUp}[1]{}
  \newcommand{\showCommentForReviewers}[1]{}
  \newcommand{\removed}[1]{}
\fi

\newcommand{\R}{\mathbb{R}}
\newcommand{\N}{\mathbb{N}}
\newcommand{\Z}{\mathbb{Z}}

\newcommand{\sphere}{S}
\newcommand{\sphered}{\sphere^{d-1}}
\newcommand{\thins}{\theta\in\sphered}

\newcommand{\domstat}{\Omega}
\newcommand{\domdyn}{\Lambda}
\newcommand{\projdomstat}{\Xi}
\newcommand{\projdomdyn}{\Gamma}
\newcommand{\domt}{\Sigma}

\newcommand{\alltimes}{{\mathcal T}}
\newcommand{\goodtimes}{{\mathcal T'}}
\newcommand{\ngoodtimes}{{r}}

\newcommand{\alldirs}{\Theta}
\newcommand{\gooddirs}{\Theta'}

\renewcommand{\emptyset}{\varnothing}

\newcommand{\M}{\mathcal{M}}
\newcommand{\Mp}{\mathcal{M}_+}

\newcommand{\fopstat}{\mathrm{Ob}}
\newcommand{\fopdyn}{\overline\fopstat}
\newcommand{\fopdynp}{\underline\fopstat}

\newcommand{\BregmanDistance}[2]{D^{#1}_{#2}}

\newcommand{\Rs}{ \mathrm{Rd} } %
\newcommand{\Rj}{ \mathrm{Rj} } %
\newcommand{\Rf}{ \overline{\Rs} } %
\newcommand{\Rfp}{ \underline{\Rs} } %
\newcommand{\Rfm}{ \overline{\underline{\Rs}} } %

\newcommand{\mv}{ \mathrm{Mv} } %
\newcommand{\amv}[1][]{ \overline{\mv^{#1}} } %
\newcommand{\amvt}[2][]{ \overline{\mv^{#1}_{#2}} } %
\newcommand{\amvp}[1][]{ \underline{\mv^{#1}} } %
\newcommand{\amvm}[1][]{ \overline{\underline{\mv^{#1}}} } %

\newcommand{\ft}{ \mathcal{F} } %
\newcommand{\ftrunc}{ \mathcal{F}_{\maxFrequency} } %

\newcommand{\fidelity}[1]{F_{#1}}
\newcommand{\regularizer}{G}
\newcommand{\energy}{E}

\newcommand{\prodm}[1]{ {\otimes}_{#1}}

\newcommand{\sdot}[2]{#1\cdot #2}
\newcommand{\thdot}[1]{\sdot{\theta}{#1}}

\newcommand{\data}{f}
\newcommand{\noisyData}{f^\delta}
\newcommand{\noiseFreeData}{f^\dagger}
\newcommand{\approximation}{y^\delta}
\newcommand{\groundTruth}{y^\dagger}
\newcommand{\optimalDual}{w^\dagger}

\newcommand{\ghostDistance}{\Delta'}
\newcommand{\minSeparation}{\Delta}
\newcommand{\maxFrequency}{\Phi}

\newcommand{\liftVar}{\gamma}

\newcommand\problemTagNoLink[2][\alpha]{$\mathrm P_{#1}(#2)$}
\newcommand\problemTag[2][\alpha]{\hyperref[eq:dim_reduced_general]{\problemTagNoLink[#1]{#2}}}
\newcommand\PIProblemTagNoLink[2][\alpha]{$\Pi$-$\mathrm P_{#1}(#2)$}
\newcommand\PInoisyProblemTag[2][\alpha]{\hyperref[eq:dim_reduced_product_only_noisy]{\PIProblemTagNoLink[#1]{#2}}}

\newcommand{\dsupp}{S}

\newcommand{\lebesgue}{\mathcal{L}}
\newcommand{\hd}{\mathcal H} %
\newcommand{\mtime}{\Hc_\domt}
\newcommand{\mdirs}{\Hc_\Theta}

\newcommand{\weakstarto}{\stackrel*\rightharpoonup}
\newcommand{\narrowto}{\overset{\textbf{n}}{\rightharpoonup}}
\newcommand{\restr}{{\mbox{\LARGE$\llcorner$}}}
\newcommand{\RadonNikodym}[2]{\frac{\mathrm d#1}{\mathrm d#2}}

\newcommand{\wrt}{\:\mathrm{d}}
\newcommand{\pf}[1]{{#1}_\#}

\newcommand{\unbalancedWasserstein}[2][2]{W_{#1,#2}^{#1}}

\newcommand{\weight}{h}

\DeclareMathOperator{\sgn}{sgn}

\newcommand{\e}{{\mathrm e}}
\newcommand{\ii}{{\mathrm i}} %
\renewcommand{\d}{{\mathrm d}}

\newcommand{\Bc}{ \mathcal{B} }

\newcommand{\Hc}{ \mathcal{H} }

\newcommand{\Cc}{ \mathcal{C} }

\DeclarePairedDelimiter{\mnorm}{\lVert}{\rVert_\M}
\DeclarePairedDelimiter{\norm}{\lVert}{\rVert}

\newcommand{\notinclude}[1]{}

\DeclareMathOperator{\supp}{supp}
\DeclareMathOperator{\range}{Rg}

\DeclareMathOperator{\domain}{\text{dom}}
\DeclareMathOperator{\card}{card}

\DeclareMathOperator{\vspan}{span}

\DeclareMathOperator{\dist}{dist}

\newcommand{\st}{ \, \left|\right.\, }
\DeclarePairedDelimiterX\set[1]\lbrace\rbrace{\def\given{\;\delimsize\vert\;}#1}
\DeclarePairedDelimiterX{\abs}[1]{|}{|}{#1}

\newlist{assumptions}{enumerate}{1}
\setlist[assumptions]{label=\arabic*)}
\crefname{assumptionsi}{assumption}{assumptions}

\newcommand{\gridmu}[1]{G_\liftVar(#1)}
\newcommand{\gridu}[1]{G_u(#1)}
\newcommand{\gridr}[2]{G_r(#1,#2)}

\newcommand{\matmu}[2]{M_\mv(#1,#2)}
\newcommand{\matu}[2]{M_\Rs(#1,#2)}
\newcommand{\measmat}[1]{M_{\fopstat}(#1)}

\newcommand{\muvec}[1]{\vec{\liftVar}_{#1}}
\newcommand{\uvec}[1]{\vec{u}_{#1}}

\newcommand{\nod}{M}

\newcommand{\redconsbnd}{\tau}
\newcommand{\clusterthresh}{w_{\mathrm{min}}}

\newcommand{\dataset}[1]{$D_{#1}$}
\newcommand{\noiseconst}{C_{\alpha}}
\newcommand{\dynsep}{\Delta_{\mathrm{dyn}}}
\newcommand{\maxT}{T}

\newcommand{\domdynone}{\domdyn_1}
\DeclareMathOperator{\conv}{conv}
\newcommand{\domdynnum}{\tilde{\domdyn}}

\begin{document}
\newpage\pagenumbering{arabic}
\setcounter{page}{1}
\begin{frontmatter}
\title{Dimension reduction, exact recovery, and error estimates for sparse reconstruction in phase space}

\author[label1]{M. Holler\fnref{fn1}}
\ead{martin.holler@uni-graz.at}

\author[label2]{A. Schlüter\corref{cor1}}
\ead{alex.schlueter@uni-muenster.de}

\author[label2]{B. Wirth}
\ead{benedikt.wirth@uni-muenster.de}

\fntext[fn1]{MH is a member of NAWI Graz (\url{https://www.nawigraz.at/en}) and BioTechMed Graz (\url{https://biotechmedgraz.at/en}).}
\cortext[cor1]{Corresponding author}

\affiliation[label1]{organization={Institute of Mathematics and Scientific Computing},
            addressline={Heinrichstraße 36},
            city={Graz},
            postcode={A-8010},
            country={Austria}}
\affiliation[label2]{organization={Institute for Applied Mathematics, University of Münster},
            addressline={Orléans-Ring 10},
            city={Münster},
            postcode={48149},
            country={Germany}}

\begin{abstract}
An important theme in modern inverse problems
is the reconstruction of \emph{time-dependent} data from only \emph{finitely many} measurements.
To obtain satisfactory reconstruction results in this setting
it is essential to strongly exploit temporal consistency between the different measurement times.
The strongest consistency can be achieved by reconstructing data directly in \emph{phase space},
the space of positions \emph{and} velocities.
However, this space is usually too high-dimensional for feasible computations.
We introduce a novel dimension reduction technique, based on projections of phase space onto lower-dimensional subspaces, which provably circumvents this curse of dimensionality:
Indeed, in the exemplary framework of superresolution we prove that known exact reconstruction results stay true after dimension reduction,
and we additionally prove new error estimates of reconstructions from noisy data in optimal transport metrics
which are of the same quality as one would obtain in the non-dimension-reduced case.
\end{abstract}

\end{frontmatter}

\tableofcontents

\section{Introduction}\label{sec:intro}

In several practically relevant inverse problems, the objects to be reconstructed are moving,
and one is interested in information on that motion in addition to the shape and distribution of the objects themselves.
For instance, one might try to reconstruct the distribution of radioactively labelled blood cells and their velocities in a patient from positron emission tomography measurements.
Simultaneous reconstruction of objects and motions is typically a highly nonlinear task since the forward operator involves the composition of the objects with the motion.
Sometimes, though, nonlinear problems can be transformed into convex optimization problems via so-called functional lifting.
The underlying idea is to replace a $Y$-valued variable with a (probability) measure on $Y$.
In imaging, this was first applied to the Mumford--Shah functional \cite{AlbertiBouchitteDalMaso2003}, but many other models followed.
In the context of dynamic inverse problems the approach of functional lifting was pursued by Alberti et al.\ \cite{AlbertiAmmariRomeroWintz2019}
for the reconstruction of linear particle trajectories from indirect observations.
The resulting convex program is very high-dimensional, though.
In this work we propose a technique to reduce the number of these dimensions without sacrificing too much of the reconstruction properties.
Essentially, using different variants of the Radon transform,
we replace the original reconstruction problem by a collection of reconstruction problems with only one spatial dimension.
Here we analyse this technique completely for the reconstruction of linear particle trajectories,
including results on exact reconstruction without noise and error estimates for reconstruction from noisy data,
but the technique is more general and may in principle be applied to other settings as well.

\subsection{Functional lifting and novel dimension reduction for reconstruction of linear particle trajectories} \label{sec:intro_dim_red}

We first briefly introduce a classical method for reconstruction of stationary particles as analysed by Cand\`es and Fernandez-Granda in \cite{Candes2014}
and then the method for moving particles proposed by Alberti et al.\ in \cite{AlbertiAmmariRomeroWintz2019},
to which we finally apply our dimension reduction technique.
The presentation in this section is expository; all involved spaces and operators will be introduced more rigorously at a later point.

Stationary particles of mass $m_i>0$ at location $x_i\in\R^d$ can be described
as a linear combination $u=\sum_{i=1}^Nm_i\delta_{x_i}$ of Dirac measures in $\R^d$.
If an observation $\data$ of this particle configuration is obtained with a \emph{forward} or \emph{observation operator} $\fopstat$,
then one may try to reconstruct $u$ from the observation by solving
\begin{equation}\label{eqn:staticModel}
\min_{u \in \Mp(\R^d)} \| u \|_\M \quad \text{such that } \fopstat u =  \data.
\end{equation}
(We denote by $\M(X)$ and $\Mp(X)$ the signed and the nonnegative Radon measures on some domain $X$ \added{\lookUp{\ref{itemd}}and by $\| \cdot \|_\M$ the total variation norm on Radon measures}).
In \cite{Candes2014} Cand\`es and Fernandez-Granda consider the exemplary case where $u$ is supported in the unit square and $\fopstat$ yields a finite number of Fourier coefficients.
\added{\label{txt:item1}\lookUp{\ref{item1}}In our setting, this would mean to take $\fopstat u=\ftrunc u$, where $(\ftrunc u)_\xi=\int_{[0,1]^d} e^{-\ii2\pi x\cdot\xi}\wrt u(x)$ for all frequencies $\xi\in\Z^d$ with $\norm{\xi}_\infty\leq\maxFrequency$, but we will consider a general forward operator in our model}.
Under a bound on the minimum particle distance Cand\`es and Fernandez-Granda prove that the solution of \eqref{eqn:staticModel} yields exact reconstruction of the ground truth
(in fact they even allow complex-valued masses).
In case of noisy observations $\data$ they replace the constraint $\fopstat u = \data$ with a fidelity term and obtain error estimates for the reconstruction \cite{Candes2013}.

In \cite{AlbertiAmmariRomeroWintz2019} Alberti et al.\ generalize the above reconstruction problem to the case of \emph{moving} particles with observations at \emph{multiple} points in time.
The motivation is twofold.
First, one hopes to gain motion information in addition to the mere particle reconstruction at the different observation times.
Second, the observations at multiple times may help to improve the reconstruction quality at every single point in time.
Alberti et al.\ describe a dynamic particle configuration as a linear combination $\lambda=\sum_{i=1}^Nm_i\delta_{(x_i,v_i)}$\label{txt:highDimMeasure} of Dirac measures in $\R^d\times\R^d$,
where this time $x_i\in\R^d$ is the initial particle position and $v_i\in\R^d$ the particle velocity.
They, too, allow complex-valued masses, but our reduction technique relies to some extent on the nonnegativity of the measure
(their numerical optimization only applies to nonnegative measures as well).
\added{\label{txt:item2}\lookUp{\ref{item2},\ref{itemd}}In the following, we will use the notation $\pf{g}\nu$ to denote the pushforward measure of $\nu\in\M(X)$ under a measurable map $g\colon X\to Y$, that is the measure on $Y$ defined by $\pf{g}\nu(A)=\nu(g^{-1}(A))$ for all measurable sets $A$.}
The spatial configuration at time $t$ is then obtained by a \emph{``move''-operator}, the pushforward of $\lambda$ under $(x,v)\mapsto x+tv$,
\begin{equation*}
\mv^d_t:\M(\R^d\times\R^d)\to\M(\R^d),\,\mv^d_t\lambda=\pf{[(x,v)\mapsto x+tv]}\lambda\added{,}
\end{equation*}
\added{\lookUp{\ref{item1}}which, when applied to a discrete measure $\lambda=\sum_{i=1}^Nm_i\delta_{(x_i,v_i)}$, simply yields
\begin{equation*}
    \mv^d_t\lambda=\sum_{i=1}^Nm_i\delta_{x_i+tv_i}.
\end{equation*}%
}%
Similarly to the stationary case, if at \added{\lookUp{\ref{item5},\ref{iteme}}a finite number of} times $t\in\alltimes\subset\R$ observations $\data_t$ of this spatial configuration are taken with a forward operator $\fopstat_t$,
then the authors reconstruct $\lambda$ by solving
\begin{equation}\label{eqn:highDimensionalModel}
\min_{\lambda \in \Mp(\R^d\times\R^d)} \| \lambda \|_\M \quad \text{such that } \fopstat_t\mv^d_t \lambda =  \data_t \quad \text{for all }t \in \alltimes.
\end{equation}
Based on the theory by Cand\`es and Fernandez-Granda,
Alberti et al.\ show exact reconstruction of particle trajectories in the noise-free case
and in the discretized setting stability of their reconstruction with respect to noise (again replacing the equality constraint by a fidelity term).

The sought Radon measure $\lambda$ lives in a high-dimensional space and thus can typically not be discretized to solve the optimization problem.
As a remedy, the numerical optimization in \cite{AlbertiAmmariRomeroWintz2019} is performed by iteratively placing Dirac measures into $\R^d\times\R^d$ via the Frank--Wolfe algorithm.
However, it remains challenging to generalize this algorithm to non-Dirac measures (see for instance \cite{bredies2020cond_gradient_opt_transport_mh} for a recent approach 
that works with Dirac measures moving along absolutely continuous curves 
using optimal transport regularization).
For this reason we suggest the following dimension reduction.
For a unit vector $\theta\in\sphere^{d-1}$ we introduce the \emph{Radon transform} and the \emph{joint Radon transform} as
\begin{align*}
\Rs_{\theta}&:\M(\R^d)\to\M(\R),&&\Rs_\theta\nu=\pf{[x\mapsto\theta\cdot x]}\nu,\\
\Rj_{\theta}&:\M(\R^d\times\R^d)\to\M(\R^2),&&\Rj_\theta\lambda=\pf{[(x,v)\mapsto(\theta\cdot x,\theta\cdot v)]}\lambda.
\end{align*}
Intuitively, the joint Radon transform projects the system onto the span of $\theta$
so that one only observes the one-dimensional positions $\theta\cdot x$ and velocities $\theta\cdot v$ of any moving object.
Our new, dimension-reduced variables will then be the \emph{snapshots} $u_t\in\Mp(\R^d)$ of the system at a collection $\domt\supset\alltimes$ of times,
\begin{equation*}\phantomsection\label{eqn:snapshots}
u_t=\mv^d_t\lambda,\quad t\in\domt,
\end{equation*}
and the collection of one-dimensional \emph{position-velocity projections} $\liftVar_\theta\in\Mp(\R^2)$ for a set $\alldirs\subset\sphere^{d-1}$ of directions,
\begin{equation*}
\liftVar_\theta=\Rj_\theta\lambda,\quad \theta\in\alldirs.
\end{equation*}
To ensure that the snapshots and position-velocity projections are compatible with each other we make use of the identity $\Rs_\theta\mv^d_t=\mv^1_t\Rj_\theta$,
thus $\Rs_\theta u_t=\mv^1_t\liftVar_\theta$ for all $t\in\domt$, $\theta\in\alldirs$.
\added{\label{txt:item3}\lookUp{\ref{item3}}%
Collecting the position-velocity projections $\liftVar_\theta\in\Mp(\R^2)$ for $\theta \in \alldirs$ in a variable $\liftVar \in \Mp(\alldirs \times \R^2)$ and using further that the nonnegativity of $\lambda$ ensures $\|\lambda\|_\M=\|\liftVar\|_\M$,    as a dimension-reduced version of \eqref{eqn:highDimensionalModel} we suggest to solve
}
\removed{Furthermore, the nonnegativity of $\lambda$ ensures $\|\lambda\|_\M=\|\liftVar_{\hat\theta}\|_\M$ for an arbitrarily fixed $\hat\theta\in\alldirs$.
Summarizing, as a dimension-reduced version of \eqref{eqn:highDimensionalModel} we suggest to solve}
\begin{equation}\label{eqn:lowDimensionalModel}
\min_{ \substack{ \liftVar \in \Mp(\alldirs \times \R^2) \\ u \in \Mp(\R^d)^{\domt} }}
\| \liftVar \|_\M \quad \text{such that }
\begin{cases}
\mv_t^1 \liftVar_\theta = \Rs_\theta u_t &\text{for all }t \in \domt,\theta\in\alldirs, \\
\fopstat_tu_t = \data_t &\text{for all }t \in \alltimes,
\end{cases}
\end{equation}
\added{where the slice $\liftVar_\theta$ is defined by the relation $\liftVar=\mdirs(\theta)\otimes\liftVar_\theta$, with  $\mdirs$ a probability measure on $\alldirs$.}
\added{\label{txt:itemd}\lookUp{\ref{itemd}}Above, $\Mp(X)^I$ denotes the (possibly uncountably infinite) Cartesian product of spaces of nonnegative Radon measures for an index set $I$.}
In the case of noisy data $\data_t$, the constraint $\fopstat_tu_t = \data_t$ is again replaced with a fidelity term.
Note that $\liftVar$ and $u$ can be interpreted as measures on the $(d+1)$-dimensional manifolds $\alldirs \times \R^2\subset\sphere^{d-1}\times\R^2$ and $\R^d\times\domt\subset\R^d\times\R$, respectively,
which is indeed a reduction compared to the $2d$-dimensional domain of the original variable $\lambda$.

\removed{In fact, there are different versions of the above dimension-reduced problem,
depending on whether one interprets $\liftVar$ and $u$ as elements of product spaces as in \eqref{eqn:lowDimensionalModel} or as Radon measures on product manifolds.
In this work we will
1) analyse the relation between these different versions,
2) prove exact reconstruction of particles in the noise-free case, and
3) stable reconstruction of particles from noisy data with convergence rates.}
\added{\lookUp{\ref{item3}}
Besides ensuring well-posedness of the dimension-reduced problem, in this work we will prove 1) exact reconstruction of particles in the noise-free case and 2) stable reconstruction of particles from noisy data with convergence rates.}
The conditions for exact and for stable reconstruction are essentially the same as for the high-dimensional model \eqref{eqn:highDimensionalModel},
so the dimension reduction did not cause any deterioration of the reconstruction properties. \added{
\lookUp{\ref{item3}}In addition to these main results, we also prove equivalence of different variants of the reduced problem \eqref{eqn:lowDimensionalModel}
in which the snapshots and position-velocity projections are taken from different spaces.}
The theoretical results are supported by numerical experiments.

\subsection{Main results}

To model a more realistic setting and to avoid technicalities associated with Radon measures on unbounded or open domains, from now on we restrict ourselves to compact domains
(though in principle one may also work with infinite domains for position, velocity, and time).
In detail, the moving particles are assumed to stay within the compact region $\domstat$,
and the sets $\domt\subset\R$ of times and $\alldirs\subset\sphere^{d-1}$ of directions considered in our dimension-reduced model are also assumed compact.
Exemplary choices are $\domt=[-1,1]$ and $\alldirs=\sphere^{d-1}$ or finite sets.
For later convenience we equip both $\Sigma$ and $\alldirs$ with probability measures $\mtime$ and $\mdirs$.
Finally, we choose a compact $\projdomdyn\subset\R^2$ to represent all allowed pairs of (one-dimensional) projected positions and velocities.

For simplicity we assume the forward operator $\fopstat_t$ to map $\Mp(\domstat)$ into a Hilbert space $H$ (typically finite-dimensional),
and in case of noisy data we employ the quadratic penalty term
\begin{equation*}
\frac1{2\alpha}\sum_{t\in\alltimes}\|\fopstat_tu_t-\data_t\|_H^2
\end{equation*}
with weight parameter $\alpha\geq0$ (our analysis will not depend on that particular choice).
For notational convenience we explicitly allow the value $\alpha=0$, in which case the fidelity term shall be interpreted as the constraint $\fopstat_tu_t=\data_t$ for all $t\in\alltimes$.

\removed{
Our first result relates different variants of the reduced problem \eqref{eqn:lowDimensionalModel}
in which the snapshots and position-velocity projections are taken from different spaces.
In detail, snapshots could be interpreted as elements of the product space $\Mp(\domstat)^\domt$ (with the product topology)
or alternatively as Radon measures on the product manifold $\domstat\times\domt$.
Likewise, the space of position-velocity projections could be $\Mp(\projdomdyn)^\alldirs$ or $\Mp(\alldirs\times\projdomdyn)$.
(In fact one could also think of other topological spaces for snapshots and position-velocity projections
that encode some additional regularity inherited from the dimension reduction, but we found little benefit from that.)
This leads to the different model variants
\begin{align}
\min_{ \substack{ \liftVar \in \Mp( \projdomdyn)^{\alldirs}  \\ u\in \Mp(\domstat)^{\domt} }}\!
\| \liftVar_{\hat\theta} \|_\M
\!+\!\frac1{2\alpha}\!\sum_{t\in\alltimes}\|\fopstat_tu_t\!-\!\data_t\|_H^2
&\text{ s.\,t.\ }
\mv_t^1 \liftVar_\theta = \Rs_\theta u_t \text{ for all }t\!\in\!\domt, \, \theta\!\in\!\alldirs,
\label{eqn:productTopologyModel}\\
\min_{ \substack{ \liftVar \in \Mp(\alldirs \times \projdomdyn)  \\ u \in \Mp(\domstat)^{\domt} }}\!
\| \liftVar \|_\M
\!+\!\frac1{2\alpha}\!\sum_{t\in\alltimes}\|\fopstat_tu_t\!-\!\data_t\|_H^2
&\text{ s.\,t.\ }
\mv_t^1 \liftVar_\theta = \Rs_\theta u_t \text{ for all }t\!\in\!\domt,\, \mdirs\text{-a.a.\ }\theta\!\in\!\alldirs,
\label{eqn:mixedModel}\\
\min_{ \substack{ \liftVar \in \Mp(\alldirs \times \projdomdyn)  \\ u \in \Mp(\domstat \times \domt) }}\!
\| \liftVar \|_\M
\!+\!\frac1{2\alpha}\!\sum_{t\in\alltimes}\|\fopstat_tu_t\!-\!\data_t\|_H^2
&\text{ s.\,t.\ }
\mv_t^1 \liftVar_\theta = \Rs_\theta u_t \text{ for $\mtime$-a.a.\,}t\!\in\!\domt,\, \mdirs\text{-a.a.\,}\theta\!\in\!\alldirs.
\label{eqn:measureTopologyModel}
\end{align}
where for a measure $\liftVar\in\Mp(\alldirs\times\projdomdyn)$ the slice $\liftVar_\theta$ is defined by the relation $\liftVar=\mdirs(\theta)\otimes\liftVar_\theta$
and for a measure $u\in\Mp(\domstat\times\domt)$ the slice $u_t$ is defined by $u=u_t\otimes\mtime(t)$
(the condition $\mv_t^1\liftVar_\theta=\Rs_\theta u_t$ thus implies that the above disintegrations have to exist;
in \cref{sec:alternativeFormulations} we will use a slightly different notation that does not use measure slices but instead adapts the involved linear operators).
Note that one could also consider a fourth combination, however, it has no advantages over the above stated ones, so we do not treat it explicitly in this work.
While product spaces are most convenient if the reconstruction error in single or all snapshots or position-velocity projections is analysed,
classical Radon measures allow a simpler and standard variational well-posedness analysis,
so all these models have their advantages and disadvantages.
If $\domt$ and $\alldirs$ are such that the Radon transform and the move operator are injective (for instance if they have nonempty interior \cite[Thm.~II.3.4]{Natterer86_mh}), however,
we prove that the models are even equivalent.
Below, we say that $\liftVar\in\Mp(\alldirs\times\projdomdyn)$ and $\tilde\liftVar\in\Mp(\projdomdyn)^\alldirs$ coincide
if the slices $\liftVar_\theta$ are uniquely defined and coincide with $\tilde\liftVar_\theta$ for all $\theta\in\alldirs$.
Similarly, $u\in\Mp(\domstat\times\domt)$ and $\tilde u\in\Mp(\domstat)^\domt$ coincide
if the slices $u_t$ are uniquely defined and coincide with $\tilde u_t$ for all $t\in\domt$.

\begin{thm}[Model well-posedness and equivalence]\label{thm:modelEquivalences}
Let $\data_t\in H$ for all $t\in\alltimes$ and $\alpha\geq0$,
where for $\alpha=0$ we assume that there exists a configuration compatible with the data, that is, some $\lambda\in\Mp(\R^d\times\R^d)$
with $\mv_t^d\lambda\in\Mp(\domstat)$ for all $t\in\domt$ and $\fopstat_t\mv_t^d\lambda=\data_t$ for all $t\in\alltimes$.
\begin{enumerate}
\item
If the operator $(\Rs_\theta)_{\theta\in\supp\mdirs}$ is injective, then for any admissible pair $(\liftVar,u)$ in \eqref{eqn:measureTopologyModel}
the slice $u_t$ is uniquely defined for all $t\in\domt$
so that the fidelity term in \eqref{eqn:measureTopologyModel} is well-defined.
\item
The minimization problems \eqref{eqn:productTopologyModel}, \eqref{eqn:mixedModel}, and \eqref{eqn:measureTopologyModel} all admit a solution,
the latter under the condition that $(\Rs_\theta)_{\theta\in\supp\mdirs}$ is injective.
\item
If the operator $(\mv_t^1)_{t\in\supp\mtime}$ is injective, then the minimization problems \eqref{eqn:productTopologyModel} and \eqref{eqn:mixedModel} are equivalent
(the optima $(\liftVar,u)$ coincide).
If the operator $(\Rs_\theta)_{\theta\in\supp\mdirs}$ is injective, then the minimization problems \eqref{eqn:mixedModel} and \eqref{eqn:measureTopologyModel} are equivalent
(the optima $(\liftVar,u)$ coincide).
\end{enumerate}
\end{thm}

Our further main results concern the reconstruction properties of the above models.
} %
\added{\lookUp{\ref{item3}}
The resulting dimension-reduced problem, including the case of noisy data, is given as
\begin{equation}
\min_{ \substack{ \liftVar \in \Mp(\alldirs \times \projdomdyn)  \\ u \in \Mp(\domstat)^{\domt} }}\!
\| \liftVar \|_\M
\!+\!\frac1{2\alpha}\!\sum_{t\in\alltimes}\|\fopstat_tu_t\!-\!\data_t\|_H^2
\text{ s.\,t.\ }
\mv_t^1 \liftVar_\theta = \Rs_\theta u_t \text{ for all }t\!\in\!\domt,\, \mdirs\text{-a.a.\ }\theta\!\in\!\alldirs.
\label{eqn:mixedModel}
\end{equation}
After establishing well-posedness of the above model, our first, main results concern its reconstruction properties for $\alpha=0$.}
Since the work of de Castro and Gamboa \cite{Castro2012} and Cand\`es and Fernandez-Granda \cite{Candes2014} it is known
that a measure of the form $u_t=\sum_{i=1}^Nm_i\delta_{x_i}$
(i.\,e.\ a finite collection of particles with masses $m_i\in\R$ and locations $x_i\in\R^d$)
can be exactly reconstructed via \eqref{eqn:staticModel} even from a finite-dimensional observation $\fopstat_tu_t$
as long as the particle configuration satisfies certain spatial regularity conditions depending on the forward operator $\fopstat_t$
(for instance, if $\fopstat_t\added{=\ftrunc}$ yields a truncated Fourier series, the particles have to satisfy a minimum distance condition).
Additionally, in case of noisy observations the reconstruction error can be quantified in terms of the noise strength \cite{Candes2013}.
The interesting insight of Alberti et al.\ in \cite{AlbertiAmmariRomeroWintz2019} essentially was
that these reconstruction properties for static particle configurations $u_t$ can be transferred
to the reconstruction of dynamic particle configurations $\lambda$ via \eqref{eqn:highDimensionalModel}.
In detail, if $\lambda=\sum_{i=1}^Nm_i\delta_{(x_i,v_i)}$ represents the ground truth dynamic particle configuration
and if the snapshot $u_t=\mv_t^d\lambda$ can be exactly reconstructed from the noise-free observation $\fopstat_tu_t$
for times $t\in\goodtimes$ from a small set $\goodtimes\subset\alltimes$ of observation times,
then model \eqref{eqn:highDimensionalModel} is guaranteed to reconstruct the correct $\lambda$ exactly as long as $\lambda$ satisfies a certain dynamic regularity.
This means that not only the particle positions $x_i$ and masses $m_i$ are correctly recovered, but in addition also their velocities $v_i$.
The required dynamic regularity is that the snapshots $u_t$ for $t\in\goodtimes$ neither contain coincidences nor ghost particles.

\begin{defn}[Coincidence and ghost particle]\label{def:coincidenceGhostParticle}
Let $\lambda=\sum_{i=1}^Nm_i\delta_{(x_i,v_i)}\in\M(\R^d\times\R^d)$ be a particle configuration
with $m_i\in\R$, $x_i,v_i\in\R^d$ for $i=1,\ldots,N$, and let $\goodtimes\subset\R$ be a set of times.
\begin{enumerate}
\item
Configuration $\lambda$ has a \emph{coincidence} at some time $t\in\goodtimes$ if the locations of two distinct particles $i\neq j$ coincide at time $t$, that is, $x_i+tv_i=x_j+tv_j$.
\item
Configuration $\lambda$ admits a \emph{ghost particle} with respect to $\goodtimes$
if there exists $(x,v)\in\R^d\times\R^d$ outside the support of $\lambda$
such that at each time $t\in\goodtimes$ the ghost particle location coincides with a particle location, that is, $x+tv=x_i+tv_i$ for some $i\in\{1,\ldots,N\}$ depending on $t\in\goodtimes$.
\end{enumerate}
\end{defn}

Our contribution is to show that this behaviour persists after the dimension reduction.
In fact, we prove that in absence of coincidences or ghost particles
\removed{models \eqref{eqn:productTopologyModel}-\eqref{eqn:mixedModel}}\added{the model \eqref{eqn:mixedModel}} with $\alpha=0$ yields exact reconstruction of the ground truth $\liftVar$ and $u$ in case of noise-free data
and that in case of noisy observations the reconstruction error can be quantified in terms of the noise strength. \removed{(we consider model \eqref{eqn:measureTopologyModel} no further since, if well-posed, it is equivalent to \eqref{eqn:mixedModel} anyway).}
To state the results, let the ground truth configuration and corresponding noise-free observations be labelled with $\dagger$,
\begin{equation*}\phantomsection\label{eqn:groundTruth}
\lambda^\dagger=\sum_{i=1}^Nm_i\delta_{(x_i,v_i)}\in\Mp(\R^d\times\R^d),\quad
u_t^\dagger=\mv_t^d\lambda^\dagger,\quad
\liftVar_\theta^\dagger=\Rj_\theta\lambda^\dagger,\quad
\data_t^\dagger=\fopstat_tu_t^\dagger.
\end{equation*}
\removed{Further, we will call a snapshot $u_t^\dagger$ \emph{reconstructible} from observations via $\fopstat_t$ if it satisfies a particular type of source condition
which is sufficient to reconstruct $u_t^\dagger$ exactly from the observation $\fopstat_tu_t^\dagger$ via \eqref{eqn:staticModel}.
This source condition involves the existence of two dual variables $v^\dagger,v$ and three constants $\kappa,\mu,R>0$,
and it will be stated in detail in \cref{def:reconstructible}.
The main result of Cand\`es and Fernandez-Granda in \cite{Candes2013,Candes2014} essentially concerns this source condition:
if the forward operator $\fopstat_t$ represents a truncated Fourier series,
then $u_t^\dagger$ is reconstructible as soon as the minimum distance between its particles is (up to an explicit constant) no smaller than the inverse truncation frequency.
}%
\added{\label{txt:itemeReconstructible}\lookUp{\ref{iteme}}%
Further, we will say that the snapshot $u_t^\dagger$ can be \emph{reconstructed exactly} from the observation $\data_t^\dagger$ with the observation operator $\fopstat_t$
if it is the unique solution to the static problem at the time instant $t\in\alltimes$,
\begin{equation*}\tag*{$\mathrm P_{\text{stat}}^t(\data^\dagger_t)$}\label{eqn:erstat_intro}
    \min_{u\in\Mp(\domstat)}\mnorm{u}
    \quad\text{such that } \fopstat_tu=\data^\dagger_t.
\end{equation*}
}%
To obtain \added{dynamic} exact reconstruction or a bound on the reconstruction error, the following assumption will be relevant.

\begin{assn}[Ground truth regularity]\label{ass:regularity}
Assume that there is $\emptyset \neq \goodtimes\subset\alltimes$ with the following properties:
\begin{enumerate}
\item\label{enm:reconstructibility}
\removed{For all $t\in\goodtimes$, snapshot $u_t^\dagger$ is reconstructible from observations with $\fopstat_t$.}%
\added{\label{txt:itemeReconstructible}\lookUp{\ref{iteme}}%
For all $t\in\goodtimes$, snapshot $u_t^\dagger$ can be reconstructed exactly from $\data_t^\dagger$ with $\fopstat_t$.
}%
\item\label{enm:coincidences}
Configuration $\lambda^\dagger$ has no coincidences in $\goodtimes$.
\item\label{enm:ghosts}
Configuration $\lambda^\dagger$ admits no ghost particle with respect to $\goodtimes$.
\end{enumerate}
\end{assn}

{\Cref{ass:regularity} is the same as in \cite{AlbertiAmmariRomeroWintz2019} for model \eqref{eqn:highDimensionalModel}\added{\lookUp{\ref{itema}}, see \cref{subsec:comparison_sota} below for a detailed comparison.}

Depending on the choice of the measure $\mdirs$, we obtain exact reconstruction results under this assumption or an even weaker version of it. 

The first result in that direction relies on \cref{ass:regularity} and is applicable in particular the case that $\mdirs$ is the (rescaled) $(d-1)$-dimensional Hausdorff measure restricted to a relatively open subset of $\sphere^{d-1}$.
\begin{thm}[Exact reconstruction for noise-free data -- Hausdorff measure]\label{thm:exactReconstruction_hausdorff}
Take $\mdirs$ to have a nonzero lower semi-continuous density with respect to the $(d-1)$-dimensional Hausdorff measure on $\sphere^{d-1}$.

If \cref{ass:regularity} holds, the solution $(u,\liftVar)$ of  \eqref{eqn:mixedModel} \removed{or \eqref{eqn:productTopologyModel}}for $\data=\data^\dagger$ and $\alpha=0$ satisfies 
\begin{enumerate}
\item
$u_t=u_t^\dagger$ for all $t\in\domt$, and
\item
$\liftVar_\theta=\liftVar_\theta^\dagger$ for $\mdirs$-almost all $\theta\in\alldirs$.
\end{enumerate}
\end{thm}
In a second branch of results, we consider the case that $\mdirs$ is a counting measure. Here we obtain exact reconstruction by only requiring \added{that} a certain number of snapshots $u_t^\dagger$ \added{can be exactly reconstructed} from observations with $\fopstat_t$ (without any assumptions on ghost particles or coincidences); see \cref{sec:numDirections}.

While those results are mostly deterministic, one particular case that we highlight here shows that \emph{most} point configuration can be reconstructed already if only three snapshots $u_t^\dagger$ \added{can be reconstructed exactly} from observations with $\fopstat_t$.
\begin{thm}[Exact reconstruction for noise-free data -- Counting measure]\label{thm:exactReconstruction_counting}
Let the location vector of $\lambda^\dagger$ in $(\R^d\times\R^d)^N$ be distributed
according to a probability distribution which is absolutely continuous with respect to the $2Nd$-dimensional Lebesgue measure, and choose $\mdirs$ to be the counting measure on $d+1$ directions in $\alldirs$ with any $d$ of them being linearly independent.

If there are three times $t\in \goodtimes$ at which snapshot $u_t^\dagger$ \added{can be reconstructed exactly} from observations with $\fopstat_t$, then almost surely the solution $(u,\liftVar)$ of \eqref{eqn:mixedModel} \removed{or \eqref{eqn:productTopologyModel} }for $\data=\data^\dagger$ and $\alpha=0$ satisfies 
\begin{enumerate}
\item
$u_t=u_t^\dagger$ for all $t\in\domt$, and
\item
$\liftVar_\theta=\liftVar_\theta^\dagger$ for $\mdirs$-almost all $\theta\in\alldirs$.
\end{enumerate}
\end{thm}
Our final main result concerns convergence rates in case of noisy data.
The noise strength $\delta>0$ of noisy observations $\data^\delta$ is for simplicity defined as the size of the data term at the ground truth,
\begin{equation*}\phantomsection%
\delta=\tfrac12\sum_{t\in\alltimes}\|\data_t^\delta-\data_t^\dagger\|_H^2.
\end{equation*}
The reconstruction error will be estimated using the unbalanced optimal transport cost
\begin{equation*}
\unbalancedWasserstein[2]{R}(\nu_1,\nu_2)=\inf\left\{W_2^2(\nu,\nu_2)+\tfrac12R^2\|\nu_1-\nu\|_{\M}\,\middle|\,\nu\in\Mp(\R^n)\right\},
\end{equation*}
where \added{\label{txt:itemd2}$R>0$ is a parameter
indicating how far mass may be transported (any mass requiring a farther transport is simply removed, paying $R^2$ times the removed amount of mass) and
\lookUp{\ref{itemd},\ref{iteme}}
\begin{multline*}
W_2(\nu_1,\nu_2)=\inf\left\{\left(\int_{\R^n\times\R^n}\dist(x,y)^2\,\d\pi(x,y)\right)^{1/2}\,\middle|\,\pi\in\Mp(\R^n\times\R^n),\right.\\
\left.\vphantom{\left(\int_{\R^n}\right)^{1/p}}\pf{[(x,y)\mapsto x]}\pi=\nu_1,\pf{[(x,y)\mapsto y]}\pi=\nu_2\right\}
\end{multline*}%
}%
denotes the Wasserstein-2 distance between $\nu_1,\nu_2\in\Mp(\R^n)$.
\added{\label{txt:itemeReconstructible2}\lookUp{\ref{itemd},\ref{iteme}}%
The error estimates naturally require a quantitative version of \cref{ass:regularity}\eqref{enm:reconstructibility} that additionally encodes the stability of the reconstruction.
This is done via a particular type of source condition, which essentially states that dual variables must exist
which take any prescribed combination of the values $\{-1,1\}$ on the support of the snapshot $u_t^\dagger$ and behave like paraboloids in the vicinity.
Denoting the subdifferential of the total variation norm by $\partial\|\cdot\|_\M$ and the (pre-)dual observation operator by $\fopstat_t^*:H\rightarrow C(\domstat)$,
the source condition reads as follows.

\begin{defn}[Stably reconstructible measure]\label{def:reconstructible}
We call a measure $\nu^\dagger=\sum_{i=1}^Nm_i\delta_{x_i}\in\Mp(\domstat)$ \emph{\added{stably} reconstructible} (from measurements via $\fopstat_t$),
if one can find $\kappa,\mu,R>0$ %
such that for any $N$-tuple $s\in\{-1,1\}^N$ and associated measure $\nu^s=\sum_{i=1}^Ns_i\delta_{x_i}\in\M(\domstat)$ there exist dual variables $v^\dagger,v^s$ with
$-\fopstat_t^*v^\dagger\in\partial\|\cdot\|_{\M}(\nu^\dagger)$, $-\fopstat_t^*v^s\in\partial\|\cdot\|_{\M}(\nu^s)$ and
\begin{align*}
-\fopstat_t^*v^\dagger(x)&\leq1-\kappa\min\{R,\dist(x,\{x_1,\ldots,x_N\})\}^2
&&\text{for all }x\in\domstat,\\
\fopstat_t^*v^s(x_i)\fopstat_t^*v^s(x)&\geq1-\mu\dist(x,x_i)^2
&&\text{for all }x\in B_R(\{x_i\}),\,i=1,\ldots,N.
\end{align*}
\end{defn}

Stable reconstructibility turns out to imply exact reconstructibility.
The main result of Cand\`es and Fernandez-Granda in \cite{Candes2013,Candes2014} essentially concerns this source condition:
if the forward operator $\fopstat_t\added{=\ftrunc}$ represents a truncated Fourier series,
then $u_t^\dagger$ is stably reconstructible as soon as the minimum distance between its particles is (up to an explicit constant) no smaller than the inverse truncation frequency.
}%
We prove the following reconstruction error estimates for noisy observations.

\begin{thm}[Reconstruction error estimates]\label{thm:reconstructionErrorEstimate}
Let $(u,\liftVar)$ be the solution of \eqref{eqn:mixedModel} \removed{or \eqref{eqn:productTopologyModel} }for $\data=\data^\delta$ and $\alpha=\sqrt\delta$, and let \added{$\emptyset\neq\goodtimes\subset\alltimes$ such that $u_t^\dagger$ is stably reconstructible for all $t\in\goodtimes$}. Further,
\begin{enumerate}
\item
let $\overline\alldirs\subset\alldirs$ compact such that $\liftVar^\dagger_\theta$ contains no coincidence or ghost particle with respect to $\goodtimes$ for any $\theta\in\overline\alldirs$, and
\item
let $\overline\domt\subset\domt$ compact such that for all $t\in\overline\domt$
all particles have positive mutual distance in the projections $(\Rs_\theta u^\dagger_t)_{\theta\in\overline\alldirs}$
and the support $\supp u^\dagger_t$ can be uniquely reconstructed from $(\supp(\Rs_\theta u^\dagger_t))_{\theta\in\overline\alldirs}$.
\end{enumerate}
Then there exist constants $R,C>0$ (depending on $\lambda^\dagger$, $\goodtimes$, $\overline\alldirs$ and $\overline\domt$) such that
\begin{enumerate}
\item\label{enm:snapshotEstimateIntro}
we have $\unbalancedWasserstein[2]{R}(u_t,u_t^\dagger)\leq C\sqrt\delta$ for all $t\in\overline\domt$,
\item\label{enm:liftVarEstimateIntro}
we have $\unbalancedWasserstein[2]{R}(\liftVar_\theta,\liftVar_\theta^\dagger)\leq C\sqrt\delta$ for $\mdirs$-almost all $\theta\in\overline\alldirs$ \removed{ all $\theta\in\overline\alldirs$.}
\end{enumerate}
\end{thm}

For any given finite particle configuration $\lambda^\dagger$ satisfying \cref{ass:regularity}
the set of $\theta\in\sphere^{d-1}$ such that $\liftVar^\dagger_\theta$ contains a coincidence or ghost particle with respect to $\goodtimes$
has vanishing $(d-1)$-dimensional Hausdorff measure $\hd^{d-1}$,
thus only a $\hd^{d-1}$-nullset of directions is excluded in \cref{thm:reconstructionErrorEstimate}.
Similarly, already if $\overline\alldirs$ contains just $d+1$ generic directions,
the set of times $t\in\R$ violating the conditions on $\overline\domt$ is also a Lebesgue-nullset
so that again just a nullset of times is excluded in \cref{thm:reconstructionErrorEstimate}.

\subsection{Comparison to the state of the art} \label{subsec:comparison_sota}

\added{\label{txt:item0}\lookUp{\ref{item0}}As mentioned before, in contrast to \cite{AlbertiAmmariRomeroWintz2019} we require nonnegativity of the measures to be reconstructed,
which is the cost to pay for achieving a dimension reduction (the technical reason will be pinpointed later in \cref{rem:signedMeasures}).
From an application viewpoint, the restriction to nonnegative measures is still of significant relevance,
for instance, if the measures describe the mass distribution of fluorescent molecules in fluorescence microscopy
or if, as in \cite{AlbertiAmmariRomeroWintz2019}, they describe the distribution of ultrasonic reflectors in ultrafast ultrasonography.
This restriction is also considered elsewhere in the literature,
for instance, exact recovery and error estimates for the reconstruction of nonnegative measures are analysed in \cite{Castro2012,DenoyelleDuvalPeyre2017}.
In general it seems easier to satisfy source conditions for nonnegative measures
so that \cref{ass:regularity}\eqref{enm:reconstructibility} and stable reconstructibility in the sense of \cref{def:reconstructible} are more easily satisfied,
see \cite{DenoyelleDuvalPeyre2017}.}

\added{\lookUp{\ref{itema}} Except for positivity, \Cref{ass:regularity} is essentially the same as what is required in \cite{AlbertiAmmariRomeroWintz2019} for model \eqref{eqn:highDimensionalModel}: The main result of \cite{AlbertiAmmariRomeroWintz2019} also needs to exclude the existence of ghost particles, and in addition it requires the existence of static dual certificates, which implies properties \ref{enm:reconstructibility} and \ref{enm:coincidences} of \Cref{ass:regularity}.}

\added{\lookUp{\ref{itema}}Exact reconstruction is shown in both \cite{AlbertiAmmariRomeroWintz2019} and our \cref{thm:exactReconstruction_hausdorff,thm:exactReconstruction_counting}.
In contrast, our convergence result in \cref{thm:reconstructionErrorEstimate} extends the work of \cite{AlbertiAmmariRomeroWintz2019} in several ways: Besides being formulated for any dimension $d \in \N$ (instead of $d=1$), it considers the infinite-dimensional setting and an unmodified version of \eqref{eqn:mixedModel}. In comparison, \cite{AlbertiAmmariRomeroWintz2019} shows convergence only in a discretized setting (with the involved constants being inversely proportional to the resolution of the discretization so that they blow up in the continuum limit) and for a specific type of noise model. 
}

\added{\label{txt:item4}\lookUp{\ref{item4}}%
It is also worth noting that, while truncated Fourier measurements $\ftrunc$ are most commonly used for the operator $\fopstat$ as in \cite{Candes2013,Candes2014,DenoyelleDuvalPeyre2017,AlbertiAmmariRomeroWintz2019}, they represent just one possible example (admittedly, an important one for which dual variables for source conditions can most readily be constructed).
There exist alternative observation operators allowing \cref{ass:regularity}\eqref{enm:reconstructibility} and stable reconstructibility to be satisfied so that our theory applies:
De Castro and Gamboa list several examples of so-called generalized moments with corresponding properties for which exact reconstruction holds \cite{Castro2012}.
In particular, our theory does not rely on the specific dual variables constructed in \cite{AlbertiAmmariRomeroWintz2019} (of which we perform a dimension reduction),
but it only relies on their properties, which also can be attained for observation operators $\fopstat$ other than truncated Fourier measurements.}

\added{\subsection{Paper outline}}

\removed{The models \eqref{eqn:productTopologyModel}-\eqref{eqn:measureTopologyModel} and corresponding operators will be introduced in detail in \cref{sec:models},
where also \cref{thm:modelEquivalences} will be proven.}
\added{\lookUp{\ref{item3}} The model and corresponding operators will be introduced in detail in \cref{sec:models}. Further, in \ref{sec:alternativeFormulations}, different alternative formulations will be introduced and their equivalence will be proven.}
\Cref{sec:exactRecovery} contains the proofs of \cref{thm:exactReconstruction_hausdorff} and \cref{thm:exactReconstruction_counting} and some variants as well as a discussion of how big a direction set $\alldirs$ is useful,
which may be interesting from the viewpoint of numerical effort.
The error estimation results from \cref{thm:reconstructionErrorEstimate} and its variants for finite $\alldirs$ are derived in \cref{sec:noisyReconstruction}
based on a general strategy of how to obtain unbalanced optimal transport error estimates for inverse problems of measures via convex duality techniques.
The model is implemented and illustrated numerically in \cref{sec:numerics}.

\subsection{Notation and preliminaries}
In this paragraph we recall the (standard) notion of Radon measures and collect some properties of measures
before finally providing a list of symbols used throughout the article.
While different results in this section on measures on a space $X$ require different assumptions on $X$, we note that in our application $X$ will always be either a closed subset of $\R^n$ or of the sphere $\sphere^{d-1}$ with the induced norm or metric, respectively. Consequently, $X$ will always be a separable, complete, $\sigma$-compact metric space such that all of the below results are applicable to $X$.

\begin{defn}[Borel and Radon measure, total variation]\label{def:RadonMeasure}
Let $X $ be a topological space. A \emph{real Borel measure} is a countably-additive set function $\nu: \Bc(X) \rightarrow \R$, with $\Bc(X)$ being the Borel $\sigma$-algebra on $X$. The \emph{total variation} of a real Borel measure $\nu: \Bc(X) \rightarrow \R$ is the nonnegative Borel measure defined as
\[ |\nu|(E) = \sup \left\{ \sum_{i} |\nu(E_i)| \,\middle|\, E_1,E_2,\ldots\in\Bc(X) \text{ are pairwise disjoint},\, E = \bigcup _i E_i \right\}. \]
A \emph{finite Radon measure} is a real Borel measure whose total variation is a regular measure. We denote by $\M(X)$ the space of finite Radon measures on $X$ and by $\Mp(X) \subset \M(X)$ the set of nonnegative finite Radon measures.
\end{defn}

Note that $\M(X)$ equipped with the norm $\|\nu\|_\M = |\nu|(X)$ is a normed vector space.\phantomsection\label{def:norm}

\begin{thm}[Duality, {\cite[Thm.~6.19]{rudin2006real_complex_analysis_mh}}]\label{thm:measure_predual}
Let $X$ be a locally compact Hausdorff space. Then, the space $\M(X)$ can be identified with the dual space of $C_0(X)$, the latter being the completion of the space of compactly supported continuous functions with respect to the supremum norm $\|\cdot\|_\infty$. The duality pairing is given as
\[ \langle \nu,\phi \rangle = \int_X \phi(x) \wrt \nu(x), \]
and we have
\[ \|\nu\|_\M = \sup _{\phi \in C_0(X), \|\phi\|_\infty \leq 1 } \int_\Omega \phi(x) \wrt \nu(x) .\]
In particular, $\M(X)$ is a Banach space.
\end{thm}

In accordance with the duality $\M(X) = C_0(X)^*$ we will denote the weak-$\ast$ convergence of a sequence $\nu_1,\nu_2,\ldots\in\M(X)$ to some $\nu\in\M(X)$ by $\nu_n\weakstarto\nu$ as $n\to\infty$.
In addition, we will also need the notion of \emph{narrow convergence}: We say that $\nu_1,\nu_2,\ldots$ converges to $\nu \in \M(X)$ with respect to the narrow topology and write $\nu_n \narrowto \nu$ as $n \rightarrow \infty$, if
\[ \lim_{n \rightarrow \infty} \int_X \varphi (x) \wrt \nu_n (x) = \int_X \varphi(x) \wrt \nu(x) \]
for all $\varphi \in C^b (X)$, the latter denoting the Banach space of bounded, continuous functions on $X$ equipped with the supremum norm $\|\cdot \|_\infty$. Obviously, narrow convergence is a stronger property than weak-$\ast$ convergence. We will, however, also require equivalence of the two notions for certain compactly supported sequences of measures. To this end, we first define the support of Radon measures and provide some basic properties.
\begin{defn}[Support of a measure] \label{defn:support_measure}
The \emph{support} of a real Borel measure $\nu:\Bc(X) \rightarrow \R$ on a topological space $X$ is defined as
\[ \supp(\nu) = \left\{ x \in X \,\middle|\,|\nu|(B)>0  \text{ for each neighbourhood } B \text{ of } x \right\}. \]
\end{defn}
\begin{prop}[Support properties] \label{prop:support_radon_measure}
Let $X$ be a locally compact Hausdorff space and $\nu\in\M(X)$. Then the following holds.
\begin{itemize}
\item $\supp(\nu)$ is closed, and $\nu$ is concentrated on $\supp(\nu)$, that is, $\nu(E) = \nu(E \cap \supp(\nu))$ for all $E \in \Bc(X)$.
\item If $\nu_n\weakstarto\nu$ in $\M(X)$ as $n\to\infty$ and if $\supp(\nu_n) \subset S$ for all $n$ with $S \subset X$ closed, then $\supp(\nu) \subset S$.
\end{itemize}
\end{prop}
\begin{proof}
It is easy to see that $\supp(\nu)$ is the complement of the union of all open sets $V$ with $|\nu|(V) = 0$, thus it is closed. Also, $|\nu|(V) = 0$ for each open $V \subset X\setminus\supp(\nu)$, which implies $|\nu|(V) = |\nu|(V\setminus\supp(\nu))+|\nu|(V \cap \supp(\nu)) = |\nu|(V \cap \supp(\nu))$ for any $V $ open. Using the latter, it is immediate to show that for any $E \in \Bc(X)$ with $E \subset X\setminus\supp(\nu)$,
\[ 0
\leq\inf \{ |\nu|(V) \st E \subset V, \, V \text{ open }  \}
\leq \inf \{ |\nu|(V) \st E \subset V, \, V \text{ open},\, V \subset X\setminus\supp(\nu)   \}
=0,
\]  
which, by outer regularity, implies that $\nu(E)=|\nu|(E) = 0$.  Hence, $\nu(E) = \nu(E \cap \supp(\nu))$ also for each $E \in \Bc(X)$ as claimed.

Now assume $\nu_n\weakstarto\nu$ in $\M(X)$ with $\supp(\nu_n) \subset S$ as above. Then,  for $x \notin S$, $B = X\setminus S$ is an open neighbourhood of $x$, and for any $\varphi \in C_0(B)$ it follows that 
\[ \int_B \varphi \wrt \nu =  \lim_{n \rightarrow \infty} \int_B \varphi \wrt \nu_n = 0.\]
By \cref{thm:measure_predual} (noting that $B$ is again a locally compact Hausdorff space as being an open subset of $X$) this implies 
\[ 0 = \| \nu\|_{\M(B)} = |\nu|(B) \] such that $x \notin \supp(\nu)$.
\end{proof}

Now a basic equivalence result on weak-$\ast$ convergence and narrow convergence holds as follows.
\begin{lem}[Equivalence of weak-$\ast$ and narrow convergence]\label{lem:weak*_and_support_implies_narrow} Let $X$ be a complete separable metric space and $\nu_1,\nu_2,\ldots$ a sequence in $\M(X)$ weakly-$\ast$ converging to $\nu \in \M(X)$ such that $\supp(\nu_n) \subset S$ with $S \subset X $ compact. Then $\supp(\nu) \subset S$, and $\nu_n$ converges to $\nu$ also narrowly.
\begin{proof}
The fact that $\supp(\nu) \subset S$ was shown in Proposition \ref{prop:support_radon_measure}. Convergence with respect to the narrow topology is a direct consequence of Prohorovs theorem (see \cite[Thm.~8.6.2]{bogachev2007measure_vol2}) and the fact that the sequence $\nu_n$ is uniformly tight (see \cite[Def.~8.6.1]{bogachev2007measure_vol2}).
\end{proof}
\end{lem}

In case $X$ is a metric space and $\sigma$-compact (that is, the countable union of compact sets), any real Borel measure is regular and thus a finite Radon measure.

\begin{lem}[Borel and Radon measure] \label{lem:sigma_compact_borel_measure_regular}
Let $X$ be a $\sigma$-compact metric space. Then any real Borel measure on $\Omega$ is regular.
\begin{proof}Let $\nu$ be a real Borel measure on $X$. We need to show that the real Borel measure $|\nu|$ is regular.
This follows from \cite[Thm.~II.1.2]{parthasarathy2005probability_mh} together with the finiteness of $|\nu|$ and the fact that,
due to $\sigma$-compactness, the compact sets in the definition of inner regularity can be replaced by closed sets. Indeed, obviously the supremum gets larger if we replace compactness by closedness,
while the other inequality follows from the fact that for any sequence of compact sets $K_1,K_2,\ldots$ with $X = \bigcup _i K_i$ and any closed $A\in\Bc(\Omega)$ we have
$|\nu|(A) = \lim _{n\rightarrow \infty } |\nu| (A_n)$ with $A_n=A\cap\bigcup_{i=1}^n K_i$ compact.
\end{proof}
\end{lem}

\begin{defn}[Pushforward measure]\label{def:pushforward}
For $(X,\Sigma_X)$, $(Y,\Sigma_Y)$ measure spaces and $f:X \rightarrow Y$ a measurable function, we define the \emph{pushforward} of a measure $\nu:\Sigma_X \rightarrow \R\cup\{\infty\}$ as
\begin{equation*}
\pf{f}\nu: \Sigma_Y \to \R\cup\{\infty\},\quad
A  \mapsto  \nu(f^{-1}(A))\,.
\end{equation*}
\end{defn}

\begin{thm}[Basic properties of pushforward measures]\label{thm:pushforward_radon}
With the notation of \cref{def:pushforward}, $\pf{f}\nu$ is a measure on $(Y,\Sigma_Y)$, and a measurable function $g:Y \rightarrow \R$ is integrable with respect to $\pf{f}\nu$ if $g \circ f$ is integrable with respect to $\nu$. If $\nu$ is nonnegative, also the converse holds true. If both are integrable, it holds that
\begin{equation} \label{eq:pf_general_integral_equality}
 \int_Y g(y) \wrt \pf{f}\nu(y) = \int_X g(f(x)) \wrt \nu(x) .
 \end{equation}
Further, if $Y$ is a $\sigma$-compact metric space, then $\pf{f}\nu$ is a finite Radon measure on $\Sigma_Y$ if $\nu$ is a finite Radon measure on $\Sigma_X$.
\end{thm}
\begin{proof}
The fact that $\pf{f}\nu$ is a measure and the assertion on integrability are shown in \cite[Thm.~3.6.1 and below]{bogachev2007measure_vol1}. The fact that $\pf{f}\nu$ is a finite Radon measure is immediate from \cref{lem:sigma_compact_borel_measure_regular}, since $Y$ is a $\sigma$-compact metric space by assumption.
\end{proof}

\begin{prop}[The pushforward as operator]\label{prop:pushforward_properties} Let $X,Y$ be locally compact Hausdorff spaces, $X$ be further a complete, separable, $\sigma$-compact metric space and $f :X \rightarrow Y$ be continuous. Then
\begin{equation*}
\pf{f}:\M(X)  \rightarrow \M(Y), \;
\nu  \mapsto \pf{f}\nu
\end{equation*}
is a well-defined linear operator with $\|\pf{f}\| \leq 1$. Further, the following holds.
\begin{itemize}
\item If $\nu \geq 0$, then $\pf{f}\nu \geq 0$ and $\|\pf{f}\nu\|_{\M(Y)} = \|\nu\|_{\M(X)}$.
\item $\pf{f}$ is continuous with respect to narrow convergence of measures.
\item $ \supp(\pf{f}\nu) \subset \overline{f(\supp(\nu))}$.
\item If $\nu_1,\nu_2,\ldots\in\M(X)$ are such that $\nu_n \weakstarto \nu \in \M(X)$ and $\supp(\nu_n) \subset S$ with $S$ compact, then $\supp(\pf{f}\nu) \subset f(S)$ and $(\pf{f}\nu)_n \narrowto \pf{f}\nu$.
\end{itemize}
\begin{proof}
 Linearity and the boundedness of $\pf{f}$ follow from equality \eqref{eq:pf_general_integral_equality}, which holds for any $g \in C_0(Y)$. Likewise, continuity with respect to narrow convergence follows from \eqref{eq:pf_general_integral_equality} being true for any bounded continuous function $g$. Positivity and norm preservation in case of positive $\nu$ is immediate from the definition of the pushforward.

For showing the support inclusions, again we exploit that $\supp(\nu)$ is the complement of the union of all open sets $V$ with $|\nu|(V) = 0$. 
Take $x \in B:= Y\setminus\overline{f(\supp(\nu))}$. Then $B$ is open, and for any $E \subset B$ it follows that 
$\pf{f}\nu (E) = \nu (f^{-1}(E)) = 0$ since $f^{-1}(E) \subset X\setminus\supp(\nu)$ and $\nu$ is concentrated on $\supp(\nu)$. This implies that $|\pf{f}\nu|(B) = 0$ and hence $x \notin \supp(\pf{f}\nu)$, which shows the inclusion $\supp(\pf{f}\nu) \subset  \overline{f(\supp(\nu))}$.

The last statement is an immediate consequence of \cref{lem:weak*_and_support_implies_narrow} and continuity of the pushforward with respect to narrow convergence. 
\end{proof}
\end{prop}

\begin{defn}[Product measure]  \label{def:family_product_measure}
Let $ Y,Z$ be topological spaces, $\omega:\Bc(Z) \rightarrow \R$ be a real Borel measure and $(\nu_z)_{z \in Z}$ with $\nu_z:\Bc(Y) \rightarrow \R$ be a family of real Borel measures such that for any $A \in \Bc(Y)$,  $Z \ni z \mapsto \nu_z(A)$ is $\omega$-measurable. Then we define the \emph{product measure}
\[ \nu_z \prodm{z} \omega : \Bc(Y \times Z)  \rightarrow \R\]
on the generator $\{ A \times B \st A \in \Bc(Y),\, B \in \Bc(Z) \}$ of $\Bc(Y \times Z)$ via 
\[
(\nu_z \prodm z \omega)(A \times B) = \int_B \nu_z(A) \wrt \omega(z) .
\]
Similarly, we use the notation $ \omega \prodm z  \nu_z $ for a corresponding product measure on $\Bc(Z \times Y ) $.
\end{defn}
Note that, indeed, the product measure is a well-defined, $\sigma$-additive set function on $\Bc(Y \times Z)$ and, by  \cref{lem:sigma_compact_borel_measure_regular}, is a finite Radon measure in case $Y$ and $Z$ are $\sigma$-compact metric spaces. Also, note that, for the sake of brevity, we use the term ``product measure'' while the above measure is actually a product measure with transition kernel $k:Z \times \Bc(A)  \rightarrow \R$, $k(z,A):= \nu_z(A)$. It generalizes the standard product measure $ \nu \otimes \omega$ in the sense that the latter is obtained in the special case that $\nu_z = \nu$ independent of $z$.

Of particular interest in this work will be the case that $\nu_z = \pf{[f_z]}\nu$, with $(f_z)_{z \in Z}$ a given family of functions. In this case, the product measure $\pf{[f_z]}\nu \prodm z \omega$ can equivalently be written as $\pf{ F} (\nu \otimes  \omega)$ with $F:X \times Z \rightarrow \R$, $F(x,z):= (f_z(x),z)$. In case the mapping $(x,z) \mapsto f_z(x)$ is continuous, this allows us to directly transfer the properties of the pushforward as stated in \cref{prop:pushforward_properties} to the mapping $ (\nu,\omega) \mapsto \pf{[f_z]}\nu \prodm z \omega$. 
In the following proposition, we summarize these properties for the special case that $\omega \geq 0$ with $\omega(Z)=1$ is fixed, as this is the setting that is relevant for our analysis in the subsequent sections below.
\begin{prop}[Properties of the product measure] \label{prop:properties_product_measure}
Let $X,Y,Z$ be $\sigma$- and locally compact metric spaces with $Y,Z$ being complete and separable, $ \omega \in \Mp(Z)$ and $\nu \in \M(X)$ be finite Radon measures with $\omega(Z) = 1$, and $(f_z)_{z \in Z} $ be a family of functions $f_z:X \rightarrow Y$ such that $(x,z) \mapsto f_z(x)$ is continuous. 

Then the product measure $\pf{[f_z]}\nu \prodm z \omega: \Bc(Y \times Z) \rightarrow \R$ is well-defined, a finite Radon measure on $Y \times Z$, and the following holds.
\begin{itemize}
\item If $\nu \geq 0$, then $\pf{[f_z]}\nu \prodm z \omega \geq 0$ and $\|\pf{[f_z]}\nu \prodm z \omega\|_{\M(Y \times Z)} = \|\nu\|_{\M(X)}$.
\item $\nu \mapsto \pf{[f_z]}\nu \prodm z \omega$ is continuous with respect to narrow convergence of measures.
\item $\supp(\pf{[f_z]}\nu \prodm z \omega) \subset \overline{F( \supp(\nu) \times \supp(\omega))}$.
\item Assume that $Z$ is compact. Then, if $\nu_1,\nu_2,\ldots\in\M(X)$ are such that $\nu_n \weakstarto \nu \in \M(X)$ and $\supp(\nu_n) \subset S$ with $S$ compact, it follows that $\supp(\pf{[f_z]}\nu \prodm z \omega) \subset F(S \times Z)$ and $\pf{[f_z]}\nu_n \prodm z \omega(z) \narrowto \pf{[f_z]}\nu \prodm z \omega(z)$.
\end{itemize}
The same holds true for the product measure $\omega \prodm  z \pf{[f_z]}\nu : \Bc(Z \times Y) \rightarrow \R$.
\begin{proof} 
The first two items are an immediate consequence of \cref{prop:pushforward_properties}, and the third one follows from \cref{prop:pushforward_properties} by observing that $\supp ( \nu \otimes \omega) \subset \supp (\nu) \times \supp(\omega)$, which can be deduced directly from the definition of the support.
Given the third statement, the fourth is again a direct consequence of \cref{prop:pushforward_properties}.
\end{proof}
\end{prop}

For the reader's convenience we finally provide a list of frequently employed symbols.\\
\begin{longtable}{ll}
$B_R(\dsupp)$
& $R$-neighbourhood of a set $\dsupp$ (\cref{thm:unbalancedWasserstein})\\
$C(X),C_c(X)$
& continuous and compactly supported continuous functions on $X$\\
$C_0(X)$
& completion of $C_c(X)$ with respect to the supremum norm (\cref{thm:measure_predual})\\
$\|\cdot\|_\infty$
& supremum norm (\cref{thm:measure_predual})\\
$\BregmanDistance\energy w$
& Bregman distance with respect to $w\in\partial\energy$ (\cpageref{sec:BregmanEstimates})\\
$d$
& space dimension\\
$\delta_p$
& Dirac measure in $p$\\
$\delta$
& noise strength (\cpageref{eqn:noiseStrength})\\
$\data,\data^\delta,\data^\dagger$
& observations and observations with and without noise (\cpageref{eqn:groundTruth,eqn:noiseStrength})\\
$\ft=\hat\cdot,\ft^{-1}=\check\cdot$
& Fourier and inverse Fourier transform (\cref{def:FourierTransform})\\
\added{$\ftrunc$}
& \added{truncated Fourier series (\cref{ex:truncatedFourier})}\\
$G(\dsupp,\goodtimes)$
& ghost particles of $\dsupp$ with respect to $\goodtimes$ (\cref{def:coincidence_ghost_particle})\\
$\projdomdyn$
& domain in one-dimensional position-velocity space (\cpageref{sec:models})\\
$\liftVar,\liftVar^\dagger$
& position-velocity projections and their ground truth (\cpageref{eqn:snapshots,eqn:groundTruth})\\
$H$
& space of observations (\cref{defn:pointwise_measurement_operator})\\
$\hd^n$
& $n$-dimensional Hausdorff measure (\cpageref{txt:measures})\\
$\mtime,\mdirs$
& measures on time and direction domain (\cpageref{txt:measures})\\
$\lebesgue^n$
& $n$-dimensional Lebesgue measure (\cpageref{txt:measures})\\
$\domdyn$
& domain in position-velocity space (\cpageref{sec:models})\\
$\lambda,\lambda^\dagger$
& particle configuration and its ground truth (\cpageref{txt:highDimMeasure,eqn:groundTruth})\\
$\M(X),\Mp(X)$
& sets of finite and nonnegative finite Radon measures on $X$ (\cref{def:RadonMeasure})\\
$\|\cdot\|_\M$
& norm on the space of finite Radon measures (\cpageref{def:norm})\\
$\mv^n,\amv[n],\amvp[n],\amvm[n]$
& standard and anglewise (with three variants) move operator (\cref{defn:move_operators},\\ & \cpageref{thm:Candes2013,eqn:mixedTopologyOperators})\\
$\domstat$
& spatial domain (\cpageref{sec:models})\\
$\fopstat,\fopdyn,\fopdynp$
& forward or observation operator  (\cref{defn:pointwise_measurement_operator,defn:radon_space_measurement_operator}, \cpageref{thm:Candes2013})\\
$\Rs,\Rf,\Rfp, \Rfm,\Rj$
& standard, timewise (with three variants) and joint Radon transform \\ & (\cref{def:radon_transform_measure}, \cpageref{thm:Candes2013,eqn:mixedTopologyOperators})\\
$\dsupp$
& particle configuration without particle masses (\cpageref{sec:exactRecovery})\\
$\sphere^{d-1}$
& $(d-1)$-dimensional sphere\\
$\supp$
& support of a finite Radon measure (\cpageref{prop:support_radon_measure})\\
$\domt$
& time domain (\cpageref{sec:models})\\
$\alltimes,\goodtimes$
& observation times (\cpageref{sec:models},   \cref{ass:regularity,thm:reconstructionErrorEstimate})\\
$\alldirs,\gooddirs$
& set of directions and specifically chosen directions (\cpageref{sec:models} and \cref{thm:estimateBadU})\\
$u,u^\dagger$
& snapshots and their ground truth (\cpageref{eqn:snapshots,eqn:groundTruth})\\
$\unbalancedWasserstein[p]{R}(\nu_1,\nu_2)$
& unbalanced Wasserstein divergence between $\nu_1$ and $\nu_2$ (\cref{def:Wasserstein})\\
$\projdomstat$
& domain in one-dimensional position space (\cpageref{sec:models})\\
$\pf f\liftVar$
& pushforward of measure $\liftVar$ under map $f$ (\cref{def:pushforward})
\end{longtable}

\section[The reconstruction model and its properties]{\added{\lookUp{\ref{item3}}The reconstruction model and its properties}}\label{sec:models}
We aim to recover the state of particles in a compact domain $\domstat \subset \R^d$ (for instance $\Omega = [0,1]^d$) at different times within a compact temporal domain  $\domt \subset  \R$, for example $\domt = [-1,1]$ or $\domt$ being a finite collection of points. To this end, we assume that measurements are available at finitely many different points in time $t \in \alltimes$, where $\emptyset \neq \alltimes \subset \domt$ and $|\alltimes| <  \infty$. 

Our reconstruction approach uses lifted measures on the product space $\R^d \times \R^d$, as well as Radon-transform-based projections thereof. For the latter, we denote by $\alldirs \subset \sphere^{d-1}$ a compact set of projection directions, where we have either $\alldirs = \sphere^{d-1}$ or $\alldirs$ being a finite set of directions in $\sphere^{d-1}$ in mind. 
As domain for lifted measures on $\R^d \times \R^d$, we introduce the set
\begin{equation} \label{eq:domdyn_definitoin}
 \domdyn := \left\{ (x,v) \in \R^d \times \R^ d \,\middle|\, x + tv \in \domstat \text{ for all } t \in \domt \right\}  \subset \R^{d}\times\R^d. 
 \end{equation}
This ensures that no particle leaves or enters the domain during the whole duration of the experiment.
\begin{rem}[Size of the domains]
In practice, the region of interest (that region which is observed by the forward operator) will be a small area in the centre of $\domstat$
so that particles may leave or enter the observed region, but can still be represented in $\domdyn$.
Alternatively, to avoid a too strong enlargement of the domain around the region of interest (which would be accompanied by an increased computational effort)
one could also modify the involved operators such that all mass exiting the region of interest is represented in an auxiliary point reservoir.
\end{rem}
Projections of lifted measures, the position-velocity projections, are defined on the set
\[ \projdomdyn:=  \{ (\theta \cdot x,\theta \cdot v) \st (x,v) \in \domdyn, \, \theta \in \alldirs \}.\]
We will also require a domain for projections of measures living in $\domstat$ that we denote by
\[ \projdomstat := \{ \theta \cdot x \st x \in \domstat, \, \theta \in \alldirs \} .\]
Note that all the sets $\domdyn$, $\projdomdyn$ and $\projdomstat$ are compact as being the image of compacta under continuous functions.

The sets $\domstat, \domt, \alldirs, \domdyn, \projdomdyn$ and $\projdomstat$ are equipped with the standard metric of the underlying spaces $\R^d, \R, \sphere^{d-1}, \R^{d} \times \R^d, \R^2$ and $\R$, respectively, rendering them to be compact, separable metric spaces.
Further, in order to define integrals of continuous functions on them, these spaces are equipped with the Borel measures $\lebesgue^d, \mtime, \mdirs, \lebesgue ^d \times \lebesgue^d,\lebesgue^2$ and $\lebesgue^1$,\phantomsection\label{txt:measures} respectively, where $\lebesgue^n$ denotes the standard Lebesgue measure in $\R^n$ (restricted to the underlying set) and $\mtime$, $\mdirs$ are general Borel probability measures such that $\alldirs=\supp\mdirs$. For the latter two, we have in mind in particular the case of discrete sets with counting measures $\mtime,\mdirs$ and the case of sets with non-empty relative interior and the (rescaled) $1$-dimensional Lebesgue measure $\lebesgue^1$ or the (rescaled) $(d-1)$-dimensional Hausdorff measure $\hd^{d-1}$ (for $\mtime$ and $\mdirs$, respectively).

Besides the space of Radon measures on sets like $\domstat, \domt, \alldirs, \domdyn, \projdomdyn$ and $\projdomstat$ and products thereof, we will also deal with product spaces of the form
\[ \M(X)^I = \{ (\nu_i)_{i\in I} \st \nu_i \in \M(X)  \text{ for all } i \in I\}, \]
where $I$ is a (possibly uncountable) index set. While such sets can be generically equipped with the standard product topology (that is, the coarsest topology such that all projections to a single component are continuous), we will mostly not use any topology on these spaces but rather deal only with individual elements $\nu_i$ for $i \in I$. 
For $\nu\in\M(X)^\domt$ we will denote the individual elements by $\nu_t$ and for $\nu\in\M(X)^\alldirs$ by $\nu_\theta$.

We will also consider product spaces for finite index sets $J$, such as $J = \alltimes$. To highlight the difference, we will denote such spaces as $Z^{|J|}$, where $|J|$ defines the number of elements in $J$. In particular, we will deal with measurements $\data = (\data_t)_{t\in\alltimes}$ in a finite product of Hilbert spaces denoted by
$ H^{|\alltimes|}$.

\subsection{Radon and move operators and their properties}\label{cha:operators}
We will need the following transform operators.
\begin{defn}[Radon operators] \label{def:radon_transform_measure} \
\begin{itemize}
\item With $\Rs_\theta \nu: = \pf{[ x \mapsto x \cdot \theta  ]}\nu$, the \emph{Radon transform} $\Rs: \M(\R^d) \rightarrow \M(\alldirs \times \R)$ is defined as
\[ \Rs \nu = \mdirs \prodm \theta \Rs_\theta \nu.
\]
\item With $\Rf_\theta \nu: = \pf{[ (x,t) \mapsto (x \cdot \theta,t)  ]}\nu$, the \emph{timewise Radon transform} $\Rf: \M(\R^d \times \domt ) \rightarrow \M(\alldirs \times \R\times \domt )$ is defined as
\[ \Rf \nu = \mdirs \prodm \theta \Rf_\theta \nu.
\]
\item With $\Rj_\theta \nu: = \pf{[ (x,y) \mapsto (x \cdot \theta,y\cdot \theta)  ]}\nu$,  the \emph{joint Radon transform} $\Rj: \M(\R^d \times \R^d) \rightarrow \M(\alldirs \times \R \times \R)$ is defined as  
\[ \Rj \nu = \mdirs \prodm \theta \Rj_\theta \nu.
\]
\end{itemize}
\end{defn}
\begin{rem}[Properties of Radon operators] \label{prop:radon_properties}
It follows from \cref{thm:pushforward_radon,prop:pushforward_properties,prop:properties_product_measure} that all Radon operators define bounded linear operators with norm bounded by $1$, that they map nonnegative measures to nonnegative measures of same norm, and that they are continuous with respect to narrow convergence of measures. Further, as can be seen by direct computation, they map $L^1$-functions (regarded as a subset of Radon measures) to $L^1$-functions, and the following explicit representations hold true almost everywhere,
\begin{align*}
\nu \in L^1(\R^d): &  \qquad \Rs_\theta \nu(s) = \int_{ \{ x \in \R^d \st x \cdot \theta = s\}} \nu(x) \wrt \hd^{d-1}(x), \\
\nu \in L^1(\R^d \times \domt): & \qquad \Rf_\theta \nu(s,t) = \int_{ \{ x \in \R^d \st x \cdot \theta = s\}} \nu(x,t) \wrt \hd^{d-1}(x), \\
\nu\in L^1(\R^d \times \R^d): & \qquad \Rj_\theta \nu(s,r) = \int_{ \{ (x,y)\in \R^d \times \R^d \st x \cdot \theta = s,\,  y \cdot \theta = r\}} \nu(x,y) \wrt \hd^{2d-2}(x,y),
\end{align*}
where $\hd^n$ denotes the $n$-dimensional Hausdorff measure.
\end{rem}

\begin{rem}[Decomposition of Radon operators]\label{rem:radon_operator_decomposition}
A straightforward computation shows that, in case $\nu = \nu_t \prodm t \mtime \in \M(\R^d \times \domt)$, 
the Radon operators as above can be decomposed as 
\[ \Rf \nu = \Rs \nu_t \prodm t \mtime =  (\mdirs  \prodm \theta \Rs_\theta \nu_t )\prodm t  \mtime. \]
\end{rem}

In the variational problem setting considered later on, we will only deal with measures with support contained in a compact set and exploit weak-$\ast$ sequential compactness properties. In such a setting, \cref{prop:pushforward_properties,prop:properties_product_measure} provide a characterization of the support of the Radon transforms of such measures and corresponding weak-$\ast$ continuity assertions, which are summarized in the following proposition.

\begin{prop}[Radon operators for compactly supported measures]\label{prop:supports}
With $\theta \in \alldirs$ arbitrary, we have the following.
\begin{itemize}
\item For $\nu \in \M(\R^d)$ with $\supp(\nu) \subset \domstat$ it holds that 
$ \supp(\Rs_\theta \nu) \subset  \projdomstat$, $ \supp(\Rs \nu) \subset \alldirs \times \projdomstat$, 
and, using zero extension, the operator restrictions 
\[ \Rs_\theta : \M(\domstat) \rightarrow \M( \projdomstat)  \quad \text{and}\quad \Rs: \M(\domstat) \rightarrow \M(\alldirs \times \projdomstat) \] are weakly-$\ast$ continuous.

\item For $\nu \in \M(\R^d\times \domt)$ with $\supp(\nu) \subset \domstat\times \domt$ it holds that 
$ \supp(\Rf_\theta \nu) \subset  \projdomstat\times \domt$, $ \supp(\Rf \nu) \subset \alldirs \times \projdomstat\times \domt$, 
and, using zero extension, the operator restrictions 
\[ \Rf_\theta : \M(\domstat\times \domt) \rightarrow \M( \projdomstat\times \domt)  \quad \text{and}\quad \Rf: \M(\domstat\times \domt) \rightarrow \M(\alldirs \times \projdomstat\times \domt) \] are weakly-$\ast$ continuous.

\item For $\nu \in \M(\R^d\times \R^d)$ with $\supp(\nu) \subset \domdyn$ it holds that 
$ \supp(\Rj_\theta \nu) \subset  \projdomdyn$, $ \supp(\Rj \nu) \subset \alldirs \times \projdomdyn$,
and, using zero extension, the operator restrictions 
\[ \Rj_\theta : \M(\domdyn) \rightarrow \M( \projdomdyn) \quad \text{and}\quad \Rj: \M(\domdyn) \rightarrow \M(\alldirs \times \projdomdyn ) \] are weakly-$\ast$ continuous.
\end{itemize}
\end{prop}

\begin{rem}[Alternative definition of Radon operators] Note that an alternative definition of $\Rs:\M(\R^d ) \rightarrow \M(\alldirs \times \R)$ (and similarly for $\Rf$ and $\Rj$) could be the dual of the operator $\Rs^*:C_0(\alldirs \times \R) \rightarrow C_0(\R^d)$ given as
\[ \Rs^* \varphi(x) = \int_{\alldirs}\varphi (\theta,x \cdot \theta) \wrt \mdirs (\theta). \]
While this would directly yield weak-$\ast$ continuity of $\Rs$ also for measures on unbounded domains, it requires the measure $\mdirs$ to be absolutely continuous with respect to the standard Hausdorff measure on the sphere $\sphere^{d-1}$.  Indeed, the crucial point here is to show that $\ x \mapsto \int_{\alldirs}\varphi (\theta,x \cdot \theta) \wrt \mdirs (\theta)$ indeed vanishes at infinity, which can be done as in \cite[Lemma 1]{boman2009radon_transform_measure} in case of absolute continuity of $\mdirs$, but does not hold true for instance if $\alldirs$ is finite and $\mdirs$ is the counting measure.
\end{rem}

We will also need injectivity of the Radon operator $\Rs$ acting on measures. As in the classical setting with Schwartz functions, this can be obtained from a Fourier slice theorem provided that sufficiently many directions $\theta \in \Theta$ are measured.
\begin{defn}[Fourier transform]\label{def:FourierTransform}
For $\nu \in \M(\R^n)$ and $\xi \in \R^n$, we define the Fourier transform
\[ \ft(\nu)(\xi) = \hat{\nu}(\xi) = \int_{\R^n} e^ {-\ii x \cdot \xi} \wrt \nu (x) \]
and the inverse Fourier transform 
\[ \ft^{-1}(\nu)(\xi) = \check{\nu}(\xi) = (2\pi)^{-n} \int_{\R^n} e^ {\ii x \cdot \xi} \wrt \nu (x) \]
\end{defn}
Note that, since $x \mapsto e ^{\pm i x \cdot \xi}$ is continuous and bounded, $\hat{\nu}$ and $\check{\nu}$ are well-defined and are uniformly continuous, bounded functions.

The Fourier slice theorem relates the Fourier transform and the Radon transform.
\begin{prop}[Fourier slice theorem]\label{thm:fourierSlice} For $\nu \in \M(\R^d)$  we have
\[ \widehat{R_\theta \nu}(\sigma) = \int_\R \e^{-\ii s \sigma} \wrt R_\theta \nu (s) = \int_{\R^d} \e^{-\ii x \cdot \theta \sigma } \wrt \nu(x) = \hat{\nu}(\theta \sigma) \qquad \text{for all }\sigma\in\R,\,\theta\in\sphere^{d-1} .\]
\end{prop}

This allows to reduce injectivity of the Radon transform to injectivity of the Fourier transform on $\M(\R^d)$.

\begin{prop}[Injectivity of Radon operators] \label{prop:radon_injective}
If $\mdirs$ has a (nonzero) lower semi-continuous density with respect to the $(d-1)$-dimensional Hausdorff measure on $\sphere^{d-1}$, then
the Radon transform $\Rs:\M(\R^d) \rightarrow \M(\alldirs \times \R)$ is injective.
\begin{proof}
If $\Rs \nu = 0$ for $\nu  \in \M(\R^d)$, we obtain that $R_\theta \nu = 0$ for $\mdirs$-almost every $\theta \in \alldirs$ and, from the Fourier slice theorem, that $\hat{\nu}(\theta \sigma) = 0$ for $\mdirs$-almost every $\theta \in \alldirs$ and every $\sigma \in \R$. Now using the assumption on $\mdirs$, there exists a point $\theta_0 \in \alldirs$ such that the density of $\mdirs$ does not vanish in a neighbourhood of $\theta$ in $\sphere^{d-1}$. 
This implies that $\int_{\R^d } \varphi (x) \wrt \nu(x) = 0$ for any $\varphi  \in  A:= \vspan (\{ x \mapsto e^{-\ii \sigma \theta \cdot x} \st \sigma \in \R, \theta \in \sphere^{d-1}, \, \| \theta - \theta_0 \|< \epsilon \})$,
with $\epsilon>0$ suitably chosen. Now given any $\psi \in C_c(\R^d)$, $A$ is a subalgebra of $C(\supp(\psi))$ that separates points and contains the constant one function, thus by the Stone--Weierstrass Theorem, $A$ is dense in $C(\supp(\psi))$. Approximating $\psi$ with elements in $A$ we thus obtain that also $\langle \psi, \nu\rangle = 0$ and, as $\psi \in C_c(\R^d)$ was arbitrary, it follows from density that $\nu = 0$.
\end{proof}
\end{prop}

Next, we introduce two different operators that project lifted measures to a time-series.
\begin{defn}[Move operators] \label{defn:move_operators} Take $n \in \N$. 
\begin{itemize}
\item  With $\mv^n_t \nu := \pf{[ (x,v) \mapsto x + tv ]}\nu$, the \emph{move operator} $\mv^n:\M(\R^n \times \R^n) \rightarrow \M(\R^n \times \domt)$ is defined as
\[ \mv^n \nu =\mv^n_t \nu \prodm t \mtime.
\]
We set $\mv := \mv^1$.
\item With $\amvt[n]{t} \nu := \pf{[ (\theta,x,v) \mapsto (\theta,x+tv)]}\nu$, the \emph{anglewise move operator} $\amv[n]:\M(\alldirs \times \R^n \times \R^n) \rightarrow \M(\alldirs \times \R^n \times \domt)$  is defined as
\[ \amv[n] \nu = \amvt[n]{t} \nu \prodm t \mtime.
\]
We set $\amv := \amv[1]$.
\end{itemize}
\end{defn}

\begin{rem}[Properties of move operators] \label{prop:move_properties}
Again, it follows from \cref{thm:pushforward_radon}, \cref{prop:pushforward_properties} and \cref{prop:properties_product_measure} that all move operators define bounded linear operators with norm bounded by $1$, that they map nonnegative measures to nonnegative measures of same norm, and that they are continuous with respect to narrow convergence of measures. Just like for the Radon operators, they map $L^1$-functions to $L^1$-functions with the following explicit representations almost everywhere,
\begin{align*}
\nu \in L^1(\R^n \times \R^n): &  \qquad \mv_t \nu(z) = \int_{ \{ (x,v)  \in \R^n \times \R^n  \st x + tv  = z\}} \nu(x,v) \wrt \hd^n(x,v), \\
\nu \in L^1(\alldirs \times \R^n \times \R^n): & \qquad \amv_t \nu (\theta,z) =\int_{ \{ (x,v)  \in \R^n \times \R^n  \st x + tv  = z\}} \nu( \theta,x,v) \wrt \hd^n(x,v).
\end{align*}
\end{rem}

\begin{rem}[Decomposition of move operators] \label{rem:move_operator_decomposition}
Again, a straightforward computation shows that in case $\nu = \mdirs  \prodm \theta \nu_\theta \in \M(\alldirs \times \R^n \times \R^n)$ the move operators as above can be decomposed as 
\[ \amv[n] \nu = \mdirs \prodm \theta \mv^n \nu _\theta  = \mdirs \prodm \theta (\mv^n_t \nu_\theta \prodm t  \mtime).
 \]
\end{rem}

Again, we are interested in the restriction of move operators to measures on bounded domains and corresponding weak-$\ast$ continuity assertions, for which \cref{prop:pushforward_properties,prop:properties_product_measure} provide several properties which are summarized in the following \namecref{prop:move_supports}.

\begin{prop}[Move operators for compactly supported measures]\label{prop:move_supports}
With $t \in \domt$ arbitrary, we have the following.
\begin{itemize}
\item For $\nu \in \M(\R^d \times \R^d )$ with $\supp(\nu) \subset \domdyn$ it holds that $\supp(\mv^d_t \nu) \subset \domstat$, $ \supp(\mv^d \nu) \subset \domstat \times \domt $, and, using zero extension, the operator restrictions 
\[ \mv^d_t : \M(\domdyn) \rightarrow \M( \domstat )\quad \text{and}\quad \mv^d: \M(\domdyn) \rightarrow \M(\domstat \times \domt)  \] are weakly-$\ast$ continuous.

\item For $\nu \in \M(\alldirs \times \R^2)$ with $\supp(\nu) \subset \alldirs \times \projdomdyn$ it holds that $\supp(\amv_t \nu) \subset \alldirs \times  \projdomstat$, $ \supp(\amv \nu) \subset \alldirs \times  \projdomstat \times \domt $, and, using zero extension, the operator restrictions 
\[ 
\amv_t : \M(\alldirs \times \projdomdyn) \rightarrow \M( \alldirs \times  \projdomstat)
 \quad \text{and}\quad 
 \amv: \M(\alldirs \times \projdomdyn ) \rightarrow \M(\alldirs \times  \projdomstat\times \domt) \]
 are weakly-$\ast$ continuous.
\end{itemize}
\end{prop}

\begin{rem}[Alternative definition of move operators]
As with the Radon operators, an alternative definition of $\mv:\M(\R^d  \times \R^d) \rightarrow \M(\R^d \times \domt)$ (and similarly for $\amv$) could be the dual of the operator $\mv^*:C_0(\R^d \times \domt) \rightarrow C_0(\R^d \times \R^d)$ given as
\[ \mv^* \varphi(x,v) = \int_{\domt }\varphi (x + tv,t) \wrt \mtime(t), \]
which in particular implies weak-$\ast$ continuity of $\mv$ also for measures on unbounded domains. However, this requires %
to show that $(x,v) \mapsto \int_{\domt }\varphi (x + tv,t) \wrt \mtime(t)$ indeed vanishes at infinity. Again, while this can be done with similar techniques as in \cite[Lemma 1]{boman2009radon_transform_measure} in case of absolute continuity of $\mtime$, it does not hold true for instance if $\domt$ is finite and $\mtime$ is the counting measures.
\end{rem}

The following \namecref{lem:radon_move_exchange} provides a commutation rule for certain Radon and move operators.
\begin{lem}[Commutation of Radon and move operators] \label{lem:radon_move_exchange}
It holds that
\[ \Rf \mv^d = \amv \Rj \quad \text{as well as} \quad  \Rs_\theta \mv^d_t = \mv_t \Rj_\theta \text{ for all }t \in \R,\theta\in\sphere^{d-1}.
\]
\begin{proof}
Take $\nu \in \M(\R^d \times \R^d)$ and $\varphi \in C_0(\alldirs \times \R \times \domt)$. Then
\begin{align*}
\langle \Rf \mv^d \nu,\varphi \rangle 
& = \int _{\R^d \times \R^d}  \int _\domt \int_\alldirs \varphi(\theta,\theta \cdot (x + t y),t) \wrt \mdirs (\theta) \wrt \mtime(t) \wrt \nu(x,y) \\
& = \int _{\R^d \times \R^d}   \int_\alldirs \int _\domt \varphi(\theta,(\theta \cdot x) + t (\theta \cdot  y),t) \wrt \mtime(t) \wrt \mdirs (\theta)  \wrt \nu(x,y) \\
& = \int _{\R^d \times \R^d}   \int_\alldirs (\amv ^* \varphi)(\theta,\theta \cdot x, \theta \cdot y)  \wrt \mdirs (\theta)  \wrt \nu(x,y) \\
& = \int _{\R^d \times \R^d}   (\Rj^* \amv[]^* \varphi)( x,  y)  \wrt \nu(x,y) = \langle \amv \Rj \nu, \varphi  \rangle ,
\end{align*}
where we applied Fubini's theorem since both $(\domt,\mtime)$ and $(\alldirs,\mdirs)$ are finite measure spaces by assumption. %
The equality $\Rs_\theta \mv^d_t = \amv_t \Rj_\theta$  for all $t \in \R,\theta\in\sphere^{d-1}$ follows similarly.
\end{proof}
\end{lem}

\added{\lookUp{\ref{item3}}
We will also need that the constraint $\Rs u_t = \amv_t \liftVar   \quad \text{for all }t \in \domt $ for  $u \in\Mp(\R^d)^{\domt}$ and $\liftVar  \in \Mp(\alldirs \times  \R^2)$  provides a decomposition of $\liftVar$ as in the following lemma. 
\begin{lem}[Measure decomposition with projected lifting] \label{lem:projected_constraint_decomposition_mixed_model}
Let $u \in\Mp(\R^d)^{\domt}$ and $\liftVar  \in \Mp(\alldirs \times  \R^2)$ be such that
\[ \Rs u_t = \amv_t \liftVar   \quad \text{for all }t \in \domt \]
Then $\liftVar$ can be decomposed as
\[ \liftVar = \mdirs \prodm \theta \liftVar_\theta \]
with $(\liftVar_\theta)_{\theta\in \alldirs}$ a family of measures in $\Mp(\R^2)$ such that $\|\liftVar_\theta\| _\M = \|\liftVar \|_\M$. Further, $\|u_t\|_\M = \|\liftVar\|_\M$ for all $t \in \domt$ and 
\[ \Rs _\theta u_t = \mv_t \liftVar_\theta
\]
for every $t$ and $\mdirs$-almost every $\theta$.
\begin{proof}
At first, by the properties of the Radon and move operator, $\Rs u_t = \amv_t \liftVar $ implies that $\|u_t\|_\M = \|\liftVar\|_\M$ and 
\[ \mdirs \prodm \theta \Rs_\theta u_t = \amv_t \liftVar
\] 
for all $t \in \domt$. Now using the disintegration theorem \cite[Thm.\,5.3.1]{ambrosio2008gradient}, we decompose $\liftVar = \nu_\liftVar  \prodm \theta \liftVar_\theta$  with $\nu_{\liftVar}=\frac1{\|\liftVar\|_\M}\pf{(\pi^\alldirs)}\liftVar$ a probability measure and $\|\liftVar_\theta\|_\M = \|\liftVar\|_{\M}$ for almost every $\theta$. Here, $\pi^\alldirs :  \alldirs  \times \R^2  \rightarrow \alldirs$ denotes the canonical projection.
From this decomposition it follows by a direct computation that
\[
\amv_t \liftVar = \nu_\liftVar  \prodm  \theta \mv_t \liftVar_\theta .
\]
Hence we obtain
\[\mdirs \prodm \theta \Rs_\theta u_t= 
\nu_\liftVar  \prodm \theta \mv_t \liftVar_\theta \]
for all $t \in \domt$. Evaluating the measures on both sides 
at $\omega \times \R $ for $\omega \subset \alldirs$ an arbitrary Borel set leads to $ \|\liftVar\|_\M\nu_\liftVar (\omega) = \|u_t\|_\M\mdirs (\omega) $ for $t \in \domt$ arbitrary and thus $\nu_\liftVar=\mdirs$. This proves the decomposition of $\liftVar$.
In order to show that $ \Rs _\theta u_t = \mv_t \liftVar_\theta $ for every $t$ and $\mdirs$-almost every $\theta$, fix $t$ and note that we have already obtained that, for every $A \subset \R$ and $\omega \subset \alldirs$,
\[ \int _\omega \Rs _\theta u_t(A) - \mv_t \liftVar_\theta (A) \wrt \mdirs (\theta) =0.\]
In particular, this implies that, for any $A \subset \R$ there exists $\omega_A \subset \alldirs$ with $\mdirs (\omega_A) =0 $ such that $\Rs _\theta u_t(A) - \mv_t \liftVar_\theta (A) = 0$ for all $\theta \in \alldirs \setminus \omega_A$. Now taking $(A_n)_n$ to be a countable generator of the Borel $\sigma$-algebra of $\R$, $\omega_{A_n} \subset \alldirs$ with $\mdirs (\omega_{A_n} ) =0 $ such that $\Rs _\theta u_t({A_n} ) - \mv_t \liftVar_\theta ({A_n} ) =0$ for all $\theta \in \alldirs \setminus \omega_{A_n} $, and $\omega = \bigcup _n \omega_{A_n}$, it follows that $\Rs _\theta u_t({A_n} ) - \mv_t \liftVar_\theta ({A_n} ) =0$ for every $n$ and $\theta \in \alldirs \setminus \omega $ with $\mdirs (\omega ) =0 $, and, consequently, that $\Rs _\theta u_t = \mv_t \liftVar_\theta  $ for every $\theta \in \alldirs \setminus \omega $ as claimed.
\end{proof}
\end{lem}
}

Finally, we also obtain an injectivity result on the move operator as follows.
\begin{prop}[Injectivity of move operators]\label{prop:move_injective}
If 
$\mtime$ has a (nonzero) lower semi-continuous density with respect to the Lebesgue measure on $\R$
the move operator $\mv:\M(\R \times \R) \rightarrow \M(\R \times \domt)$ is injective.
\begin{proof}
Analogously to the Fourier slice theorem, we first observe for $\nu \in \M(\R \times \R)$, every $t \in \domt$ and $s \in \R$ the equality
\[ \widehat{\mv_t \nu} (s) 
= \int_\R e^{- \ii s r} \wrt \mv_t \nu (r) 
= \int_{\R^2} e^{- \ii s(x + tv)} \wrt  \nu (x,v) 
= \hat{\nu}(s,st), 
\]
from which injectivity again follows as in the proof of \cref{prop:radon_injective}.
\end{proof}
\end{prop}

Finally we introduce the operators providing the observations of the particle configurations.

\begin{defn}[Observation operators] \label{defn:pointwise_measurement_operator}
With $H$ a Hilbert space, for each $t \in \alltimes $ we set  
$\fopstat_t:\M(\domstat) \rightarrow H$ to be an \emph{observation operator} that is continuous with respect to the weak-$\ast$ topology in $\M(\domstat)$ and $H$
(or equivalently, that is the dual to a bounded linear operator $\fopstat_t^*: H \rightarrow C(\domstat)$).
The observation of a configuration $u\in \M(\domstat)^{\domt}$ is then given as
\[ u\mapsto (\fopstat_t u_t)_{t \in \alltimes}. \]
\end{defn}

For the purpose of this article, the exact form of $\fopstat_t$ does not matter so much
-- as long as it leads to good reconstructions for the static problem \eqref{eqn:staticModel}, it will do so as well for our proposed dimension-reduced model.
Typically the Hilbert space $H$ of measurements is finite-dimensional.
The example considered in \cite{Candes2014} is a truncated Fourier series.

\begin{ex}[Truncated Fourier series observation]\label{ex:truncatedFourier}
Assume $\Omega\subset(0,1)^d$, then one can interpret $[0,1)^d$ with periodic boundary conditions as the flat torus and take $\fopstat_t$ to be a Fourier series,
truncated at some maximum frequency $\maxFrequency\added{\geq 0}$,
\added{i.e. $\fopstat_t=\ftrunc$ where}
\begin{equation*}%
\added{\ftrunc u}\coloneqq\left(\int_\Omega e^{-\ii2\pi x\cdot\xi}\wrt u(x)\right)_{\xi\in\Z^d,\norm{\xi}_\infty\leq\maxFrequency}.
\end{equation*}
\removed{
\begin{equation*}\textstyle
\fopstat_t u_t=\left(\int_\Omega e^{-\ii2\pi x\cdot\xi}\wrt u_t(x)\right)_{\xi\in\Z^d,\norm{\xi}_\infty\leq\maxFrequency}.
\end{equation*}
}
\end{ex}

\subsection{Problem setting and dimensionality reduction}

As outlined in \cref{sec:intro_dim_red}, the approach of Alberti et al.\ in \cite{AlbertiAmmariRomeroWintz2019} for reconstructing a measure  $\lambda=\sum_{i=1}^nm_i\delta_{(x_i,v_i)}$ (describing the masses $m_i$, initial positions $x_i$ and velocities $v_i$ of a finite number of particles) from data $(\data_t)_{t\in\alltimes} \in H^{|\alltimes|}$ can be written as
\begin{equation} \label{eq:lifted_exact}
\min_{\lambda \in \Mp(\domdyn)} \| \lambda \|_\M \quad \text{such that } \fopstat_t\mv^d_t \lambda =  \data _t \quad \text{for all }t \in \alltimes.
\end{equation}
Using that both $\mv_t^d$ as well as $\fopstat_t$ are weakly-$\ast$ continuous, it follows by standard arguments that a solution to \eqref{eq:lifted_exact} exists. Further, in the case that each $\fopstat_t\added{=\ftrunc}$ samples a finite number of Fourier coefficients and under a separation condition on the particle distribution, Alberti et al.\ show exact reconstruction of particle trajectories in the noise-free case and, in the discretized setting, stability of the reconstruction approach with respect to noise (replacing the equality constraint by a fidelity term).

Our focus here is on reformulations of \eqref{eq:lifted_exact} that replace the high-dimensional unknown $\lambda$, which lives in a space of dimension $2d$, by unknowns defined in spaces of dimension at most $d+1$. To this end we first introduce explicit variables $u_t$ (which we term \emph{snapshots}) for the spatial particle configuration at different times $t$ in the extended time-domain $\domt \supset \alltimes$. With those we can write \eqref{eq:lifted_exact} equivalently as
\begin{equation} 
\begin{aligned}
\min_{ \substack{ \lambda \in \Mp(\domdyn)  \\ u \in \Mp(\domstat)^{\domt} }}
\| \lambda \|_\M \quad \text{such that } 
\left\{
\begin{aligned}
\mv_t^d  \lambda = u_t \quad \text{for all }t \in \domt, \\
\fopstat_tu_t = \data_t \quad \text{for all }t \in \alltimes.
\end{aligned}
\right.
\end{aligned}
\end{equation}
Indeed, the problems are equivalent since, for $\lambda \in \Mp(\domdyn)$,  $\mv^d_t \lambda \in \Mp(\domstat)$ for any choice of $t$.

Now we further reformulate the above problem by applying the Radon transform $\Rs$ on both sides of the constraint $\mv^d_t\lambda = u_t$. Using that $\Rs \mv^d_t = \amv_t \Rj$ (see \cref{lem:radon_move_exchange}) and introducing $\liftVar = \Rj \lambda$, we arrive at
\begin{equation} 
\begin{aligned}
\min_{ \substack{ \lambda \in \Mp(\domdyn) \\ \liftVar \in \Mp(\alldirs \times \projdomdyn)  \\ u \in \Mp(\domstat)^{\domt} }}
\| \liftVar \|_\M \quad \text{such that } 
\left\{
\begin{aligned}
\amv_t \liftVar  &= \Rs u_t \quad && \text{for all }t \in \domt, \\
 \fopstat_tu_t  &= \data_t \quad && \text{for all }t \in \alltimes, \\
 \Rj \lambda  &= \liftVar.
\end{aligned}
\right.
\end{aligned}
\end{equation}
Using that $\|\lambda \|_\M = \|\Rj \lambda \|_\M $ for $\lambda \in \Mp(\domdyn)$, this problem is equivalent to \eqref{eq:lifted_exact} whenever $\alldirs$ is such that $\Rs :\M(\domstat) \rightarrow \M(\alldirs \times \projdomstat)$ is injective (see \cref{prop:radon_injective}). Also, we note that the full-dimensional variable $\lambda$ defined on a domain of dimension $2d$ only appears in the constraint $\Rj \lambda = \liftVar$, all other variables are (depending on the choice of $\alldirs$ and $\domt$) of dimension at most $d+1$.

In the following, we will drop the constraint $\Rj \lambda = \liftVar$, thereby getting rid of the high-dimensional variable $\lambda$. Allowing also a deviation from the data $(\data_t)_{t\in\alltimes}$ in case of measurement noise we arrive at
\leqnomode
\begin{equation*} \label{eq:dim_reduced_general}  \tag*{\problemTagNoLink[\alpha]{\data}}
\min_{ \substack{ \liftVar \in \Mp(\alldirs \times \projdomdyn)  \\ u \in \Mp(\domstat)^{\domt} }}
\| \liftVar \|_\M  + \frac{1}{2\alpha}\sum_{t \in \alltimes} \|\fopstat_tu_t -  \data_t \|_H^2 \quad \text{such that } 
\amv_t \liftVar = \Rs u_t \quad \text{for all }t \in \domt 
\end{equation*}
\reqnomode
where $\alpha \in [0,\infty)$ and we include the hard constraint $\fopstat_tu_t = \data_t$ on the data by setting $\frac{1}{2\alpha} = \infty$ in case $\alpha = 0$
(specific instances of the optimization problem with parameter $\tilde\alpha$ and data $\tilde\data$ will be referred to as \problemTag[\tilde\alpha]{\tilde\data}).
Note that the choice of the data term is not essential, for instance one could develop the theory analogously without the square on the norm.
The main question of research in this work is to what extent such a dimension-reduced problem is a good approximation of the original problem \eqref{eq:lifted_exact}.

\added{
\begin{rem}[Equivalent model formulations] \label{rem:equivalent_formulations}\lookUp{\ref{item3}}
Note that here, different from the formal model \eqref{eqn:mixedModel} as presented in \cref{sec:intro}, we consider the constraint $\amv_t \liftVar = \Rs u_t$  for all $t \in \domt$  rather than $\mv_t^1 \liftVar_\theta = \Rs_\theta u_t \text{ for all }t\!\in\!\domt,\, \mdirs\text{-a.a.\,}\theta\!\in\!\alldirs$. The equivalence of these two formulations is apparent if $\liftVar$ can be disintegrated to $\liftVar=\mdirs(\theta)\otimes\liftVar_\theta$, and we refer to \cref{prop:unique_decomp_move} for a proof on when the constraint $\amv_t \liftVar = \Rs u_t$ implies the disintegration to exist.

In this context, an alternative would be to directly use the product space $\Mp( \projdomdyn)^{\alldirs}$ instead of the measure space $\Mp(\alldirs \times \projdomdyn)$ for the dimension-reduced measure. In fact, there are different alternative formulations of \ref{eq:dim_reduced_general} depending on whether product or measure spaces are used, and we refer to  \ref{sec:alternativeFormulations} for a detailed discussion of different model variants and their equivalence.

For the well-posedness and exact recovery results of this paper, \ref{eq:dim_reduced_general} will be the main setting, because using $\Mp(\domstat)^\domt$ for $u$ (as opposed to $\Mp(\domstat \times \domt)$) allows to apply the observation operator $\fopstat_t$ at every $t\in\domt$ and because using $\Mp(\alldirs \times \projdomdyn)$ for $\liftVar$ (as opposed to $\Mp( \projdomdyn)^\alldirs $) already provides some regularity of the measures $\liftVar$ with respect to $\theta \in \alldirs$. For the results on convergence for noisy data in \cref{sec:noisyReconstruction}, however, we will start by treating an (in that case simpler) product topology version of \ref{eq:dim_reduced_general}.
\end{rem}
}
At first, we deal with well-posedness of \ref{eq:dim_reduced_general}, showing in particular that existence is always ensured for $\alpha>0$ and holds for $\alpha=0$ whenever the original problem \eqref{eq:lifted_exact} has a solution.
We further show stability of the solution with respect to changes in the data.
To this end, let us call $(u,\liftVar)\in\Mp(\Omega)^\domt\times\Mp(\alldirs\times\projdomdyn)$ a \emph{$\alltimes$-cluster point} of a sequence $(u^n,\liftVar^n)\in\Mp(\Omega)^\domt\times\Mp(\alldirs\times\projdomdyn)$, $n=1,2,\ldots$,
if there exists a subsequence $n_1,n_2,\ldots$ with $u_t^{n_k}\weakstarto u_t$ for all $t\in\alltimes$ and $\liftVar^{n_k}\weakstarto\liftVar$ as $k\to\infty$
and if for all $t\in\domt\setminus\alltimes$ the measure $u_t$ is a (standard) cluster point of the subsequence $u_t^{n_1},u_t^{n_2},\ldots$ with respect to weak-$\ast$ convergence.

\begin{prop}[Well-posedness] \label{prop:well_posedness_dim_reduced_noisy} Let $\data = (\data_t)_{t\in\alltimes} \in H^{|\alltimes|}$ and $\alpha \in [0,\infty)$.
If $\alpha=0$, assume $f$ to be such that there exists $\lambda \in \Mp(\domdyn)$ with $\fopstat_t\mv^d_t\lambda = \data _t$ for all $t \in \alltimes$.

Then problem \ref{eq:dim_reduced_general} has a solution.
Moreover, for $\alpha>0$ the solutions are stable in the following sense:
Let $\data^1,\data^2,\ldots\in H^{|\alltimes|}$ be a sequence with $\data^n_t\to\data_t\in H$ for all $t \in \alltimes$ as $n\to\infty$, and let $(u^n,\liftVar^n)$ be a corresponding sequence of solutions to \problemTag[\alpha]{\data^n}. Then the following holds.
\begin{enumerate}
\item[i)] The sequence $(u^n,\liftVar^n)$ has a $\alltimes$-cluster point.
\item[ii)] Any $\alltimes$-cluster point $(u,\liftVar)$ of the sequence $(u^n,\liftVar^n)$ is a solution of \problemTag[\alpha]{\data}.
\item[iii)] If the solution $(u,\liftVar)$ to \problemTag[\alpha]{\data} is unique, the entire sequences $u_t^n$ and $\liftVar^n$ weakly-$\ast$ converge to $u_t$ and $\liftVar$, respectively, for all $t\in\domt$.
\end{enumerate}
\begin{proof}

We start with showing existence. To this end, first note that $(u,\liftVar) = (0,0)$ is admissible for \problemTag[\alpha]{\data} in case $\alpha>0$ and $((\mv^d_t \lambda)_{t\in\domt},\Rj\lambda)$ is admissible in case $\alpha =0$, hence the minimization is not over the empty set.

Take $(u^n,\liftVar^n)$, $n=1,2,\ldots$, to be a minimizing sequence. Then $\|\liftVar^n\|_\M$ is uniformly bounded, and the constraint $\amv_t \liftVar = \Rs u_t$ for all $t \in \domt$ together with \cref{prop:radon_properties,prop:move_properties} implies that also $\|u^n_t\|_\M$ is bounded uniformly in $n$ and $t$. Hence, there exists a subsequence $\liftVar^{n_k}$, $k=1,2,\ldots$, converging weakly-$\ast$ to some $\liftVar\in\M(\alldirs \times \projdomdyn)$, and for every $t\in\domt$ there is a further subsequence $u_t^{n_{k,t}}$ of $u_t^{n_k}$ (depending on $t$) which converges weakly-$\ast$ to some $u_t \in \M(\domstat)$.

Now we show that $((u_t)_{t\in\domt},\liftVar)$ is a solution of \problemTag[\alpha]{\data}. First note that weak-$\ast$-closedness of the positivity constraints on measures implies that both $\liftVar \in \Mp(\alldirs \times \projdomdyn)$ and $u_t \in \Mp(\domstat)$ for all $t \in \domt$. Further, for each $t \in \domt$, both $\liftVar^{n_{k,t}}$ and $u_t^{n_{k,t}}$ weakly-$\ast$ converge to $\liftVar$ and $u_t$, respectively. Hence, taking the limit $k \rightarrow \infty$ in $\amv_t \liftVar^{n_{k,t}} = \Rs u_t^{n_{k,t}}$ and using weak-$\ast$-to-weak-$\ast$ continuity of both $\mv_t$ and $\Rs$ we obtain $\amv_t \liftVar = \Rs u_t$. In addition, for $\alpha=0$ the weak-$\ast$-to-weak continuity of $\fopstat_t$ together with $\fopstat_tu_t^{n_{k,t}} = \data_t$ for all $k$ implies $\fopstat_t u_t = \data_t$ for all $t \in \alltimes $. Finally, weak-$\ast$ lower semi-continuity of the Radon norm on $\M(\alldirs \times \projdomdyn)$ and weak lower semi-continuity of $v \mapsto \|v - \data_t \|_H^2$ in case $\alpha>0$  imply that $((u_t)_{t\in\domt},\liftVar)$ is a minimizer as claimed.

Regarding the stability assertion, take $(u^n,\liftVar^n)$, $n=1,2,\ldots$, to be a sequence of solutions as claimed. With $(\tilde{u},\tilde{\liftVar})$ a solution of \problemTag[\alpha]{\data} it follows that
\begin{equation} \label{eq:stability_proof_optimality_estimate}
 \| \liftVar^n \|_\M  + \frac{1}{2\alpha}\sum_{t \in \alltimes} \|\fopstat_tu^n_t -  \data^n_t \|_H^2
\leq
\| \tilde \liftVar \|_\M  + \frac{1}{2\alpha}\sum_{t \in \alltimes} \|\fopstat_t\tilde u_t -  \data^n_t \|_H^2 
\to
\| \tilde \liftVar \|_\M  + \frac{1}{2\alpha}\sum_{t \in \alltimes} \|\fopstat_t\tilde u_t -  \data_t \|_H^2
\end{equation}
as $n\to\infty$.
This implies in particular that $\|\liftVar^n\|_\M$ is uniformly bounded and, consequently, also $\|u^n_t\|_\M$ for each $t \in \domt$. Hence, as $\alltimes$ is finite, we can select a joint subsequence $(u^{n_k},\liftVar^{n_k})$, $k=1,2,\ldots$, such that $u_t^{n_k}$ and $\liftVar^{n_k}$ converge weakly-$\ast$ to some $u_t$ and $\liftVar$ as $k\to\infty$ for each $t \in \alltimes$. In addition, for each $t \in \domt\setminus\alltimes$ we can select a further subsequence $u^{n_{k,t}}_t$ that weakly-$\ast$ converges to some $u_t$. This shows the existence of a $\alltimes$-cluster point.

Now assume that $(u,\liftVar)$ is a $\alltimes$-cluster point of the sequence $(u^n,\liftVar^n)$.
Then, convergence of $\liftVar^{n_{k,t}}$ to $\liftVar$ for each $t$ together with weak-$\ast$ continuity of $\amv_t$ and $\Rs$ implies that $\amv_t \liftVar = \Rs u_t$ for all $t\in\domt$. Also, convergence of $\liftVar^{n_k}$ and $u^{n_k}_{t}$ for all $t \in \alltimes$ to $\liftVar$ and $u_{t}$, respectively, allows us to take the limit inferior over indices $n=n_k$ in the left-hand side of \eqref{eq:stability_proof_optimality_estimate}, which, by lower semi-continuity, can be estimated from below by  $\|\liftVar \|_\M  + \frac{1}{2\alpha}\sum_{t \in \alltimes} \|\fopstat_tu_t -  \data_t \|_H^2$. Thus, \eqref{eq:stability_proof_optimality_estimate} shows that $(u,\liftVar)$ has energy no larger than $(\tilde{u},\tilde{\liftVar})$ and so is optimal as claimed.

Now assume that the solution $(u,\liftVar)$ to \problemTag[\alpha]{\data} is unique. If $\liftVar^n$ does not converge weakly\nobreakdash-$\ast$ to $\liftVar$, we can select $\varphi \in C(\alldirs \times \projdomdyn)$, $\epsilon >0$ and a subsequence $\liftVar^{n_m}$, $m=1,2,\ldots$, such that $|\langle \liftVar^{n_m} ,\varphi \rangle - \langle \liftVar ,\varphi \rangle | > \epsilon$ for all $m$. Applying the same arguments as above we can obtain a $\alltimes$-cluster point $(\tilde{u},\tilde{\liftVar})$ of that subsequence that is a solution to \ref{eq:dim_reduced_general}. Uniqueness implies $\tilde{\liftVar} = \liftVar$ and hence a contradiction. Now assume that there is some $t \in \domt$ such that $u^n_t$ does not converge weakly-$\ast$ to $u_t$. In an analogous manner as for $\liftVar$ we select a subsequence $u^{n_m}_t$, $m=1,2,\ldots$, and show that $(u,\liftVar)$ must be a $\alltimes$-cluster point of $(u^{n_m},\liftVar^{n_m})$, which again yields a contradiction. Hence also $u^n_t$ converges to $u_t$ for all $t\in\domt$ as claimed.
\end{proof}
\end{prop}

Next, we show a result on convergence for vanishing noise.
This result will also follow from the error estimates in \cref{sec:noisyReconstruction},
but those error estimates require some additional properties of the ground truth particle configuration and the observation operator.

\begin{prop}[Convergence for vanishing noise]
Let $\data^\dagger \in H^{|\alltimes|}$ such that there exists $\lambda^\dagger  \in \Mp(\domdyn)$ with $\fopstat_t\mv^d_t \lambda^\dagger = \data^\dagger _t$ for all $t \in \alltimes$. Let $\data^1,\data^2,\ldots\in H^{|\alltimes|}$ be a sequence of noisy data such that
\[ \delta^n:=\frac{1}{2}\sum_{t \in \alltimes} \|\data^n_t - \data^\dagger _t \|_H^2 \rightarrow 0 \quad \text{as } n \rightarrow \infty.
\] 
Take $\alpha^1,\alpha^2,\ldots>0$ to be a sequence of parameters such that
\[ \alpha^n \rightarrow 0 \text{ and } \frac{ \delta^n}{\alpha^n} \rightarrow 0 \quad \text{as } n \rightarrow \infty, \]
and let $(u^n,\liftVar^n)$, $n=1,2,\ldots$, be corresponding solutions to \problemTag[\alpha^n]{\data^n}.
Then the sequence $(u^n,\liftVar^n)$ has a $\alltimes$-cluster point, and any such cluster point is a solution to \problemTag[0]{\data^\dagger}.
Further, if the solution to \problemTag[0]{\data^\dagger} is unique, the entire sequences $\liftVar^n$ and $u^n_t$ weakly-$\ast$ converge to $\liftVar^\dagger$ and $u^\dagger_t$, respectively, for all $t \in \domt$.
\begin{proof}
By \cref{prop:well_posedness_dim_reduced_noisy} there exists a solution $(u^\dagger,\liftVar^\dagger)$ to \problemTag[0]{\data^\dagger}. By optimality of $(u^n,\liftVar^n)$ for \problemTag[\alpha^n]{\data^n} it follows that
\begin{equation} \label{eq:convergence_proof_optimality_estimate}
 \| \liftVar^n \|_\M  + \frac{1}{2\alpha^n}\sum_{t \in \alltimes} \|\fopstat_tu^n_t -  \data^n_t \|_H^2
\leq
\|  \liftVar^\dagger  \|_\M  + \frac{\delta^n}{\alpha^n}
\to
\|  \liftVar^ \dagger \|_\M 
\end{equation}
as $n \rightarrow \infty$.
This implies boundedness of $\|\liftVar^n\|_\M$ and $\|u^n_t \|_\M$ uniformly in $n$ and $t \in \domt$ and, arguing as in the proof of \cref{prop:well_posedness_dim_reduced_noisy}, existence of a $\alltimes$-cluster point $(u,\liftVar)$ as claimed.
Also, we note that estimating the left-hand side in \eqref{eq:convergence_proof_optimality_estimate} from below by the data term and multiplying by $\alpha^n$ yields
\begin{equation} \label{eq:convergence_proof_data_estimate}
 \limsup_{n\to\infty}\frac12\sum_{t\in\alltimes}\|\fopstat_{t} u^n_{t} - \data^n_t\|_H^2 \leq \limsup_{n\to\infty}\alpha^n\|\liftVar^\dagger\|_\M+\delta^n=0,
\end{equation}
which together with $\data^n\to\data^\dagger$ implies $\fopstat_tu^n_t\to\data^\dagger_t$ as $n\to\infty$ for all $t\in\alltimes$ and in particular $\fopstat_tu_t=\data^\dagger_t$.
Also, weak-$\ast$ lower semi-continuity of $\|\cdot \|_\M$ implies by \eqref{eq:convergence_proof_optimality_estimate} that $\|\liftVar\|_\M \leq \|\liftVar^\dagger \|_\M$. Finally, as in the proof of \cref{prop:well_posedness_dim_reduced_noisy} we obtain $\amv_t \liftVar = \Rs u_t$ for all $t \in \domt$ such that  $(u,\liftVar)$ is a solution of \problemTag[0]{\data^\dagger}.
In  case of uniqueness, we again argue as in the proof of \cref{prop:well_posedness_dim_reduced_noisy} to obtain convergence of the entire sequence $(u^n,\liftVar^n)$.
 \end{proof}
\end{prop}

\section{Exact reconstruction in the absence of noise}\label{sec:exactRecovery}
As discussed in \cref{sec:intro_dim_red} it was shown in \cite{Candes2014} that the exact positions and intensities of point sources (or equivalently positions and masses of particles) in $\R^d$ can be recovered from incomplete Fourier measurements. Alberti et al.\ \cite{AlbertiAmmariRomeroWintz2019} extended this result to the dynamic setting by representing moving particles as a collection of point sources at points $(x_i,v_i)\in\R^d\times\R^d$, $i=1,\dotsc,N$, where $x_i$ is the position of the $i$-th particle at time $t=0$ and $v_i$ is its velocity.
Our aim in this \namecref{sec:exactRecovery} is to prove exact reconstruction of the target measures also in our dimension-reduced problem \problemTag[0]{\data^\dagger}.

The conditions under which we prove exact reconstruction are of a geometric nature and just concern the particle positions and velocities, but not their masses.
For this reason we will in this \namecref{sec:exactRecovery} sometimes identify a particle configuration $\lambda=\sum_{i=1}^nm_i\delta_{(x_i,v_i)}\in\Mp(\R^d\times\R^d)$ with its support $\supp\lambda\subset\R^d\times\R^d$.
We also extend the definition of the joint Radon transform correspondingly by defining
\[
\Rj_\theta(\dsupp)=\set{(\thdot{x},\thdot{v})\given(x,v)\in\dsupp}
\]
for any $\dsupp\subset\R^d\times\R^d$ and $\theta\in\alldirs$.

The geometric conditions to extend exact reconstruction results from a static to the dynamic setting are based on the notion of \emph{ghost particles} from \cite{AlbertiAmmariRomeroWintz2019}.
To introduce it, for a given particle configuration $\dsupp=\set{(x_i,v_i)\given i=1,\dotsc,N}\subset\R^n\times\R^n$ (where we will use either $n=d$ or $n=1$)
we denote by $L_{i,t}$ the set of all points in position-velocity space for which the corresponding particle shares the same position at time $t$ with the $i$-th particle of $\dsupp$,
\begin{equation*}
    L_{i,t}\coloneqq\set{(x,v)\in\R^n\times\R^n\given x+tv=x_i+tv_i}.
\end{equation*}

\begin{defn}[Coincidence and ghost particle]\label{def:coincidence_ghost_particle}
Let $\dsupp=\set{(x_i,v_i)\given i=1,\dotsc,N}\subset\R^n\times\R^n$ be a particle configuration and $\goodtimes\subset\R$ a set of times.
\begin{enumerate}
\item
Configuration $\dsupp$ has a \emph{coincidence} with respect to $\goodtimes$ if the locations of two distinct particles $i\neq j$ coincide at some time $t\in\goodtimes$, that is, $x_i+t v_i= x_j+t v_j$
or equivalently $(x_i,v_i)\in L_{j,t}$ or $(x_j,v_j)\in L_{i,t}$.
\item
A point $(x,v)\in\R^n\times\R^n\setminus\dsupp$ is called a \emph{ghost particle} of $\dsupp$ with respect to $\goodtimes$ if there exist distinct indices $(i_t)_{t\in\goodtimes}\subset\set{1,\dotsc,N}$ such that
\begin{equation*}
    \bigcap_{t\in\goodtimes} L_{i_t, t} = \set{(x,v)}.
\end{equation*}
    The set of ghost particles of $\dsupp$ with respect to $\goodtimes$ is denoted $G(\dsupp,\goodtimes)$.
\end{enumerate}
\end{defn}
Note that the equality $x + t_i v = \tilde{x}+t_i \tilde{v}$ at two distinct time points $t_i \in  \goodtimes $, $i=1,2$, implies already $(x,v) = (\tilde{x},\tilde{v})$. In view of this, in the relevant case $|\goodtimes |\geq 2$, the intersection $\bigcap_{t\in\goodtimes} L_{i_t, t}$ contains always at most one point and  
it is straightforward to see that \cref{def:coincidence_ghost_particle} is consistent with \cref{def:coincidenceGhostParticle};
while the latter introduces coincidences and ghost particles for particle configurations $\lambda\in\Mp(\R^d\times\R^d)$, the former does so for their supports $\supp\lambda$.

Throughout this \namecref{sec:exactRecovery} we will denote the (noisefree) observations by $(\data^\dagger_t)_{t\in\alltimes}$
and the \emph{ground truth} particle configuration that led to the observations by $\lambda^\dagger\in\Mp(\domdyn)$, thus $\data^\dagger_t=\fopstat_t\mv^d_t\lambda^\dagger$.
Furthermore, we will set $u^\dagger_t=\mv^d_t\lambda^\dagger$, $t\in\alltimes$, to be the ground truth snapshots and $\liftVar^\dagger_\theta=\Rj_\theta\lambda^\dagger$, $\theta\in\alldirs$, to be the ground truth position-velocity projections.

\subsection{Exact reconstruction for dimension-reduced problem} \label{sec:exact_recon_dim_reduced}

The basic idea of Alberti et al.\ was to simply transfer the exact reconstruction property of the \emph{static} reconstruction problem from \cite{Candes2014} to their dynamic one.
The static reconstruction problem for time instant $t\in\alltimes$ is here given by
\begin{equation*}\tag*{$\mathrm P_{\text{stat}}^t(\data^\dagger_t)$}\label{eqn:erstat}
    \min_{u\in\Mp(\domstat)}\mnorm{u}
    \quad\text{such that } \fopstat_tu=\data^\dagger_t.
\end{equation*}

\begin{defn}[Exact reconstruction]
We say that the static problem \emph{\ref{eqn:erstat} reconstructs exactly} if the unique solution to \ref{eqn:erstat} is $u^\dagger_t$.
\end{defn}

Below we prove stepwise that our dimension-reduced model \problemTag[0]{\data^\dagger} essentially recovers the ground truth snapshots and position-velocity projections
as soon as the static problem reconstructs exactly at selected time points $t\in\alltimes$.
\added{\label{txt:itemeProofLogic}\lookUp{\ref{iteme}}%
The overall flow of the argument is the following:
We first show that if there are time points $t$ at which a static reconstruction of the correct particle configuration $u_t^\dagger$ is possible,
then also our dimension-reduced dynamic problem will reconstruct the snapshot at those time points exactly;
in other words, the exact reconstruction at single time points is not hindered by the dynamic setting.
These exactly reconstructed particle configurations at a few time points then serve as a seed to allow exact reconstruction also at other time points.
However, since the snapshots $u_t$ at the different time points only interact indirectly via the position-velocity projections $\liftVar_\theta$,
we first have to show that those position-velocity projections are indeed exactly reconstructed, from which it follows in the last step that all snapshots $u_t$ are exactly reconstructed.

We begin by showing that the dynamic setting does not interfere with exact reconstruction at single time points.
}%

\begin{thm}[Exact reconstruction of good snapshots]\label{thm:exactrecov}
    Let $\emptyset \neq \goodtimes\subset\alltimes$ be a subset of times $t$ at which the static problem \ref{eqn:erstat} reconstructs exactly,
    and let $(u,\liftVar)$ solve \problemTag[0]{\data^\dagger}.
    Then $u_t=u^\dagger_t$ for all $t\in\goodtimes$.
\end{thm}
\begin{proof}
    Let $(u,\liftVar)$ solve \problemTag[0]{\data^\dagger} and
    fix an arbitrary time $t\in\goodtimes$. Then $\fopstat_tu_t=\data^\dagger_t$ so that $u_t$ is feasible for \ref{eqn:erstat}.
    Furthermore, optimality of $(u,\liftVar)$ in \problemTag[0]{\data^\dagger} yields
    \begin{equation*}
        \mnorm{\liftVar}\leq\mnorm{\liftVar^\dagger}
    \end{equation*}
    since $(u^\dagger,\liftVar^\dagger)$ is an admissible competitor in \problemTag[0]{\data^\dagger}.
    The condition $\amv_t\liftVar=\Rs u_t$ for all $t\in\domt$ together with \cref{prop:radon_properties,prop:move_properties} then implies
    \begin{equation*}
        \mnorm{u_t}=\mnorm{\liftVar}\leq\mnorm{\liftVar^\dagger}=\mnorm{u^\dagger_t}.
    \end{equation*}
    Thus we see that $u_t$ is in fact a solution to \ref{eqn:erstat}.
    Since we assumed $u^\dagger_t$ to be the unique solution to \ref{eqn:erstat}, we can conclude $u_t=u^\dagger_t$.
\end{proof}

\added{\label{txt:itemeProofLogic2}\lookUp{\ref{iteme}}%
As explained before, the above proven exact snapshot reconstruction at few time points induces the exact recovery of the position-velocity projections, which is shown next.
}%

\begin{thm}[Exact reconstruction of position-velocity projections]\label{thm:exactrecov_gooddirs}
    As in \cref{thm:exactrecov} let $\emptyset \neq \goodtimes\subset\alltimes$ be a subset of times $t$ at which the static problem \ref{eqn:erstat} reconstructs exactly,
    and let $(u,\liftVar)$ solve \problemTag[0]{\data^\dagger}.
    Further define
    \begin{equation*}
        \gooddirs\coloneqq\{\theta\in\alldirs\mid\text{$\Rj_\theta(\supp\lambda^\dagger)$ has neither coincidences nor ghost particles with respect to $\goodtimes$}
        \}.
    \end{equation*}
    Then for any disintegration $\liftVar=\mdirs\prodm\theta \liftVar_\theta$ (which exists by \cref{lem:projected_constraint_decomposition_mixed_model})
    we have $\liftVar_\theta=\liftVar^\dagger_\theta$ for $\mdirs$-almost all $\theta\in\gooddirs$.
\end{thm}
\begin{proof}
    By \cref{lem:projected_constraint_decomposition_mixed_model}, for $\mdirs$-almost all $\theta\in\gooddirs$ we have 
    \begin{equation*}
        \mnorm{\liftVar_\theta}=\mnorm{\liftVar}\quad\text{and}\quad
        \Rs_\theta u_t=\mv_t\liftVar_\theta\quad\text{for all }t\in\domt.
    \end{equation*}
    Now fix one such $\theta\in\gooddirs$ and
    introduce the reconstruction error $\xi\coloneqq\liftVar_\theta-\liftVar^\dagger_\theta\in\M(\R^2)$.
    Letting $\supp\lambda^\dagger=\{(x_i,v_i)\,|\,i=1,\ldots,N\}$, we consider the Lebesgue decomposition of $\xi$ with respect to $\abs{\liftVar^\dagger_\theta}$, which yields a representation
    \begin{equation*}
        \xi=\xi_a+\xi_s,\quad
        \xi_a=\sum_{i=1}^N\beta_i\delta_{(\theta\cdot x_i,\theta\cdot v_i)},\quad
        \xi_s\perp\liftVar^\dagger_\theta
    \end{equation*}
    with weights $\beta_i\in\R$. Next, for each $t\in\goodtimes$ define the function $\tilde{q}_t:\R\mapsto\R$ by
    \begin{equation*}
        \tilde{q}_t(s)\coloneqq\begin{cases}
            \sgn\beta_i&\text{if $s=\theta\cdot(x_i+t v_i)$,}\\
            0 & \text{otherwise.}
        \end{cases}
    \end{equation*}
    These functions are well-defined, since the projected configurations are free of coincidences by choice of $\theta$. Further define $q:\R^2\to\R$ by
    \begin{equation*}
        q(y,w)\coloneqq\frac{1}{|\goodtimes|}\sum_{t\in\goodtimes}\tilde{q}_t(y+t w),
    \end{equation*}
    then we have $q(\thdot{x_i},\thdot{v_i})=\sgn\beta_i$ for all $i\in\set{1,\dotsc,N}$.

    The function $q$ is defined as an average of terms with modulus no larger than 1. Therefore, if $(y,w)\in\R^2$ with $\abs{q(y,w)}=1$, we necessarily have $\abs{\tilde{q}_t(y+tw)}=1$ for all $t\in\goodtimes$ and thus
    \begin{equation}\label{eqn:proj-lin}
        \forall t\in\goodtimes\,\exists i_t\in\set{1,\dotsc,N}:\quad y+tw=\thdot{(x_{i_t}+tv_{i_t})}.
    \end{equation}
    There are now two possible cases:
    \begin{enumerate}
    \item
    One of the indices $i_t$ repeats. In this case the point $(y,w)$ is uniquely determined by the equation~\eqref{eqn:proj-lin} at the two corresponding times, and it follows that $(y,w)=(\thdot{x_{i_t}},\thdot{v_{i_t}})$.
    \item
    The indices $(i_t)_{t\in\goodtimes}$ are all distinct. This would mean that $(y,w)$ is a ghost particle of the projected configuration, which is not allowed by choice of $\theta$.
    \end{enumerate}
    We have thus shown that
    \begin{equation}\label{eqn:avgc-strict}
        \abs{q(y,w)}<1\quad\text{for all } (y,w)\in\R^2\setminus\supp\liftVar_\theta^\dagger.
    \end{equation}
    By \cref{thm:exactrecov} we have the exact reconstruction $u_t=u^\dagger_t$ for all $t\in\goodtimes$. Thus, by choice of $\theta$ we obtain
    \begin{equation*}
        \mv_t\liftVar_\theta=\Rs_\theta u_t
        =\Rs_\theta u^\dagger_t=\mv_t\liftVar^\dagger_\theta%
    \end{equation*}
    and therefore $0=\mv_t\xi$ for all $t\in\goodtimes$. Using the previously defined function $q$ we can calculate
    \begin{equation*}
        0=\frac{1}{|\goodtimes|}\sum_{t\in\goodtimes}\int_\R\tilde{q}_t\wrt\mv_t\xi
        =\int_{\R^2}q\wrt \xi
        =\int_{\R^2}q\wrt \xi_a+\int_{\R^2}q\wrt \xi_s
        =\sum_{i=1}^N|\beta_i|+\int_{\R^2}q\wrt \xi_s.
    \end{equation*}
    If $\xi_s=0$, then also $\mnorm{\xi_a}=\sum_{i=1}^N|\beta_i|=0$, so there is nothing left to show.
    On the other hand, $\xi_s\neq0$ leads to a contradiction:
    Since $|q|<1$ outside the support of $\liftVar^\dagger_\theta$ by \eqref{eqn:avgc-strict}, we get the inequality
    \begin{equation*}\label{eqn:erm-xaxs}
        \mnorm{\xi_a}=\sum_{i=1}^N|\beta_i|=-\int_{\R^2}q\wrt \xi_s
        <\mnorm{\xi_s}.
    \end{equation*}
    By optimality of $\liftVar$ we get
    \begin{equation*}\label{eqn:erm-ineq-start}
        \mnorm{\liftVar^\dagger_\theta}=\mnorm{\liftVar^\dagger}\geq\mnorm{\liftVar}=\mnorm{\liftVar_\theta}=\mnorm{\liftVar^\dagger_\theta+\xi},
    \end{equation*}
    and from $\liftVar^\dagger_\theta+\xi_a\perp\xi_s$ it follows that
    \begin{equation*}
        \mnorm{\liftVar^\dagger_\theta+\xi}=\mnorm{\liftVar^\dagger_\theta+\xi_a}+\mnorm{\xi_s}
        \geq\mnorm{\liftVar^\dagger_\theta}-\mnorm{\xi_a}+\mnorm{\xi_s}
        >\mnorm{\liftVar^\dagger_\theta}.
    \end{equation*}
  Combining the last two inequalities yields the desired contradiction.
\end{proof}

\added{\label{txt:itemeProofLogic3}\lookUp{\ref{iteme}}%
Finally, from the exactly reconstructed position-velocity projections we can infer the snapshots at all times
(or at least those times for which the information in the position-velocity projections suffices to uniquely specify a particle configuration).
}%

\begin{thm}[Exact reconstruction of arbitrary snapshots]\label{thm:exactrecovAllSnapshots}
    Again as in \cref{thm:exactrecov} let $\emptyset \neq \goodtimes\subset\alltimes$ be a subset of times $t$ at which the static problem \ref{eqn:erstat} reconstructs exactly,
    and let $(u,\liftVar)$ solve \problemTag[0]{\data^\dagger}.
    Further, fix some $t\in\domt$. If $\gooddirs$ from \cref{thm:exactrecov_gooddirs} is such that $u^\dagger_t$ is uniquely determined by $(\mdirs\restr\gooddirs)\prodm\theta\Rs_\theta u_t^\dagger$, the Radon transform restricted to the directions in $\gooddirs$, then $u_t=u^\dagger_t$.
\end{thm}
\begin{proof}
    This follows immediately from the relation $\Rs_\theta u_t=\mv_t\liftVar_\theta=\mv_t\liftVar^\dagger_\theta=\Rs_\theta u^\dagger_t$, which holds for all $t\in\domt$ and by \cref{thm:exactrecov_gooddirs} for $\mdirs$-almost all $\theta\in\gooddirs$.
\end{proof}

The above reconstruction results depend on the set $\gooddirs$ for which reason we analyse it in more detail in the next section.
In particular, we show that $\gooddirs$ coincides with $\alldirs$ up to a $\hd^{d-1}$-nullset,
which for instance implies that for $\alldirs=\sphere^{d-1}$ and $\mdirs$ the (scaled) $(d-1)$-dimensional Hausdorff measure both $\gamma^\dagger$ and $u^\dagger$ are exactly reconstructed.
This fact is discussed in the next section, while the subsequent section considers the opposite case of a finite set $\alldirs$.

\subsection{Coincidences and ghost particles in projections}\label{sec:coincidencesGhostsProjections}

In this section we show that $\gooddirs$ from \cref{thm:exactrecov_gooddirs} contains $\hd^{d-1}$-almost all directions of $\alldirs$
\added{\lookUp{\ref{iteme}}(which will eventually imply \cref{thm:exactReconstruction_hausdorff})}.
To this end we consider projections $\Rj_\theta(\dsupp)$ of particle configurations $\dsupp=\set{(x_i,v_i)\given i=1,\dotsc,N}$,
which themselves represent particle configurations in one-dimensional space.
Their sets of ghost particles with respect to $\goodtimes$ are $G(\Rj_\theta(\dsupp),\goodtimes)$, and we furthermore abbreviate
\begin{equation*}
    L_{i,t}(\theta)\coloneqq\set{(x,v)\in\R\times\R\given
    x+t v=\thdot{x_i}+t\thdot{v_i})}
    =\Rj_\theta(L_{i,t}).
\end{equation*}
We will separately consider coincidences and ghost particles
\added{\label{txt:itemeProofLogic4}\lookUp{\ref{iteme}}%
and show that for $\hd^{d-1}$-almost all directions $\theta$,
if a particle configuration $\dsupp$ is free of coincidences or ghost particles, then so is its projection $\Rj_\theta(\dsupp)$.
We begin with coincidences.
}%

\begin{lem}[Coincidences in projections]\label{thm:projcoinc}
Let $\goodtimes$ be a finite set of times and
let the particle configuration $\dsupp$ be free of coincidences at those times.
Then also $\Rj_\theta(\dsupp)$ is free of coincidences at those times for $\hd^{d-1}$-almost all $\theta\in\sphere^{d-1}$.
\end{lem}
\begin{proof}
    For a $\thins$ and indices $i,j\in\set{1,\dotsc,N}$, $i\neq j$, the projected particles $(\thdot{x_i, \thdot{v_i}})$ and $(\thdot{x_j, \thdot{v_j}})$ take the same position at a time $t\in\goodtimes$ if and only if $\theta$ fulfills the equation
    \begin{equation*}
        \thdot{(x_i-x_j+t (v_i - v_j))}=0.
    \end{equation*}
    Therefore we can write
    \begin{multline*}
        \set{\thins\given
        \text{$\Rj_\theta(\dsupp)$ has a coincidence at a time in $\goodtimes$}}\\
        = \bigcup_{t\in\goodtimes}\bigcup_{i\neq j}
        \set{\thins\given\thdot{(x_i-x_j+t (v_i - v_j))}=0}.
    \end{multline*}
    We have $x_i-x_j+t (v_i - v_j)\neq 0$ for all $i\neq j$ since the configuration is assumed to be free of coincidences for $t\in\goodtimes$.
    Hence, the above is a finite union of $(d-1)$-dimensional hyperplanes in $\R^d$ intersected with the sphere $\sphere^{d-1}$.
    Since those intersections are $\hd^{d-1}$-nullsets, their union is as well.
\end{proof}

\added{\label{txt:itemeProofLogic5}\lookUp{\ref{iteme}}%
For ghost particles the situation is slightly more complicated than for coincidences:
As the below result shows, the projection $\Rj_\theta$ may induce new, additional ghost particles.
However, it does not for $\hd^{d-1}$-almost all directions $\theta$.
}%

\begin{lem}[Ghost particle sets in projections]\label{thm:projghost-ext}
    Let $\goodtimes$ be a finite set of times.
    The ghost particles of a particle configuration $\dsupp$ and of its projections $\Rj_\theta\dsupp$ satisfy the following relation.
   \begin{enumerate}
    \item\label{claim:ghost-ind} For all $\thins$, as long as $\Rj_\theta$ is injective on $\dsupp$, it holds that
    $ G(\Rj_\theta(\dsupp), \goodtimes)\supset\Rj_\theta(G(\dsupp,\goodtimes))$.
    \item\label{claim:ghost-aa} For $\hd^{d-1}$-almost all $\theta\in\sphere^{d-1}$ we even have
    $G(\Rj_\theta(\dsupp), \goodtimes)=\Rj_\theta(G(\dsupp,\goodtimes))$.
   \end{enumerate}
\end{lem}
\begin{proof}
    To prove the first statement, take a ghost particle $({x},{v})\in G(\dsupp, \goodtimes)$ of the original configuration. By definition, there exist pairwise distinct indices $(i_t)_{t\in\goodtimes}\subset\set{1,\dotsc,N}$ such that
    \begin{equation*}
        x_{i_t}+t v_{i_t}={x}+t{v}\quad\text{for all }t\in\goodtimes.
    \end{equation*}
    Taking the inner product with $\theta$ on both sides of these equations, we see that the projected ghost particle $(\thdot{{x}},\thdot{{v}})=\Rj_\theta({x},{v})$ is a ghost particle of the projected configuration $\set{(\thdot{x_i},\thdot{v_i})\given i=1,\dotsc,N}$. The assumption of injectivity is necessary to ensure that the indices $(i_t)_t$ still address distinct particles after the projection.

    To prove the second statement, we first show that the inclusion from the first statement holds for $\hd^{d-1}$-almost all $\thins$. Indeed, in order for $\Rj_\theta$ to not be injective on $\dsupp$, there need to exist indices $1\leq i,j\leq N, i\neq j$, such that $(\thdot{x_i},\thdot{v_i})=(\thdot{x_j},\thdot{v_j})$.
    We can thus write
    \begin{equation*}
        \set{\thins\given\text{$\Rj_\theta$ is not injective on $\dsupp$}}
        \subset\bigcup_{i\neq j}\set{\thins\given\text{$\thdot{(x_i-x_j)}=0$ and $\thdot{(v_i-v_j)}=0$}}.
    \end{equation*}
    Since the particles in the original configuration are assumed to be distinct, for $i\neq j$ we either have $x_i-x_j\neq 0$ or $v_i-v_j\neq 0$, which implies that the sets in the union are intersections of $\sphered$ with hyperplanes and thus $\hd^{d-1}$-nullsets. Thus $\Rj_\theta$ is injective on $\dsupp$ for $\hd^{d-1}$-almost all $\theta$ as desired.
    
    It remains to show the inclusion $G(\Rj_\theta(\dsupp), \goodtimes)\subset \Rj_\theta(G(\dsupp,\goodtimes))$ for $\hd^{d-1}$-almost all $\thins$. To this end, we define the set
    \begin{equation*}
        M\coloneqq\set{\thins\given
        \Rj_\theta(G(\dsupp,\goodtimes))\subsetneq G(\Rj_\theta(\dsupp),\goodtimes)}
    \end{equation*}
    of projection directions for which the above inclusion is violated.
     We need to show that this set has measure zero. Writing $\goodtimes=\set{t_1,\dotsc,t_\ngoodtimes}$,  $M$ can be expressed as
\begin{align}
    M&=\set*{\thins\given
    \Rj_\theta(G(\dsupp,\goodtimes))\subsetneq\bigcup_{(i_p)_p}\bigcap_{p=1}^{\ngoodtimes} L_{i_p,t_p}(\theta)} \nonumber\\
    &\subset\bigcup_{(i_p)_p}\set*{\thins\given
    \bigcap_{p=1}^{\ngoodtimes} L_{i_p,t_p}(\theta)
    \setminus\Rj_\theta(G(\dsupp,\goodtimes))\neq\varnothing}\label{eqn:ghost-union},
\end{align}
where the union runs over all families of distinct indices $(i_p)_{1\leq p \leq\ngoodtimes}\subset\set{1,\dotsc,N}$.
Since the union is finite, it suffices to show that every set in the union is a $\hd^{d-1}$-nullset.

Let $(i_p)_{1\leq p\leq\ngoodtimes}$ be such a family of indices and fix a $\thins$.
It is straightforward to see that the pairwise intersections $L_{i_p,t_p}\cap L_{i_{p+1},t_{p+1}}$ are singletons. We write
\begin{equation*}
    \set{(a_p,b_p)}=L_{i_p,t_p}\cap L_{i_{p+1},t_{p+1}},\quad p=1,\dotsc,\ngoodtimes-1,
\end{equation*}
and calculate
\begin{equation}\label{eqn:ghost-caps}
\bigcap_{p=1}^\ngoodtimes L_{i_p,t_p}(\theta)
=\bigcap_{p=1}^{\ngoodtimes-1}L_{i_p,t_p}(\theta)\cap L_{i_{p+1},t_{p+1}}(\theta)\\
=\bigcap_{p=1}^{\ngoodtimes-1}\Rj_\theta(L_{i_p,t_p}\cap L_{i_{p+1},t_{p+1}})
=\bigcap_{p=1}^{\ngoodtimes-1}\Rj_\theta\set{(a_p,b_p)}.
\end{equation}
Now we can distinguish two cases:
\begin{enumerate}
\item
All points $(a_p, b_p)$ coincide, that is,
\begin{equation*}
    \bigcap_{p=1}^\ngoodtimes L_{i_p,t_p} = \set{(a_1,b_1)}.
\end{equation*}
In this case, $(a_1,b_1)\in G(\dsupp, \goodtimes)$ is a ghost particle of the full-dimensional configuration, and \eqref{eqn:ghost-caps} implies
\begin{equation*}
    \bigcap_{p=1}^\ngoodtimes L_{i_p,t_p}(\theta)\setminus\Rj_\theta(G(\dsupp, \goodtimes)) =\varnothing.
\end{equation*}
\item
There exists a $q\in\set{1,\dotsc,\ngoodtimes-1}$ such that $(a_q,b_q)\neq(a_{q+1},b_{q+1})$. Assume without loss of generality that $a_q\neq a_{q+1}$.
If the intersection in \eqref{eqn:ghost-caps} is nonempty, we see from its right-hand side that
\begin{equation*}
    (\thdot{a_q},\thdot{b_q})=\Rj_\theta(a_q,b_q)=\Rj_\theta(a_{q+1},b_{q+1})=(\thdot{a_{q+1}},\thdot{b_{q+1}}).
\end{equation*}
This can only hold if $\theta$ is contained in the hyperplane
\begin{equation*}
    \set{\theta\in\R^d\given\thdot{(a_q-a_{q+1})=0}}.
\end{equation*}
\end{enumerate}
By these cases, every set in the union in \eqref{eqn:ghost-union} is either empty or contained in a hyperplane intersected with $\sphere^{d-1}$, which means that $M$ is contained in a finite union of $\hd^{d-1}$-nullsets.
\end{proof}

The following \namecref{thm:ghost} is an immediate consequence.

\begin{lem}[Ghost particles in projections]\label{thm:ghost}
    Let $\goodtimes$ be a finite set of times
    and assume that the particle configuration $\dsupp$ admits no ghost particle with respect to $\goodtimes$.
    Then for $\hd^{d-1}$-almost all $\theta$ the projected configuration $\Rj_\theta(\dsupp)$ does not admit any ghost particle with respect to $\goodtimes$ either.
\end{lem}

In fact, the arguments in this section even show that $\gooddirs$ from \cref{thm:exactrecov_gooddirs} differs from $\alldirs$ at most by a set of finite $(d-2)$-dimensional Hausdorff measure. 
\removed{\Cref{thm:exactReconstruction_hausdorff} is now merely a combination of \cref{thm:exactrecov,thm:exactrecov_gooddirs,thm:exactrecovAllSnapshots,thm:projcoinc,thm:ghost,prop:radon_injective}.}
\added{\label{txt:itemb}\lookUp{\ref{itemb}}\Cref{thm:exactReconstruction_hausdorff} now simply follows by a combination of previous results (note that we use solutions of \hyperref[eq:dim_reduced_general]{\problemTag[0]{\data^\dagger}} instead of the formal model \eqref{eqn:mixedModel} from the introduction, cf.\ \cref{rem:equivalent_formulations}):

\begin{proof}[Proof of \cref{thm:exactReconstruction_hausdorff}]
   Let $(u, \liftVar)$ be a solution of \problemTag[0]{\data^\dagger} and let $\liftVar=\mdirs\prodm\theta \liftVar_\theta$ be a disintegration of $\liftVar$ (which exists by \cref{lem:projected_constraint_decomposition_mixed_model}). Further, let $\goodtimes\subset\alltimes$ be a subset of time steps which fulfills \cref{ass:regularity}. Since for $t\in\goodtimes$ the snapshot $u_t^\dagger$ is assumed to be reconstructed exactly from observations with $\fopstat_t$ by the static problem \ref{eqn:erstat}, we immediately get $u_t=u_t^\dagger$ for all $t\in\goodtimes$ by \cref{thm:exactrecov}.

   Denote by $\dsupp=\supp(\lambda^\dagger)$ the ground truth configuration. \Cref{thm:exactrecov_gooddirs} now yields $\liftVar_\theta=\liftVar_\theta^\dagger$ for $\mdirs$-almost all directions $\theta$ from the set
    \begin{equation*}
        \gooddirs=\{\theta\in\alldirs\mid\text{$\Rj_\theta(\dsupp)$ has neither coincidences nor ghost particles with respect to $\goodtimes$}
        \}.
    \end{equation*}
    Again by \cref{ass:regularity}, the ground truth configuration $\dsupp$ contains neither coincidences nor ghost particles with respect to $\goodtimes$. Thus, combining \cref{thm:projcoinc} and \cref{thm:ghost}, for $\hd^{d-1}$-almost all $\theta\in \sphere^{d-1}$, the projected configuration $\Rj_\theta(\dsupp)$ is free of coincidences and ghost particles as well.
    Since $\mdirs$ is assumed to be absolutely continuous with respect to $\hd^{d-1}$, the set $\gooddirs$ contains $\mdirs$-almost all directions $\theta\in\alldirs$.

    This fact allows us to conclude that $\liftVar_\theta=\liftVar_\theta^\dagger$ holds even for $\mdirs$-almost all $\theta\in\alldirs$.
    Secondly, the Radon transform, which we know to be injective when applied to all directions $\theta\in\alldirs$ by \cref{prop:radon_injective}, remains injective when restricted to the directions in $\gooddirs$.
    Therefore, by \cref{thm:exactrecovAllSnapshots}, we see that $u_t=u_t^\dagger$ for all $t\in\domt$, which completes the proof.
\end{proof}
}

\subsection{Exact reconstruction with finitely many directions}\label{sec:numDirections}

It is clear intuitively that exact reconstruction becomes more difficult the more particles need to be reconstructed.
Thus, given $N$ particles, a natural question is how large the set $\goodtimes$ of exactly reconstructible snapshots needs to be
in order to exactly reconstruct the correct velocity information in any $\liftVar_\theta$.
Likewise one can ask how big a set $\alldirs$ is necessary in order to be able to also exactly reconstruct all snapshots
(not just those that are exactly reconstructible directly from the measurements).
To this end, we briefly generalize a theorem on the X-ray transform due to Haj\'os and R\'enyi \cite{Renyi1952} to the setting of the Radon transform.
Several proofs can be found in the literature for the theorem by Haj\'os and R\'enyi;
a particularly simple one is given in \cite{BianchiLonginetti1990} whose idea we also use here.

\begin{prop}[Determination of finite set from Radon transform]\label{thm:setFromRadon}
Let $\dsupp\subset\R^d$ be finite, containing at most $N$ points,
and let $\alldirs\subset\sphere^{d-1}$ contain $(d-1)N+1$ directions such that any $d$ of them are linearly independent.
Then $\dsupp$ is uniquely determined by its Radon transform $\Rs_\theta \dsupp$ for $\theta\in\alldirs$.
\end{prop}
\begin{proof}
Let the set $\dsupp'$ satisfy $\Rs_\theta \dsupp=\Rs_\theta \dsupp'$ for all $\theta\in\alldirs$.
Assume there exists $x\in \dsupp'\setminus \dsupp$, then for each of the $(d-1)N+1$ directions $\theta\in\alldirs$ there is a point in $\dsupp$ being projected under $\Rs_\theta$ to the same position as $x$.
Since there are at most $N$ points in $\dsupp$, there is at least one $y\in \dsupp$ which is projected onto the same position as $x$ for at least $d$ many projections $\Rs_\theta$.
As a point is uniquely determined by its projections onto $d$ linearly independent coordinate lines, we must have $y=x$, a contradiction.
Knowing now that $\dsupp'\subset\dsupp$ and thus that $\dsupp'$ contains no more than $N$ points, the same argument with the roles of $\dsupp$ and $\dsupp'$ swapped yields $\dsupp\subset\dsupp'$.
\end{proof}

\begin{cor}[Determination of discrete measure from Radon transform]\label{thm:pointMeasureFromRadon}
Let $\nu\in\Mp(\R^d)$ be a nonnegative Radon measure with support in at most $N$ points,
and let $\alldirs\subset\sphere^{d-1}$ contain $(d-1)N+1$ directions such that any $d$ of them are linearly independent.
Then $\nu$ is uniquely determined among all nonnegative Radon measures by its Radon transform $\Rs_\theta\nu$ for $\theta\in\alldirs$.
\end{cor}
\begin{proof}
For a contradiction assume there is a second nonnegative Radon measure $\nu'$ satisfying $\Rs_\theta\nu=\Rs_\theta\nu'$ for all $\theta\in\alldirs$.
Then in particular $\Rs_\theta\supp\nu=\Rs_\theta\supp\nu'$ for all $\theta\in\alldirs$ so that necessarily $\supp\nu=\supp\nu'$.
Now consider the signed measure $\xi=\nu-\nu'\in\M(\R^d)$ and decompose it into its positive and negative part $\xi_+$ and $\xi_-$.
By linearity of $\Rs_\theta$, we have $\Rs_\theta\xi_+=\Rs_\theta\xi_-$ for all $\theta\in\alldirs$.
Therefore, by \cref{thm:setFromRadon} we have $\supp\xi_+=\supp\xi_-$, which is impossible unless $\xi_+=\xi_-=0$.
\end{proof}

An immediate consequence is the following result.
It exploits the fact, already used implicitly before, that the move operator $\mv$ is equivalent to the two-dimensional Radon transform in the sense
\begin{equation*}
\mv_t\nu=\pf{[s\mapsto\sqrt{1+t^2}s]}(\Rs_{(1,t)/\sqrt{1+t^2}}\nu).
\end{equation*}

\begin{cor}[Exact reconstruction]
Let the ground truth particle configuration $\lambda^\dagger$ contain at most $N$ particles,
let $\alldirs\subset\sphere^{d-1}$ contain $(d-1)N+1$ directions such that any $d$ of them are linearly independent,
and let $\mdirs$ be the counting measure on $\alldirs$.
If there are $N+1$ times $t$ at which the static problem \ref{eqn:erstat} reconstructs exactly,
then the solution $(u,\liftVar)$ to \problemTag[0]{\data^\dagger} satisfies $(u,\liftVar)=(u^\dagger,\liftVar^\dagger)$.
\end{cor}
\begin{proof}
Let $\goodtimes$ denote the $N+1$ times, then by \cref{thm:exactrecov} we have $u_t=u^\dagger_t$ for all $t\in\goodtimes$.
Now for any $\theta\in\alldirs$ we know $\mv_t\liftVar_\theta=\Rs_\theta u_t=\Rs_\theta u^\dagger_t=\mv_t\liftVar^\dagger_\theta$
and thus $\Rs_{(1,t)/\sqrt{1+t^2}}\liftVar_\theta=\Rs_{(1,t)/\sqrt{1+t^2}}\liftVar^\dagger_\theta$ for all $N+1$ times $t\in\goodtimes$.
By \cref{thm:pointMeasureFromRadon} this uniquely determines $\liftVar_\theta=\liftVar^\dagger_\theta$ for all $\theta\in\alldirs$.
Now for any $t\in\domt$ we have $\Rs_\theta u_t=\mv_t\liftVar_\theta=\mv_t\liftVar^\dagger_\theta=\Rs_\theta u^\dagger_t$ for all $(d-1)N+1$ many directions $\theta\in\alldirs$,
which again by \cref{thm:pointMeasureFromRadon} uniquely determines $u_t=u_t^\dagger$.
\end{proof}

The above estimates can be improved in various directions.
We only considered the worst case analysis: we even obtain exact reconstruction if the times $\goodtimes$ and the directions $\alldirs$ are configured in the worst possible case
(we did so since the $\goodtimes$ is in general prescribed by the given observations and cannot be chosen).
If one chooses the times $\goodtimes$ and directions $\alldirs$ optimally, exact reconstruction may even be achieved for more points than considered above.
To indicate directions of possible improvement we briefly summarize results from the literature on the two-dimensional Radon transform:
\begin{itemize}
\item
If the directions $\alldirs$ are not a subset of the sides of an affinely deformed regular $N$-gon
(or if the points do not form an affinely deformed regular $N$-gon),
then $N$ points are uniquely determined by $m$ projections already if $m>N-\frac{\sqrt{N+9}-3}2$ \cite[Prop.\,3]{BianchiLonginetti1990}.
\item
If the directions $\alldirs$ are chosen optimally (not just satisfying the previous condition),
then $N$ points are uniquely determined by $m$ projections already if $N\leq\max\{m_0,2^{cm/\log m}\}$ for two constants $m_0,c>0$ \cite[Thm.\,1.1]{MatousekJiriSkovron2008}.
This bound is quite sharp since for $N>6^{m/3}$ one can always find $N$ points that are not uniquely determined by $m$ fixed projections \cite[Thm.\,1.2]{MatousekJiriSkovron2008}.
\item
In higher dimensions $d$, the previous point still holds true (up to changed constants) for the X-ray transform,
while for $m$ arbitrary, non-optimal projection directions spanning $\R^d$
one can asymptotically no longer guarantee reconstruction of arbitrary $N$ points as soon as $N>m^{d+1+\epsilon}$ for $\epsilon>0$ arbitrarily small \cite[Cor.\,4.1]{AlpersLarman2015}.
\end{itemize}

All above results are concerned with worst case point configurations, that is,
they ask for fixed projections such that all possible configurations of $N$ points can be uniquely reconstructed.
In applications it is probably of higher relevance whether \emph{most} point configurations can be reconstructed.

\begin{prop}[Determination of generic set from Radon transform]\label{thm:genericPointReconstructionRadon}
Let $\alldirs\subset\sphere^{d-1}$ contain $d+1$ directions such that any $d$ of them are linearly independent,
and consider any probability distribution of $N$ points in $\R^d$ which is absolutely continuous with respect to the $Nd$-dimensional Lebesgue measure
(for instance Gaussian i.i.d.\ points).
Then almost surely a set $\dsupp$ of $N$ points is uniquely determined among all subsets of $\R^d$ by its Radon transform $\Rs_\theta \dsupp$ for $\theta\in\alldirs$.
\end{prop}
\begin{proof}
Let the directions in $\alldirs$ be numbered as $\theta_1,\ldots,\theta_{d+1}$.
Consider a set $\dsupp=\{x_1,\ldots,x_N\}\subset\R^d$ and denote the hyperplanes through those points normal to $\theta\in\alldirs$ by $H^i_\theta=\{x\in\R^d\,|\,(x-x_i)\cdot\theta=0\}$.
The set $\dsupp$ is uniquely determined if the only mappings $\iota:\alldirs\to\{1,\ldots,N\}$ such that
$H^{\iota(\theta_1)}_{\theta_1}\cap\ldots\cap H^{\iota(\theta_{d+1})}_{\theta_{d+1}}\neq\emptyset$
are the constant mappings $\iota\equiv i$ (in which case
$H^{\iota(\theta_1)}_{\theta_1}\cap\ldots\cap H^{\iota(\theta_{d+1})}_{\theta_{d+1}}=\{x_i\}$).
Thus it remains to show 
$H^{\iota(\theta_1)}_{\theta_1}\cap\ldots\cap H^{\iota(\theta_{d+1})}_{\theta_{d+1}}=\emptyset$
almost surely for all nonconstant $\iota$.
Now the condition
$H^{\iota(\theta_1)}_{\theta_1}\cap\ldots\cap H^{\iota(\theta_{d+1})}_{\theta_{d+1}}\neq\emptyset$
is equivalent to the existence of a solution $x\in\R^d$ to the overdetermined system of equations
\begin{equation*}
Ax=b_\iota
\quad\text{for }
A=(\theta_1\ \cdots\ \theta_{d+1})^T,\quad
b_\iota=(\theta_1\cdot x_{\iota(\theta_1)}\ \cdots\ \theta_{d+1}\cdot x_{\iota(\theta_{d+1})})^T
\end{equation*}
or equivalently (by the Fredholm alternative) to $b_\iota\cdot y=0$ for $y\in\sphere^d$ the vector (unique up to sign) spanning $\ker A^T$.
Thus we require that for any nonconstant $\iota$ the solutions $(x_1,\ldots,x_N)\in\R^{Nd}$ to $\sum_{i=1}^{d+1}y_i\theta_i^Tx_{\iota(\theta_i)}=b_\iota\cdot y=0$
form a set of dimension strictly smaller than $Nd$ and thus a Lebesgue-nullset.
However, the solution set can only have dimension $Nd$ if the equation is trivial, which is never the case:
The coefficient of $x_{\iota(\theta_1)}$ is the sum of at most $d$ terms of the form $y_i\theta_i^T$.
Since all $y_i$ are nonzero and any $d$ directions $\theta_i$ are linearly independent, it is therefore nonzero.
\end{proof}

\begin{cor}[Determination of generic measures from Radon transform]\label{thm:genericPointMeasureFromRadon}
Let $\alldirs\subset\sphere^{d-1}$ contain $d+1$ directions such that any $d$ of them are linearly independent,
and consider any probability distribution of discrete nonnegative measures $\sum_{i=1}^na_i\delta_{x_i}$
such that the induced distribution of their $N$ support points is absolutely continuous with respect to the $Nd$-dimensional Lebesgue measure.
Then almost surely such a measure $\nu$ is uniquely determined among all nonnegative Radon measures by its Radon transform $\Rs_\theta\nu$ for $\theta\in\alldirs$.
\end{cor}
\begin{proof}
Let $\nu$ be a measure consisting of $N$ weighted Dirac measures, then its support is almost surely uniquely determined by \cref{thm:genericPointReconstructionRadon}.
Supposing that its support is indeed uniquely determined, it remains to show that also its weights are uniquely determined.
Thus let $\nu'$ be another admissible measure with $\supp\nu'=\supp\nu$ and $\Rs_\theta\nu'=\Rs_\theta\nu$ for all $\theta\in\alldirs$.
Let $\xi_+$ and $\xi_-$ denote the positive and negative part, respectively, of $\nu'-\nu$, then $\Rs_\theta\xi_+=\Rs_\theta\xi_-$ for all $\theta\in\alldirs$.
Since $\supp\xi_+\cup\supp\xi_-\subset\supp\nu$ also the supports of $\xi_+$ and $\xi_-$ are uniquely determined by their Radon transforms,
which therefore implies $\supp\xi_+=\supp\xi_-$.
This is only possible if $\xi_+=\xi_-=0$.
\end{proof}

As a direct consequence we obtain generic exact reconstruction even with only $d+1$ directions in $\alldirs$, covering in the result of \cref{thm:exactReconstruction_counting}.
\begin{cor}[Exact reconstruction] \label{cor:exact_recon_almost_surely_inite_particles}
Assume the ground truth $\lambda^\dagger$ to contain $N$ particles (some potentially with zero mass) whose location vector in $(\R^d\times\R^d)^N$ is distributed
according to a probability distribution which is absolutely continuous with respect to the $2Nd$-dimensional Lebesgue measure.
Let $\alldirs\subset\sphere^{d-1}$ contain $d+1$ directions such that any $d$ of them are linearly independent,
and let $\mdirs$ be the counting measure on $\alldirs$.
If there are three times $t$ at which the static problem \ref{eqn:erstat} reconstructs exactly,
then almost surely the solution $(u,\liftVar)$ to \problemTag[0]{\data^\dagger} satisfies $(u,\liftVar)=(u^\dagger,\liftVar^\dagger)$.
\end{cor}
\begin{proof}
Let $\goodtimes$ denote the three times, then by \cref{thm:exactrecov} we have $u_t=u^\dagger_t$ for all $t\in\goodtimes$.
Now for any $\theta\in\alldirs$ we know $\mv_t\liftVar_\theta=\Rs_\theta u_t=\Rs_\theta u^\dagger_t=\mv_t\liftVar^\dagger_\theta$
and thus $\Rs_{(1,t)/\sqrt{1+t^2}}\liftVar_\theta=\Rs_{(1,t)/\sqrt{1+t^2}}\liftVar^\dagger_\theta$ for all three times $t\in\goodtimes$.
By \cref{thm:genericPointMeasureFromRadon} this almost surely uniquely determines $\liftVar_\theta=\liftVar^\dagger_\theta$ for all finitely many $\theta\in\alldirs$.
Now for any $t\in\domt$ we have $\Rs_\theta u_t=\mv_t\liftVar_\theta=\mv_t\liftVar^\dagger_\theta=\Rs_\theta u^\dagger_t$ for all $d+1$ many directions $\theta\in\alldirs$,
which again by \cref{thm:genericPointMeasureFromRadon} almost surely uniquely determines $u_t=u_t^\dagger$ for each fixed $t\in\domt$.
That the fixed $t\in\domt$ may even be replaced by all $t\in\domt$ follows from reinspection of the proof of \cref{thm:genericPointReconstructionRadon},
which can be extended to show that generically all $u_t$ (or rather their supports) can simultaneously be uniquely reconstructed.
\end{proof}

\begin{rem} The results in this section were obtained independent of the ones in \cref{sec:exact_recon_dim_reduced}, in particular not relying on assumptions on ghost particles and coincidences. Alternatively, it would have also been possible to investigate the set $\gooddirs$ of \cref{thm:exactrecov_gooddirs} also for the case of finitely many directions, and show that it is sufficiently large to ensure exact reconstruction (almost surely). While this would have provided a common framework for $\mdirs$ being either the Hausdorff or the counting measure, a direct treatment of the latter turned out to be more concise and, for this reason, was preferred here.
\end{rem}

\section{Reconstruction from noisy data}\label{sec:noisyReconstruction}

While the previous \namecref{sec:exactRecovery} dealt with noisefree observations $\data^\dagger$ and a reconstruction via \problemTag[0]{\data^\dagger},
in this section we will consider noisy observations $\data^\delta$ and a reconstruction via \problemTag[\alpha]{\data^\delta} with positive $\alpha$.
We aim to derive error estimates for our reconstruction.
To this end we will first introduce a measure to quantify the error in \cref{sec:unbalancedTransport},
then derive abstract error estimates for general inverse problems of Radon measures in \cref{sec:BregmanEstimates}.
\removed{,and finally apply these results to our models \PInoisyProblemTag[\alpha]{\data^\delta} and \problemTag[\alpha]{\data^\delta} in \cref{sec:dualVariables,sec:dualConstructionII}.}
\added{\lookUp{\ref{item3}}In order to finally apply these results to our model \problemTag[\alpha]{\data^\delta}, we will first consider an (in this case simpler) product-topology version of \problemTag[\alpha]{\data^\delta} in \cref{sec:dualVariables} and then extend the results to \problemTag[\alpha]{\data^\delta} in \cref{sec:dualConstructionII}.}

\subsection{Unbalanced optimal transport as error measure}\label{sec:unbalancedTransport}
In the setting of particle reconstruction, the positions and masses of the reconstructed particles will slightly deviate from the ground truth due to noise in the measurement.
A quantification of this deviation thus has to account for both sources of error simultaneously, the mass relocation and the mass change.
A clear separation of these two is not meaningful since a Dirac mass in the wrong place
can both be explained by an incorrect positioning of a ground truth Dirac mass or simply by some added mass to the ground truth (which just happened to be placed in that position).
Therefore we will quantify the reconstruction error in so-called unbalanced Wasserstein divergences (variants of optimal transport distances that allow for mass changes).
The following unbalanced Wasserstein divergences turn out to be the most natural in our setting.

\begin{defn}[Unbalanced Wasserstein divergence]\label{def:Wasserstein}
Let $\nu_1,\nu_2\in\Mp(\R^d)$ be nonnegative Radon measures on $\R^d$.
The \emph{Wasserstein-$p$ distance} between $\nu_1$ and $\nu_2$ is defined for $p\geq1$ as
\begin{multline*}
W_p(\nu_1,\nu_2)=\inf\left\{\left(\int_{\R^d\times\R^d}\dist(x,y)^p\,\d\pi(x,y)\right)^{1/p}\,\middle|\,\pi\in\Mp(\R^d\times\R^d),\right.\\
\left.\vphantom{\left(\int_{\R^d}\right)^{1/p}}\pf{[(x,y)\mapsto x]}\pi=\nu_1,\pf{[(x,y)\mapsto y]}\pi=\nu_2\right\}.
\end{multline*}
If $\nu_1$ and $\nu_2$ have unequal mass $\nu_1(\R^d)\neq\nu_2(\R^d)$, the above infimum is over the empty set and thus infinite.
For fixed parameter $R>0$ we define the \emph{unbalanced Wasserstein-$p$ divergence} between $\nu_1$ and $\nu_2$ as
\begin{equation*}
\unbalancedWasserstein[p]{R}(\nu_1,\nu_2)=\inf\left\{W_p^p(\nu,\nu_2)+\tfrac12R^p\|\nu_1-\nu\|_{\M}\,\middle|\,\nu\in\Mp(\R^d)\right\}.
\end{equation*}
\end{defn}

\begin{rem}[Interpretation of unbalanced Wasserstein divergence]
The Wasserstein distance is well-known to be a metric on the space of nonnegative Radon measures with fixed mass, supported on a fixed compact domain,
and it metrizes weak-$\ast$ convergence on that space.
The quantity $W_p^p(\nu_1,\nu_2)$ measures the cost for transporting the material from $\nu_1$ such that after the transport it is distributed according to $\nu_2$.
The unbalanced Wasserstein divergence $\unbalancedWasserstein[p]{R}(\nu_1,\nu_2)$ thus measures the cost for first changing the mass of $\nu_1$ to some intermediate measure $\nu$
and then transporting that new mass distribution to $\nu_2$.
Thus it can be used to quantify weak-$\ast$ convergence of nonnegative measures of not necessarily the same mass.
The weight $R^p$ of the mass change cost governs up to which distance a mass transport is less costly than removing the mass in the initial position and regrowing it in the destination;
this distance is exactly $R$.
\end{rem}

\begin{rem}[Symmetry of unbalanced Wasserstein divergence]
Obviously, an alternative definition would have been based on changing the mass of $\nu_2$ or of both $\nu_1$ and $\nu_2$.
It can readily be seen that those are equivalent:
Indeed, any good candidate intermediate measure $\nu$ should be absolutely continuous with respect to and no larger than $\nu_1+\nu_2$.
Then it is easy to check that $\tilde\nu=\nu_2+\nu_1-\nu\in\Mp(\R^d)$ satisfies
$W_p^p(\nu,\nu_2)=W_p^p(\nu_1,\tilde\nu)$ and $\|\nu_1-\nu\|_{\M}=\|\nu_2-\tilde\nu\|_{\M}$.
\end{rem}

The unbalanced Wasserstein divergence is nonnegative, and $\unbalancedWasserstein[p]{R}(\nu_1,\nu_2)=0$ implies $\nu_1=\nu_2$.
Furthermore, it is symmetric by the previous remark.
However, it does not satisfy the triangle inequality (unless $p=1$).
Therefore it is not a metric, but metric variants can easily be defined.

A bound on the unbalanced Wasserstein divergence between a discrete measure $\nu_1=\sum_{i=1}^Nm_i\delta_{x_i}\in\Mp(\R^d)$ and $\nu_2\in\Mp(\R^d)$
is readily obtained from estimates of the mass distribution within and outside balls around $x_1,\ldots,x_N$\added{\label{txt:itemeProofLogic6}\lookUp{\ref{iteme}}, which is the content of the following theorem}.
Such estimates were first derived by Cand\`es and Fernandez-Granda to analyse spike recovery from noisy data,
and in the subsequent sections we will show how they can also be derived for our dimension-reduced dynamic particle tracking setting.

\begin{thm}[Unbalanced Wasserstein divergence from mass estimates]\label{thm:unbalancedWasserstein} 
Let $p\geq1$, $R>0$, and let $\nu_1=\sum_{i=1}^Nm_i\delta_{x_i}\in\Mp(\R^d)$, with the distance between the $x_i$ being at least $2R$, and $\nu_2\in\Mp(\R^d)$ satisfy
\begin{align*}
|\nu_2|(\R^d\setminus B_R(\{x_1,\ldots,x_N\}))
&\leq A,\\
\sum_{i=1}^N|(\nu_1-\nu_2)(B_R(\{x_i\}))|
&\leq B,\\
\sum_{i=1}^N\int_{B_R(\{x_i\})}\dist(x,x_i)^p\wrt\nu_2(x)
&\leq C_p,
\end{align*}
where $B_R(\dsupp)$ denotes the open $R$-neighbourhood of the set $\dsupp\subset\R^d$.
Then we have
\begin{align*}
\unbalancedWasserstein[p]{R}(\nu_1,\nu_2)
&\leq \tfrac12R^p(A+B)+C_p,\\
\unbalancedWasserstein[q]{R}(\nu_1,\nu_2)
&\leq \tfrac12R^q(A+B)+R^{q-p}C_p\text{ if }q>p,\\
\unbalancedWasserstein[q]{R}(\nu_1,\nu_2)
&\leq \tfrac12R^q(A+B)+\tfrac p{p-q}\left(\tfrac{p-q}q\right)^{q/p}C_p^{q/p}\|\nu_2\|_\M^{1-q/p}\text{ if }q<p.
\end{align*}
\end{thm}
\begin{proof}
In the definition of the unbalanced Wasserstein divergence simply take
$\nu=\nu_2\restr(\R^d\setminus B_R(\{x_1,\ldots,x_N\}))+\sum_{i=1}^N\nu_2(B_R(\{x_i\}))\delta_{x_i}$,
then the first estimate immediately follows.
The second one follows, using the same $\nu$, from $\dist(x,x_i)^q\leq R^{q-p}\dist(x,x_i)^p$ on $B_R(\{x_i\})$ for any $q>p$.
Finally, if $q<p$ we again pick the same $\nu$ and estimate
\begin{multline*}
\int_{B_R(\{x_i\})}\dist(x,x_i)^q\wrt\nu_2(x)
\leq\int_{B_R(\{x_i\})\setminus B_r(\{x_i\})}\dist(x,x_i)^q\wrt\nu_2(x)
+\int_{B_r(\{x_i\})}\dist(x,x_i)^q\wrt\nu_2(x)\\
\leq r^{q-p}\int_{B_R(\{x_i\})\setminus B_r(\{x_i\})}\dist(x,x_i)^p\wrt\nu_2(x)
+r^q\|\nu_2\|_{\M}
\leq r^{q-p}C_p+r^q\|\nu_2\|_\M
\end{multline*}
for any $r>0$. Optimizing in $r$ yields $r=(\frac{p-q}q\frac{C_p}{\|\nu_2\|_\M})^{1/p}$ and thus
$$\int_{B_R(\{x_i\})}\dist(x,x_i)^q\wrt\nu_2(x)
\leq\frac p{p-q}\left(\frac{p-q}q\right)^{q/p}C_p^{q/p}\|\nu_2\|_\M^{1-q/p}$$
so that the third estimate follows.
\end{proof}

\subsection{General strategy for Bregman distance estimates}\label{sec:BregmanEstimates}
To derive noise-dependent convergence rates or error estimates for our reconstruction in case of noisy data, we follow a classical strategy in inverse problems that is based on duality.
In this section we recapitulate the corresponding procedure in an abstract setting.
It reduces the derivation of error estimates to the construction of appropriate dual variables.
The explicit construction of such dual variables for our specific setting will be performed in the subsequent section.

The abstract setting is as follows.
Let $X,Y$ be topological vector spaces, $K:X\to Y$ a continuous linear observation operator,
$\regularizer:X\to\R$ a convex regularization energy,
and $\fidelity{\data}:Y\to\R$ a nonnegative convex fidelity term checking the fidelity of its argument to a measurement $\data\in Y$.
We will assume $\fidelity{\data}(\data)=0$ and $\fidelity{\data}>0$ elsewhere
(otherwise the data fidelity would not be able to discriminate between multiple, apparently equally well fitting reconstructions).
We will denote by $\groundTruth\in X$ the ground truth configuration with noise-free measurement $\noiseFreeData=K\groundTruth\in Y$,
while the noisy measurement is denoted $\noisyData$, the amount of noise $\delta$ being quantified by the condition
\begin{equation*}
\fidelity{\noisyData}(\noiseFreeData)\leq\delta.
\end{equation*}
In our application, the spaces $X,Y$ will be given by products of spaces of Radon measures, hence we assume $X$ and $Y$ to actually be dual spaces to topological vector spaces
(otherwise the roles of primal and dual spaces simply swap in the following).
For simplicity we denote the predual spaces by $X^*$ and $Y^*$.
We now consider the problem to reconstruct an approximation to $\groundTruth$ by solving
\begin{equation}\label{eqn:reconstructionProblem}
\min_{\approximation\in X}\energy_{\alpha,\noisyData}(\approximation)
\quad\text{for }\energy_{\alpha,\noisyData}(\approximation)=\regularizer(\approximation)+\frac1\alpha\fidelity{\noisyData}(K\approximation),
\end{equation}
where $\alpha,\delta>0$ corresponds to the noisy measurement case and $\alpha=\delta=0$ %
to the noise-free setting
(with the notational convention $\tfrac10\fidelity{\data}(\tilde\data)=0$ if $\tilde\data=\data$ and $\tfrac10\fidelity{\data}(\tilde\data)=\infty$ otherwise).

Generally, if $\groundTruth$ is the minimizer of a smooth energy $\energy$ and $\approximation$ an approximation,
then the deviation of $\approximation$ from $\groundTruth$ can be estimated from the difference $\energy(\approximation)-\energy(\groundTruth)$ and lower bounds on the Hessian of $\energy$.
In the case of nonsmooth convex energies the role of Hessian-based distance estimates is naturally replaced by Bregman distances for $\energy$,
where the Bregman distance of $a$ to $b$ with respect to a convex energy $\energy$ and a subgradient element $w\in\partial\energy(b)$ is
\begin{equation*}
\BregmanDistance\energy w(a,b)=\energy(a)-\energy(b)-\langle w,a-b\rangle\geq0.
\end{equation*}
However, since in our setting the energy difference $\energy_{0,\noiseFreeData}(\approximation)-\energy_{0,\noiseFreeData}(\groundTruth)$ is infinite if $K\approximation\neq\noiseFreeData$,
the resulting Bregman distance estimates would not be very useful.
Therefore, the fidelity term needs to be dualized first, that is, for some fixed $\optimalDual\in Y^*$, instead of $\energy_{0,\noiseFreeData}$ we consider
\begin{equation*}
\regularizer(\cdot)+\langle K^*\optimalDual,\cdot\rangle-(\tfrac10\fidelity{\noiseFreeData})^*(\optimalDual)
\end{equation*}
(which by weak duality is never larger than $\energy_{0,\noiseFreeData}$).
Above, $K^*:Y^*\to X^*$ is the operator having $K$ as dual,
and $(\tfrac10\fidelity{\noiseFreeData})^*(\optimalDual)=\sup_{y\in Y}\langle y,\optimalDual\rangle-\tfrac10\fidelity{\noiseFreeData}(y)= \langle \noiseFreeData,\optimalDual \rangle$  represents the (predual) Legendre--Fenchel conjugate.
This then leads to the following estimates (as for instance derived for quadratic $\fidelity{\noiseFreeData}$ in \cite[Thm.\,2]{BurgerOsher2004}).

\begin{thm}[Bregman distance estimate for noisy reconstruction]\label{thm:BregmanEstimates}
Let $\optimalDual\in Y^*$ satisfy the source condition $-K^*\optimalDual\in\partial\regularizer(\groundTruth)$.
Then a minimizer $\approximation$ of \eqref{eqn:reconstructionProblem} satisfies
\begin{align*}
\BregmanDistance{\regularizer}{-K^*\optimalDual}(\approximation,\groundTruth)
&\leq\left(3\delta+\fidelity{\noisyData}^*(2\alpha\optimalDual)+\fidelity{\noisyData}^*(-2\alpha\optimalDual)\right)/(2\alpha),\\
\fidelity{\noisyData}(K\approximation)
&\leq\left(3\delta+\fidelity{\noisyData}^*(2\alpha\optimalDual)+\fidelity{\noisyData}^*(-2\alpha\optimalDual)\right),\\
\langle K^*w,\approximation-\groundTruth\rangle
&\leq\left(4\delta+\fidelity{\noisyData}^*(2\alpha\optimalDual)+\fidelity{\noisyData}^*(-2\alpha\optimalDual)+\fidelity{\noisyData}^*(2\alpha w)+\fidelity{\noisyData}^*(-2\alpha w)\right)/(2\alpha)
\quad\text{for all }w\in Y^*.
\end{align*}
\end{thm}
\begin{proof}
First observe that
\begin{multline*}
\left[\regularizer(\approximation)+\langle K^*\optimalDual,\approximation\rangle-(\tfrac10\fidelity{\noiseFreeData})^*(\optimalDual)\right]
-\left[\regularizer(\groundTruth)+\langle K^*\optimalDual,\groundTruth\rangle-(\tfrac10\fidelity{\noiseFreeData})^*(\optimalDual)\right]\\
=\regularizer(\approximation)-\regularizer(\groundTruth)-\langle-K^*\optimalDual,\approximation-\groundTruth\rangle
=\BregmanDistance{\regularizer}{-K^*\optimalDual}(\approximation,\groundTruth).
\end{multline*}
Now the energy difference on the left-hand side can be estimated from above exploiting the optimality of $\approximation$,
\begin{equation*}
\regularizer(\approximation)+\frac1\alpha\fidelity{\noisyData}(K\approximation)
=\energy_{\alpha,\noisyData}(\approximation)
\leq\energy_{\alpha,\noisyData}(\groundTruth)
=\regularizer(\groundTruth)+\frac1\alpha\fidelity{\noisyData}(\noiseFreeData)
\leq\regularizer(\groundTruth)+\frac\delta\alpha.
\end{equation*}
Furthermore, we use Fenchel's inequality to obtain
\begin{equation*}
\langle K^*w,\approximation-\groundTruth\rangle
=\frac{\langle2\alpha w,K\approximation\rangle+\langle-2\alpha w,K\groundTruth\rangle}{2\alpha}
\leq\frac{\fidelity{\noisyData}(K\approximation)+\fidelity{\noisyData}^*(2\alpha w)+\fidelity{\noisyData}(K\groundTruth)+\fidelity{\noisyData}^*(-2\alpha w)}{2\alpha},
\end{equation*}
which already proves the third statement in case the second holds.
Using these two inequalities we obtain
\begin{multline*}
\left[\regularizer(\approximation)+\langle K^*\optimalDual,\approximation\rangle-(\tfrac10\fidelity{\noiseFreeData})^*(\optimalDual)\right]
-\left[\regularizer(\groundTruth)+\langle K^*\optimalDual,\groundTruth\rangle-(\tfrac10\fidelity{\noiseFreeData})^*(\optimalDual)\right]\\
\leq\frac\delta\alpha-\frac1\alpha\fidelity{\noisyData}(K\approximation)+\langle\optimalDual,K(\approximation-\groundTruth)\rangle
\leq\frac1{2\alpha}\left(3\delta+\fidelity{\noisyData}^*(2\alpha\optimalDual)+\fidelity{\noisyData}^*(-2\alpha\optimalDual)\right)-\frac1{2\alpha}\fidelity{\noisyData}(K\approximation).
\end{multline*}
Identifying the left-hand side with $\BregmanDistance{\regularizer}{-K^*\optimalDual}(\approximation,\groundTruth)$
and adding $\frac1{2\alpha}\fidelity{\noisyData}(K\approximation)$ on both sides, the first two statements immediately follow.
\end{proof}

\begin{rem}[Interpretation of source condition]
The source condition $-K^*\optimalDual\in\partial\regularizer(\groundTruth)$ is one of the two necessary and sufficient primal-dual optimality conditions for the optimization problem;
the other one is $K\groundTruth\in\partial(\tfrac10\fidelity{\noiseFreeData})^*(\optimalDual)=\{\noiseFreeData\}$ and thus automatically satisfied for the ground truth.
\end{rem}

\begin{rem}[Rates from estimates]
Note that due to the strict convexity of $\fidelity{\noisyData}$ in $\noisyData$, the Legendre--Fenchel conjugate $\fidelity{\noisyData}^*$ is differentiable in $0$
so that $(\fidelity{\noisyData}^*(2\alpha w)+\fidelity{\noisyData}^*(-2\alpha w))/\alpha$ converges to zero as $\alpha\to0$.
Thus, by choosing $\alpha$ as the minimizer of $(\delta+\fidelity{\noisyData}^*(2\alpha w)+\fidelity{\noisyData}^*(-2\alpha w))/\alpha$,
the right-hand sides in \cref{thm:BregmanEstimates} become functions of $\delta$ that decrease to zero as $\delta\to0$.
\end{rem}

If the regularization contains the Radon norm of a measure,
the above abstract estimates of Bregman distances can be turned into unbalanced Wasserstein estimates of how much $\approximation$ deviates from $\groundTruth$.
In essence, this is the basic idea underlying the classical stability analysis of superresolution or spike reconstruction as introduced by Cand\`es and Fernandez-Granda in \cite{Candes2014}.

\begin{thm}[Mass distribution from Bregman distance]\label{thm:massFromBregman}
Assume $\groundTruth\in\M(\R^d)$ to be a discrete measure with support in $\{x_1,\ldots,x_N\}\subset\R^d$,
and let $v^\dagger\in\partial\|\cdot\|_{\M}(\groundTruth)\subset C(\R^d)$, that is, $|v^\dagger|\leq1$ and $v^\dagger(x_i)=\sgn(\groundTruth(\{x_i\}))$.
If $v^\dagger$ satisfies
\begin{equation*}
|v^\dagger(x)|\leq1-\kappa\min\{R,\dist(x,\{x_1,\ldots,x_N\})\}^2
\end{equation*}
for some $\kappa,R>0$ with $2R$ smaller than the minimum distance between the $x_i$, then for any $\approximation\in\M(\R^d)$ we have
\begin{align*}
|\approximation|(\R^d\setminus B_R(\{x_1,\ldots,x_N\}))
&\leq\frac1{\kappa R^2}\BregmanDistance{\|\cdot\|_\M}{v^\dagger}(\approximation,\groundTruth),\\
\sum_{i=1}^N\int_{B_R(\{x_i\})}\dist(x,x_i)^2\wrt|\approximation|(x)
&\leq\frac1{\kappa}\BregmanDistance{\|\cdot\|_\M}{v^\dagger}(\approximation,\groundTruth),\quad i=1,\ldots,N.
\end{align*}
Furthermore, let $v\in\partial\|\cdot\|_{\M}(y)\subset C(\R^d)$ for $y=\sum_{i=1}^N(\approximation-\groundTruth)(B_R(\{x_i\}))\delta_{x_i}$, thus $v(x_i)=\sgn(y(\{x_i\}))$ for $i=1,\ldots,N$. If for some $\mu>0$ the function $v$ satisfies
\begin{equation*}
v(x_i)v(x)\geq1-\mu\dist(x,x_i)^2\quad\text{for all }x\in B_R(\{x_i\}),i=1,\ldots,N,
\end{equation*}
then we additionally have
\begin{equation*}
\sum_{i=1}^N|(\approximation-\groundTruth)(B_R(\{x_i\}))|
\leq\frac{1+\mu R^2}{\kappa R^2}\BregmanDistance{\|\cdot\|_\M}{v^\dagger}(\approximation,\groundTruth)+\langle v,\approximation-\groundTruth\rangle.
\end{equation*}
\end{thm}
\begin{proof}
Abbreviate $B^c=\R^d\setminus B_R(\{x_1,\ldots,x_N\})$ and let $\RadonNikodym{\approximation}{|\approximation|}$ denote the Radon--Nikodym derivative of $\approximation$ with respect to its total variation measure, then the first two inequalities follow from
\begin{multline*}
\kappa R^2|\approximation|(B^c)
+\kappa\sum_{i=1}^N\int_{B_R(\{x_i\})}\dist(x,x_i)^2\wrt|\approximation|(x)
\leq\int_{B^c}\RadonNikodym{\approximation}{|\approximation|}-v^\dagger\wrt\approximation(x)
+\sum_{i=1}^N\int_{B_R(\{x_i\})}\RadonNikodym{\approximation}{|\approximation|}-v^\dagger\wrt\approximation(x)\\
=\int_{\R^d}\RadonNikodym{\approximation}{|\approximation|}-v^\dagger\wrt\approximation(x)
=\|\approximation\|_{\M}-\langle v^\dagger,\approximation\rangle
=\|\approximation\|_{\M}-\|\groundTruth\|_{\M}-\langle v^\dagger,\approximation-\groundTruth\rangle
=\BregmanDistance{\|\cdot\|_\M}{v^\dagger}(\approximation,\groundTruth).
\end{multline*}
The third inequality is obtained as follows.
We have
\begin{equation*}
\sum_{i=1}^N|(\approximation-\groundTruth)(B_R(\{x_i\}))|
=\int_{B_R(\{x_1,\ldots,x_N\})}v\wrt(\approximation-\groundTruth)+\sum_{i=1}^N\int_{B_R(\{x_i\})}v(x_i)-v\wrt(\approximation-\groundTruth),
\end{equation*}
where the first summand can be estimated as
\begin{multline*}
\int_{B_R(\{x_1,\ldots,x_N\})}v\wrt(\approximation-\groundTruth)
=\langle v,\approximation-\groundTruth\rangle-\int_{B^c}v\wrt(\approximation-\groundTruth)\\
\leq\langle v,\approximation-\groundTruth\rangle+|\approximation-\groundTruth|(B^c)
\leq\langle v,\approximation-\groundTruth\rangle+\frac1{\kappa R^2}\BregmanDistance{\|\cdot\|_\M}{v^\dagger}(\approximation,\groundTruth),
\end{multline*}
while the second term is estimated from above via
\begin{equation*}
\sum_{i=1}^N\int_{B_R(\{x_i\})}v(x_i)-v\wrt(\approximation-\groundTruth)
\leq\sum_{i=1}^N\int_{B_R(\{x_i\})}\mu\dist(x,x_i)^2\wrt|\approximation|(x)
\leq\frac\mu\kappa\BregmanDistance{\|\cdot\|_\M}{v^\dagger}(\approximation,\groundTruth).
\qedhere
\end{equation*}
\end{proof}

We finally combine \cref{thm:unbalancedWasserstein,thm:massFromBregman,thm:BregmanEstimates}:
Via \cref{thm:unbalancedWasserstein} an unbalanced Wasserstein error estimate can be expressed with the help of terms
that can be estimated by Bregman distances via \cref{thm:massFromBregman},
which in turn  are estimated via \cref{thm:BregmanEstimates}.
In the case $\regularizer(\approximation)=\|\approximation\|_{\M}$ the combination is immediate,
however, to allow for a little more generality (which is required for the application to our setting) we shall assume
\begin{equation*}
\regularizer(\approximation)=\|\approximation_1\|_{\M}+\regularizer_2(\approximation_2),
\quad\text{where }\approximation=(\approximation_1,\approximation_2).
\end{equation*}

\begin{cor}[Unbalanced Wasserstein estimate for noisy reconstruction]\label{thm:noisyReconstructionEstimate}
Consider the choice $\regularizer(\approximation)=\regularizer(\approximation_1,\approximation_2)=\|\approximation_1\|_{\M}+\regularizer_2(\approximation_2)$.
Let $\groundTruth_1\in\Mp(\R^d)$ have support in $\{x_1,\ldots,x_N\}\subset\R^d$,
let $\approximation$ with $\approximation_1\in\Mp(\R^d)$ be the approximation to $\groundTruth$ obtained from \eqref{eqn:reconstructionProblem},
and let $\optimalDual,w\in Y^*$ satisfy
\begin{align*}
-K^*\optimalDual=-((K^*\optimalDual)_1,(K^*\optimalDual)_2)&\in\partial\regularizer(\groundTruth),\\
-(K^*w)_1&\in\partial\|\cdot\|_{\M}\left(\sum_{i=1}^N(\approximation_1-\groundTruth_1)(B_{R}(\{x_i\}))\delta_{x_i}\right),\\
|(K^*\optimalDual)_1(x)|&\leq1-\kappa\min\{R,\dist(x,\{x_1,\ldots,x_N\})\}^2,\\
(K^*w)_1(x_i)(K^*w)_1(x)&\geq1-\mu\dist(x,x_i)^2\quad\text{for all }x\in B_R(\{x_i\}),i=1,\ldots,N,
\end{align*}
for some $\mu,\kappa,R>0$ with $2R$ smaller than the minimum distance between the $x_i$. Then we have
\begin{align*}
|\approximation_1|(\R^d\setminus B_R(\{x_1,\ldots,x_N\}))
&\leq\frac1{2\alpha\kappa R^2}\left(3\delta+\fidelity{\noisyData}^*(2\alpha\optimalDual)+\fidelity{\noisyData}^*(-2\alpha\optimalDual)\right),\\
\sum_{i=1}^N\int_{B_R(\{x_i\})}\dist(x,x_i)^2\wrt|\approximation_1|(x)
&\leq\frac1{2\alpha\kappa}\left(3\delta+\fidelity{\noisyData}^*(2\alpha\optimalDual)+\fidelity{\noisyData}^*(-2\alpha\optimalDual)\right),\\
\sum_{i=1}^N|(\approximation_1-\groundTruth_1)(B_R(\{x_i\}))|
&\leq\frac{1+\mu R^2}{2\alpha\kappa R^2}\left(3\delta+\fidelity{\noisyData}^*(2\alpha\optimalDual)+\fidelity{\noisyData}^*(-2\alpha\optimalDual)\right)
+\langle(K^*w)_2,\approximation_2-\groundTruth_2\rangle\\
&\quad+\frac1{2\alpha}\left(4\delta+\fidelity{\noisyData}^*(2\alpha\optimalDual)+\fidelity{\noisyData}^*(-2\alpha\optimalDual)+\fidelity{\noisyData}^*(2\alpha w)+\fidelity{\noisyData}^*(-2\alpha w)\right).
\end{align*}
In particular, by \cref{thm:unbalancedWasserstein} this implies
\begin{multline*}
\unbalancedWasserstein[2]{R}(\approximation_1,\groundTruth_1)
\leq\tfrac12R^2\langle(K^*w)_2,\approximation_2-\groundTruth_2\rangle\\
+\frac{12+7\max\{\kappa,\mu\}R^2}{4\kappa\alpha}\left(\delta+\fidelity{\noisyData}^*(2\alpha\optimalDual)+\fidelity{\noisyData}^*(-2\alpha\optimalDual)+\fidelity{\noisyData}^*(2\alpha w)+\fidelity{\noisyData}^*(-2\alpha w)\right).
\end{multline*}
\end{cor}
\begin{proof}
This immediately follows from choosing $v^\dagger=(-K^*\optimalDual)_1$, $v=(-K^*w)_1$ in \cref{thm:massFromBregman}
and noting that $\BregmanDistance{\|\cdot\|_{\M}}{v^\dagger}(\approximation_1,\groundTruth_1)\leq\BregmanDistance{\regularizer}{-K^*\optimalDual}(\approximation,\groundTruth)$
as well as $\langle v,\approximation_1-\groundTruth_1\rangle=\langle-K^*w,\approximation-\groundTruth\rangle+\langle(K^*w)_2,\approximation_2-\groundTruth_2\rangle$ so that \cref{thm:BregmanEstimates} can be inserted.
\end{proof}

This way, the problem of proving the error estimate is reduced to the problem of constructing $\optimalDual$ and $w$.
Note for instance that for $\regularizer_2=0$ the condition $-(K^*\optimalDual)_2\in\partial\regularizer_2(\groundTruth_2)$ implies $(K^*\optimalDual)_2=0$
so that the corresponding term in the estimate of \cref{thm:noisyReconstructionEstimate} vanishes.

\begin{rem}[Adaptation for nonnegative measures and compact domains]\label{rem:nonnegativeMeasures}
It is straightforward to check that all arguments stay the same when replacing $\R^d$ with a compact domain $A$ and when taking the slightly different regularizer choice
\begin{equation*}
\regularizer(\approximation)=\regularizer(\approximation_1,\approximation_2)=\|\approximation_1\|_++\regularizer_2(\approximation_2)
\quad\text{with }\|\approximation_1\|_+=\begin{cases}\|\approximation_1\|_{\M}&\text{if }\approximation_1\in\Mp(A),\\\infty&\text{else.}\end{cases}
\end{equation*}
For this choice one obtains the estimates of \cref{thm:noisyReconstructionEstimate} under the slightly weaker conditions
\begin{align*}
-K^*\optimalDual=-((K^*\optimalDual)_1,(K^*\optimalDual)_2)&\in\partial\regularizer(\groundTruth),\\
-(K^*w)_1&\in\partial\|\cdot\|_{\M}\left(\sum_{i=1}^N(\approximation_1-\groundTruth_1)(B_{R}(\{x_i\}))\delta_{x_i}\right),\\
-(K^*\optimalDual)_1(x)&\leq1-\kappa\min\{R,\dist(x,\{x_1,\ldots,x_N\})\}^2,\\
(K^*w)_1(x_i)(K^*w)_1(x)&\geq1-\mu\dist(x,x_i)^2\quad\text{for all }x\in B_R(\{x_i\}),i=1,\ldots,N.
\end{align*}
\end{rem}

\removed{\subsection{Construction of dual variables for \texorpdfstring{\ref{eq:dim_reduced_product_only}}{PI-P\_alpha(f)}}\label{sec:dualVariables}}
\added{\subsection{Construction of dual variables for a product-topology setting}\label{sec:dualVariables}}

In this and the next \namecref{sec:dualConstructionII} we prove error estimates for the particle reconstruction from noisy observations.
As in \cref{sec:exactRecovery} we denote the ground truth particle configuration by $\lambda^\dagger\in\Mp(\domdyn)$
with $\lambda^\dagger=\sum_{i=1}^Nm_i\delta_{(x_i,v_i)}$,
with snapshots $u^\dagger_t=\mv^d_t\lambda^\dagger$, $t\in\alltimes$,
and with position-velocity projections $\liftVar^\dagger=\Rj\lambda^\dagger$ or $\liftVar^\dagger_\theta=\Rj_\theta\lambda^\dagger$, $\theta\in\alldirs$.
The noisefree observations are denoted $\data^\dagger_t=\fopstat_t\mv^d_t\lambda^\dagger$, the noisy ones $\data^\delta_t$, $t\in\alltimes$, with noise strength
\begin{equation}\phantomsection\label{eqn:noiseStrength}
\delta=\tfrac12\sum_{t\in\alltimes}\|\data_t^\delta-\data_t^\dagger\|_H^2.
\end{equation}
The main result of these \namecrefs{sec:dualVariables} is the below stated more quantitative version of \cref{thm:reconstructionErrorEstimate}.
Its error estimates depend on how close we are to a coincidence or a ghost particle.
This proximity to a degenerate configuration is quantified via the notion of $\Delta$-coincidences and $\Delta$-ghost particles
that will be introduced in \cref{def:DeltaGhost,def:DeltaGhostII}\added{\label{txt:itemeProofLogic7}\lookUp{\ref{iteme}},
in which $\Delta$ can be thought of as the smallest necessary shift of the particles to produce a coincidence or a ghost particle}.
Vanishing $\Delta$ indicates a degenerate configuration, and in this case the error estimates become void
(since the error is in those cases measured in the trivial unbalanced Wasserstein divergence $\unbalancedWasserstein[2]{0}\equiv0$).
Finally note that those snapshots that cannot directly be reconstructed from measurements will instead be reconstructed from the position-velocity projections $\liftVar_\theta$
with $\theta$ from a certain set of \emph{good directions} $\gooddirs\subset\alldirs$.

\added{
\lookUp{\ref{item3}}We start in this subsection with a derivation of error estimates for a product-topology version of \problemTag[\alpha]{\data} given as
\begin{equation*} \label{eq:dim_reduced_product_only_noisy} \tag*{{\PIProblemTagNoLink[\alpha]{\data}}}
\min_{ \substack{ \liftVar \in \Mp( \projdomdyn)^{\alldirs}  \\ u\in \Mp(\domstat)^{\domt} }}\!
\| \liftVar_{\hat\theta} \|_\M
\!+\!\frac1{2\alpha}\!\sum_{t\in\alltimes}\|\fopstat_tu_t\!-\!\data_t\|_H^2
\quad \text{ such that }
\mv_t \liftVar_\theta = \Rs_\theta u_t \text{ for all }t\!\in\!\domt, \, \theta\!\in\!\alldirs.
\end{equation*}
Regarding the well-posedness and the equivalence of \PInoisyProblemTag[\sqrt\delta]{\data^\delta} to  \problemTag[\sqrt\delta]{\data^\delta} we refer to \ref{sec:alternativeFormulations}.
The reason for using this version first is that it simplifies the construction of dual variables. An extension of the results to \problemTag[\sqrt\delta]{\data^\delta} will subsequently be carried out in the next subsection below.
The error estimation result for both models, to be obtained in \cref{sec:dualVariables,sec:dualConstructionII}, can be summarized as follows.
}

\begin{thm}[Reconstruction error estimates]\label{thm:errorEstimatesSummary}
Let $(u,\mu)$ be the solution of \problemTag[\sqrt\delta]{\data^\delta} or \PInoisyProblemTag[\sqrt\delta]{\data^\delta}, and let \added{\label{txt:itemeReconstructible4}\lookUp{\ref{iteme}}$\emptyset\neq\goodtimes\subset\alltimes$ such that $u_t^\dagger$ is stably reconstructible for all $t\in\goodtimes$}.
Pick some $\gooddirs\subset\alldirs$.
Then there exist constants $R,C_1>0$ (depending on $\lambda^\dagger$, $\goodtimes$) and $C_2>0$ (depending on $\lambda^\dagger$, $\gooddirs$) such that
\begin{enumerate}
\item\label{enm:goodSnapshotEstimate}
we have $\unbalancedWasserstein[2]{R}(u_t,u_t^\dagger)\leq C_1\sqrt\delta$ for all $t\in\goodtimes$,
\item\label{enm:liftVarEstimate}
we have $\unbalancedWasserstein[2]{R_\theta}(\liftVar_\theta,\liftVar_\theta^\dagger)\leq\frac{C_1}{R_\theta^2}\sqrt\delta$ for all $\theta\in\alldirs$ ($\mdirs$-almost all $\theta\in\alldirs$ in case of model \problemTag[\sqrt\delta]{\data^\delta}), where $R_\theta=\min\{R,\Delta_\theta/3\}/\sqrt{1+\max_{t\in\goodtimes}t^2}$ and $\Delta_\theta\geq0$ is such that $\liftVar_\theta^\dagger$ contains no $\Delta_\theta$-coincidence or $\Delta_\theta$-ghost particle with respect to $\goodtimes$ according to \cref{def:DeltaGhost},
\item\label{enm:snapshotEstimate}
we have $\unbalancedWasserstein[2]{R_t}(u_t,u_t^\dagger)\leq C_1C_2(t^2+\frac1{R_t^4})\sqrt\delta$ for all $t\in\domt$,
where $R_t=\min\{R,\Delta_t/({3\sqrt{1+t^2}})\}$ and $\Delta_t\in[0,\inf_{\theta\in\gooddirs}\Delta_\theta]$ is such that $u_t^\dagger$ contains no $\Delta_t$-coincidence or $\Delta_t$-ghost particle with respect to $\gooddirs$ according to \cref{def:DeltaGhostII}.
\end{enumerate}
\end{thm}

\Cref{thm:errorEstimatesSummary} implies \cref{thm:reconstructionErrorEstimate}:
Given $\goodtimes$, the minimum $\Delta$ such that $\liftVar^\dagger_\theta$ contains a $\Delta$-coincidence or $\Delta$-ghost particle with respect to $\goodtimes$
is continuous in $\theta$, thus it assumes a minimum $M$ at some $\theta'$ in the compact $\overline\alldirs$.
This minimum is positive since $\liftVar^\dagger_{\theta'}$ contains no coincidence or ghost particle.
Hence, in \cref{thm:errorEstimatesSummary}\eqref{enm:liftVarEstimate} we may take $\Delta_\theta=\frac M2>0$ for all $\theta\in\overline\alldirs$,
which then yields \cref{thm:reconstructionErrorEstimate}\eqref{enm:liftVarEstimateIntro}.
To deduce \cref{thm:reconstructionErrorEstimate}\eqref{enm:snapshotEstimateIntro}
from \cref{thm:errorEstimatesSummary}\eqref{enm:goodSnapshotEstimate} and \eqref{enm:snapshotEstimate}
it suffices to show that $t^2$ can be bounded from above and $\Delta_t$ away from zero, uniformly for $t\in\overline\domt$.
Boundedness of $t^2$ follows from compactness of $\overline\domt$, so it remains to show uniform positivity of $\Delta_t$.
Picking $\gooddirs=\overline\alldirs$, let $\Delta$ be minimal such that $u^\dagger_t$ contains a $\Delta$-coincidence or $\Delta$-ghost particle with respect to $\gooddirs$.
As before, $\Delta$ depends continuously on $t$ and thus attains its minimum $N$ at some $t'$ in the compact $\overline\domt$.
That the particles have positive mutual distance in the projections $(\Rs_\theta u^\dagger_{t'})_{\theta\in\overline\alldirs}$
and the support $\supp u^\dagger_{t'}$ can be uniquely reconstructed from $(\supp(\Rs_\theta u^\dagger_{t'}))_{\theta\in\overline\alldirs}$
implies that $u^\dagger_{t'}$ contains no $0$-coincidence or $0$-ghost particle, hence $N>0$ and we may choose $\Delta_t=\frac N2>0$ for all $t\in\overline\domt$, as desired.

\removed{In this \namecref{sec:dualVariables} we prove the result for the (in this context simpler) model \PInoisyProblemTag[\sqrt\delta]{\data^\delta};
the proof will be extended to the setting of \problemTag[\sqrt\delta]{\data^\delta} in \cref{sec:dualConstructionII}. We}
\added{For proving our results on convergence rates we} will exploit that the optimization problem has multiple equivalent formulations due to our constraint of nonnegative masses.
Indeed, we may write the regularizer as the norm of the snapshots or equivalently the norm of the auxiliary variable $\liftVar$.
As seen in the previous section, the derivation of error estimates
requires the construction of appropriate dual variables, which in our case is somewhat involved.
In particular, to obtain error estimates on the particle configurations at all time points outside $\goodtimes$,
the error information needs to be passed from the snapshots at times $\goodtimes$ to the variable $\liftVar$ and then back again,
\added{\label{txt:itemeProofLogic8}\lookUp{\ref{iteme}}just like the procedure for exact reconstruction in \cref{sec:exact_recon_dim_reduced},}
which is reflected in a stepwise construction of the corresponding dual variables.
In particular, we will reduce the derivation of error estimates on $(u,\liftVar)$ to the derivation of error estimates for the much simpler static reconstruction problem
of finding an approximation $u_t$ to $u_t^\dagger$ from a noisy observation of $\fopstat_tu_t^\dagger$.
This way the role of the observation operator $\fopstat_t$ is separated from the role of combining measurements from multiple time points.
An error analysis of the \emph{static problem} is a classic superresolution problem and has for a particular choice of $\fopstat_t$ already been performed by Cand\`es and Fernandez-Granda \cite{Candes2013}
(even though they did not use our language of unbalanced Wasserstein distances).
Our contribution, in contrast, is to show how the error estimates for the static problem transfer to error estimates of the \emph{dynamic problem}.
Our analysis also applies to the reconstruction model \eqref{eq:lifted_exact} by Alberti et al.\ \cite{AlbertiAmmariRomeroWintz2019}, which forms the motivation of our work,
but we will only perform it explicitly for our novel reconstruction model \ref{eq:dim_reduced_product_only_noisy},
since this is the more complicated case and has the dimensionality advantage over \eqref{eq:lifted_exact}.
\added{\label{txt:itemeReconstructible3}\lookUp{\ref{iteme}}%
In the following we will say the measure $\nu^\dagger$ from \cref{def:reconstructible} is \emph{$(\kappa,\mu,R)$-stably reconstructible via $v^\dagger,v^s$}
to indicate the corresponding dual variables and values of $\kappa$, $\mu$, and $R$.
The static problem admits an error analysis if the measure to be reconstructed is stably reconstructible, as stated next.
}%

\removed{
If the static problem admits an error analysis we shall call the problem \added{\lookUp{\ref{iteme}}stably} reconstructible according to the following.

\begin{defn}[\added{Stably} reconstructible measure]\label{def:reconstructible}
We call a measure $\nu^\dagger=\sum_{i=1}^Nm_i\delta_{x_i}\in\Mp(\domstat)$ \emph{\added{stably} reconstructible} (from measurements via $\fopstat_t$),
if one can find $\kappa,\mu,R>0$ ($R$ smaller than the minimum point separation in $\nu^\dagger$)
such that for any $N$-tuple $s\in\{-1,1\}^N$ and associated measure $\nu^s=\sum_{i=1}^Ns_i\delta_{x_i}\in\M(\domstat)$ there exist dual variables $v^\dagger,v^s$ with
$-\fopstat_t^*v^\dagger\in\partial\|\cdot\|_{\M}(\nu^\dagger)$, $-\fopstat_t^*v^s\in\partial\|\cdot\|_{\M}(\nu^s)$ and
\begin{align*}
-\fopstat_t^*v^\dagger(x)&\leq1-\kappa\min\{R,\dist(x,\{x_1,\ldots,x_N\})\}^2
&&\text{for all }x\in\domstat,\\
\fopstat_t^*v^s(x_i)\fopstat_t^*v^s(x)&\geq1-\mu\dist(x,x_i)^2
&&\text{for all }x\in B_R(\{x_i\}),\,i=1,\ldots,N.
\end{align*}
We will say $\nu^\dagger$ is \emph{$(\kappa,\mu,R)$-\added{stably} reconstructible via $v^\dagger,v^s$} to indicate the corresponding dual variables and values of $\kappa$, $\mu$, and $R$.
\end{defn}

Of course, the terminology comes from the fact that a reconstructible measure can \added{stably} be reconstructed from its measurement
\added{\label{txt:itemeProofLogic9}\lookUp{\ref{iteme}}as stated next.}
}%

\begin{thm}[\added{Stably} reconstructible measure]\label{thm:staticErrorEstimate}
If $\nu^\dagger$ is $(\kappa,\mu,R)$-\added{stably} reconstructible via $v^\dagger,v^s$ and $\noisyData$ is an approximation to $\noiseFreeData=\fopstat_t\nu^\dagger$ with $\fidelity{\noisyData}(\noiseFreeData)\leq\delta$,
then the solution to
\begin{equation*}
\min_{\nu^\delta}\|\nu^\delta\|_++\frac1\alpha\fidelity{\noisyData}(\fopstat_t\nu^\delta)
\end{equation*}
satisfies
\begin{equation*}
\unbalancedWasserstein[2]{R}(\nu^\delta,\nu^\dagger)
\leq\tfrac{12\!+\!7\max\{\kappa,\mu\}R^2}{4\kappa\alpha}\left(\delta\!+\!\fidelity{\noisyData}^*(2\alpha v^\dagger)\!+\!\fidelity{\noisyData}^*(-2\alpha v^\dagger)\!+\!\!\!\max_{s\in\{-1,1\}^N}\!\!\fidelity{\noisyData}^*(2\alpha v^s)\!+\!\fidelity{\noisyData}^*(-2\alpha v^s)\right).
\end{equation*}
In particular, for $\delta=0$ the limit $\alpha\to0$ implies $\nu^\delta=\nu^\dagger$.
\end{thm}
\begin{proof}
This is just the special case of \cref{rem:nonnegativeMeasures} for $\approximation_1=\nu^\delta$, $K=\fopstat_t$, $\optimalDual=v^\dagger$, $w=v^s$ with $s_i=\sgn((\nu^\delta-\nu^\dagger)(B_R(\{x_i\})))$,
$\approximation_2$ and $\regularizer_2$ nonexistent.
\end{proof}

Depending on the measurement operator, simple criteria for \added{stable} reconstructibility can be found.
\added{The following} example in one space dimension is proven in \cite[Lem.\,2.4 \& 2.5]{Candes2013}
(higher-dimensional versions are also provided in that work, albeit without detail).

\begin{thm}[\added{Stably} reconstructible measure for truncated Fourier series, Cand\`es \& Fernandez-Granda 2013]\label{thm:Candes2013}
Let $\fopstat_t\added{=\ftrunc}$ be the truncated Fourier series on $\domstat=\sphere^1$ which only keeps coefficients up to a maximum absolute frequency $\maxFrequency>0$.
There exist positive constants $c_0,c_1,c_2,c_3>0$ such that,
if the Dirac locations in $u^\dagger=\sum_{i=1}^Nm_i\delta_{x_i}\in\Mp(\domstat)$ have distance no smaller than $\minSeparation=c_0/\maxFrequency$ to each other,
then $u^\dagger$ is $(c_1\maxFrequency^2,c_2\maxFrequency^2,c_3/\maxFrequency)$-\added{stably} reconstructible.
\end{thm}

Before we begin the derivation of error estimates for \PInoisyProblemTag[\sqrt\delta]{\data^\delta}, let us introduce the abbreviations
\begin{align*}
\Rfp&:\M(\domstat)^\domt\to\M(\projdomstat)^{\alldirs\times\domt},&
\Rfp u&=(\Rs_\theta u_t)_{\theta\in\alldirs,t\in\domt},&\\
\amvp[1]&:\M(\projdomdyn)^\alldirs\to\M(\projdomstat)^{\alldirs\times\domt},&
\amvp[1]\liftVar&=(\mv^1_t\liftVar_\theta)_{\theta\in\alldirs,t\in\domt},&\\
\fopdynp&:\M(\domstat)^\domt\to H^{|\alltimes|},&
\fopdynp u&=(\fopstat_t u_t)_{t\in\alltimes},&\\
\iota_+&:\prod_{i\in I}\M(A_i)\to\{0,\infty\},&
\iota_+(\nu)&=\begin{cases}
0&\text{if }\nu_i\in\Mp(A_i)\text{ for all }i\in I,\\
\infty&\text{else,}
\end{cases}&
\end{align*}
where $I$ denotes an arbitrary index set and $A_i$ are Lipschitz domains.
Note that the above linear operators are all continuous with respect to the product topology
(a map $M:\prod_{i\in I}X_i\to\prod_{j\in J}Y_j$ between product spaces with the product topology is continuous
if and only if $p_j\circ M$ is continuous for all projections $p_j:\prod_{j\in J}Y_j\to Y_j$,
and for all above operators $p_j\circ M$ is the composition of a continuous projection $q_i:\prod_{i\in I}X_i\to X_i$ with a continuous linear operator between Banach spaces).
Furthermore, $\iota_0$ shall be the convex indicator function of the set $\{0\}$.
To translate the setting from the previous section to \ref{eq:dim_reduced_product_only_noisy} we choose
\begin{gather*}
\groundTruth\equiv(u^\dagger,\liftVar^\dagger),\,
\approximation\equiv(u,\liftVar)\in\M(\domstat)^{\domt}\times\M(\projdomdyn)^{\alldirs}\equiv X,\\
K\equiv\left(\begin{smallmatrix}\fopdynp&0\\\Rfp&-\amvp[1]\end{smallmatrix}\right):X\to H^{|\alltimes|}\times\M(\projdomstat)^{\alldirs\times\domt}\equiv Y,\\
\fidelity{\data}(a,b)=\tfrac12\sum_{t\in\alltimes}\|a_t-\data_t\|_H^2+\iota_{0}(b),\,
\regularizer(u,\liftVar)=\iota_+(u,\liftVar)+\ast,
\end{gather*}
where $\ast$ will either be $\|\liftVar_\theta\|_{\M}$ or $\|u_t\|_{\M}$ for particular choices of $\theta$ and $t$,
depending on what is most convenient for the desired estimate
(by \cref{prop:radon_properties,prop:move_properties} and the constraint $\Rfp u=\amvp[1]\liftVar$, all these choices are equivalent).
The corresponding predual spaces are given by
\begin{align*}
X^*&=\{(\phi,\psi)\in C(\domstat)^\domt\times C(\projdomdyn)^\alldirs\,|\text{ only finitely many }\phi_t,\psi_\theta\text{ are nonzero}\},\\
Y^*&=\{(v,q)\in(H^*)^{|\alltimes|}\times C(\projdomstat)^{\alldirs\times\domt}\,|\text{ only finitely many }q_{\theta,t}\text{ are nonzero}\},
\end{align*}
endowed with the final topology induced by the inclusions
$\phi_t\mapsto(\phi,\psi)\in X^*$ with $\phi_{\tilde t}=0$ for all $\tilde t\neq t$,
$\psi_\theta\mapsto(\phi,\psi)\in X^*$ with $\psi_{\tilde\theta}=0$ for all $\tilde\theta\neq\theta$
and analogously for $Y^*$
(it is straightforward to see that any continuous linear functional on $X^*$ or $Y^*$ is an element of $X$ or $Y$).
Note that above we slightly changed our notation compared to the previous section in that $\data$ here only denotes the first part of $K\groundTruth$.
This change has no consequence for the validity of the results and arguments,
hence we decided not to complicate the notation of the previous section further by directly including this from the start.
This choice of fidelity has Legendre--Fenchel conjugate $\fidelity{\data}^*(w_1,w_2)=\sum_{t\in\alltimes}\tfrac12\|(w_1)_t\|_{H^*}^2+\langle(w_1)_t,\data_t\rangle$, thus
\begin{equation}\label{eqn:dualFidelity}
\fidelity{\data}^*(2\alpha w)+\fidelity{\data}^*(-2\alpha w)=4\alpha^2\sum_{t\in\alltimes}\|(w_1)_t\|_{H^*}^2=:4\alpha^2|w_1|^2.
\end{equation}
For easier reference let us also recall the following subgradients,
\begin{align*}
w\in\partial\|\cdot\|_{+}(\nu)
\quad&\Leftrightarrow\quad
w\leq1\text{ everywhere, }w=1\text{ on }\supp\nu,\\
w\in\partial\iota_{+}(\nu)
\quad&\Leftrightarrow\quad
w\leq0\text{ everywhere, }w=0\text{ on }\supp\nu.
\end{align*}
\added{\label{txt:itemeProofLogic10}\lookUp{\ref{iteme}}Just like in \cref{sec:exact_recon_dim_reduced}
we} begin with estimating the reconstruction properties \added{only} for those snapshots that are \added{stably} reconstructible:
\added{We show that the error estimate from \cref{thm:staticErrorEstimate} is not hindered by the dynamic setting}.

\begin{thm}[Error estimates for good snapshots]\label{thm:errorGoodSnapshots}
Let $t\in\alltimes$ be such that $u_t^\dagger$ is $(\kappa,\mu,R)$-\added{stably} reconstructible via $v^\dagger,v^s$ from measurements with $\fopstat_t$.
Then the solution $(u,\liftVar)$ to \PInoisyProblemTag[\sqrt\delta]{\data^\delta} with $\alpha=\sqrt\delta$ satisfies
\begin{align*}
\unbalancedWasserstein[2]{R}(u_t,u_t^\dagger)
&\leq\frac{12\!+\!7\max\{\kappa,\mu\}R^2}{4\kappa}\left(1+4|v^\dagger|^2+\max_{s\in\{-1,1\}^N}4|v^s|^2\right)\sqrt\delta.
\end{align*}
\end{thm}
\begin{proof}
We have $u_t^\dagger=\sum_{i=1}^Nm_i\delta_{x_i+tv_i}$.
Set $s=(\sgn((u_t-u_t^\dagger)(B_R(\{x_i+tv_i\}))))_{i=1,\ldots,N}$ and let $v^\dagger$ and $v^s$ be the dual variables from \cref{def:reconstructible} (with $\nu^\dagger=u_t^\dagger$).
In order to apply \cref{rem:nonnegativeMeasures} we split $\approximation=(u,\liftVar)$ into $\approximation_1=u_t$ and $\approximation_2=((u_{\tilde t})_{\tilde t\neq t},\liftVar)$
and consider the regularizer
$$\regularizer(u,\liftVar)=\|u_t\|_++\iota_+((u_{\tilde t})_{\tilde t\neq t},\liftVar).$$
Then the choice $\optimalDual={\tilde v^\dagger\choose0}$, $w={\tilde v\choose0}$ satisfies the conditions of \cref{rem:nonnegativeMeasures},
where $\tilde v^\dagger_t=v^\dagger$, $\tilde v_t=v^s$ and $\tilde v^\dagger_{\tilde t}=\tilde v_{\tilde t}=0$ for all $\tilde t\neq t$.
Furthermore, $(K^*w)_2=0$ so that the result follows from \cref{thm:noisyReconstructionEstimate}.
\end{proof}

\added{\lookUp{\ref{iteme}}In analogy to \cref{sec:exact_recon_dim_reduced},
from the error estimates for the \added{stably} reconstructible snapshots we cannot directly obtain error estimates for all other snapshots
since all snapshots are only indirectly coupled via the position-velocity projections.
We therefore first} use the just constructed dual variables to obtain dual variables for error estimates on $\liftVar$.
Here we will also obtain an unbalanced Wasserstein bound;
its maximum transport radius $R$ will depend on how close the configuration is to producing a coincidence or a ghost particle in the following sense.

\begin{defn}[$\Delta$-coincidence and $\Delta$-ghost particle]\label{def:DeltaGhost}
Let $\nu=\sum_{i=1}^Nm_i\delta_{(x_i,v_i)}\in\Mp(\R^n\times\R^n)$ represent a set of particles and $\goodtimes\subset\R$ a set of times.
\begin{enumerate}
\item
Configuration $\nu$ contains a \emph{$\Delta$-coincidence} with respect to $\goodtimes$
if two distinct particles $i\neq j$ come $\Delta$-close at some time $t\in\goodtimes$, that is, $\dist(x_j+tv_j,x_i+tv_i)\leq\Delta$ .
\item
Configuration $\nu$ contains a \emph{$\Delta$-ghost particle} at $(x,v)\in\R^n\times\R^n$ with respect to $\goodtimes$
if at each $t\in\goodtimes$ there is a particle $i$ (not the same for all $t\in\goodtimes$) with $\dist(x+tv,x_i+tv_i)\leq\Delta$.
\end{enumerate}
\end{defn}

A $\Delta$-ghost particle is thus a nonexistent particle which would be consistent with all snapshots if the finest measurement resolution at the different time points were $\Delta$.
Plain ghost particles are $0$-ghost particles in this sense, plain coincidences are $0$-coincidences.

\begin{thm}[Error estimates for $\liftVar$]\label{thm:estimateMu}
Let $\theta\in\alldirs$ and assume there is $\goodtimes\subset\alltimes$
such that $u_t^\dagger$ is $(\hat\kappa,\hat\mu,\hat R)$-\added{stably} reconstructible via $\hat v_t^\dagger,\hat v^s_t$ from measurements with $\fopstat_t$ for all $t\in\goodtimes$.
Then the solution $(u,\liftVar)$ to \PInoisyProblemTag[\sqrt\delta]{\data^\delta} satisfies
\begin{align*}
\unbalancedWasserstein[2]{R}(\liftVar_\theta,\liftVar_\theta^\dagger)
&\leq\frac{12+7\kappa R^2}{4\kappa}\left(1+4\max_{t\in\goodtimes}|\hat v_t^\dagger|^2\right)\sqrt\delta
+\frac1{2R^2}\max_{t\in\goodtimes}\unbalancedWasserstein{R}(u_t,u^\dagger_t).
\end{align*}
for $\kappa=\sigma\hat\kappa$ and $R=\min\{\hat R,\minSeparation/3,\ghostDistance\}/\sqrt{1+\max_{t\in\goodtimes}t^2}$,
where $\sigma$ is the smallest eigenvalue of $\frac1{\card\goodtimes}\sum_{t\in\goodtimes}{1\choose-t}\otimes{1\choose-t}$
and $\minSeparation,\ghostDistance\geq0$ are such that $\liftVar_\theta^\dagger$ contains no $\minSeparation$-coincidence or $\ghostDistance$-ghost particle with respect to $\goodtimes$.
\end{thm}
\begin{proof}
This time split $\approximation=(u,\liftVar)$ into $\approximation_1=\gamma_\theta$ and $\approximation_2=(u,(\liftVar_{\tilde\theta})_{\tilde\theta\neq\theta})$
and consider the regularizer
$$\regularizer(u,\liftVar)=\|\liftVar_\theta\|_++\iota_+(u,(\liftVar_{\tilde\theta})_{\tilde\theta\neq\theta}).$$
In order to apply \cref{rem:nonnegativeMeasures} we need to construct appropriate dual variables $\optimalDual$ and $w$.
To this end define $v^\dagger\in(H^*)^{|\alltimes|}$
via $v^\dagger_t=\hat v^\dagger_t/\card{\goodtimes}$ for all $t\in\goodtimes$ and $v^\dagger_{t}=0$ else.
Furthermore define $q^\dagger\in C(\projdomstat)^{\alldirs\times\domt}$ as
\begin{equation*}
q^\dagger_{\theta,t}(x)=(1-\hat\kappa\min\{\bar R,\dist(x,\supp\Rs_\theta u_t^\dagger)\}^2)/\card{\goodtimes}
\quad\text{for }t\in\goodtimes
\end{equation*}
and $q^\dagger_{\tilde\theta,\tilde t}=0$ for $\tilde t\notin\goodtimes$ or $\tilde\theta\neq\theta$,
where $\bar R=\min\{\hat R,\minSeparation/3,\ghostDistance\}$.
Due to $\bar R\leq\minSeparation/3$, each $q_{\theta,t}^\dagger$ contains $\card\goodtimes$ distinct parabolas of width $2\bar R$.
Also, let $s_i=\sgn(\liftVar_\theta-\liftVar^\dagger_\theta)(B_R(\{(x_i\cdot\theta,v_i\cdot\theta)\}))$, $i=1,\ldots,N$, and define $q\in C(\projdomstat)^{\alldirs\times\domt}$
as $q_{\tilde\theta,\tilde t}=0$ for $\tilde t\notin\goodtimes$ or $\tilde\theta\neq\theta$ and, for $t\in\goodtimes$, as
\begin{equation*}
q_{\theta,t}(x)=s_i/\card{\goodtimes}
\quad\text{if }x\in B_{\bar R}(\{\theta\cdot(x_i+tv_i)\}),
\end{equation*}
continuously extended onto the rest of $\projdomstat$ with absolute value no larger than $1/\card{\goodtimes}$ (which again is possible due to $\bar R\leq\minSeparation/3$).
Now $\optimalDual={v^\dagger\choose q^\dagger}$ and $w={0\choose q}$ satisfy the conditions of \cref{rem:nonnegativeMeasures} with $\mu=0$.
Indeed, due to $\bar R\leq\hat R$ we have $-\Rs_\theta^*q^\dagger_{\theta,t}\leq\fopstat_t^*v^\dagger_t$ for all $t\in\goodtimes$ by construction (with equality on the support of $u_t^\dagger$) so that $\fopdynp^*v^\dagger+\Rfp^*q^\dagger\geq0$, thus
\begin{equation}\label{eqn:propertyDualLiftVar}
-(\fopdynp^*v^\dagger+\Rfp^*q^\dagger)\in\partial\iota_{+}.
\end{equation}
Furthermore, by construction, in the neighbourhood
\begin{equation*}
\mathcal N_i=\bigcap_{t\in\goodtimes}\textstyle\left\{{x\choose v}\in\R^2\,\middle|\,-\bar R\leq{1\choose t}\cdot\left[{x\choose v}-{x_i\cdot\theta\choose v_i\cdot\theta}\right]\leq\bar R\right\}
\end{equation*}
of $(\theta\cdot x_i,\theta\cdot v_i)$ we have
\begin{equation*}
((\amvp[1])^*q^\dagger)_\theta(x,v)=\sum_{t\in\goodtimes}\mv^1_tq_{\theta,t}^\dagger(x,v)=1-\hat\kappa(x-\theta\cdot x_i,v-\theta\cdot v_i)Q{x-\theta\cdot x_i\choose v-\theta\cdot v_i}
\end{equation*}
for the positive definite matrix $Q=\frac1{\card\goodtimes}\sum_{t\in\goodtimes}{1\choose-t}\otimes{1\choose-t}$ so that from $\kappa\leq\sigma\hat\kappa$ we obtain
\begin{equation*}
((\amvp[1])^*q^\dagger)_\theta(x,v)\leq1-\kappa\dist((x,v),\supp{\liftVar_\theta^\dagger})^2
\end{equation*}
in the neighbourhood $\bigcup_{i=1}^N\mathcal N_i$ of $\supp\liftVar_\theta^\dagger$.
In particular, $((\amvp[1])^*q^\dagger)_\theta\in\partial\|\cdot\|_{+}(\liftVar^\dagger_\theta)$ (note that it is everywhere bounded by one).
Furthermore, due to $\bar R\leq\ghostDistance$ there are no regions other than the $\mathcal N_i$ in which $((\amvp[1])^*q^\dagger)_\theta=\sum_{t\in\goodtimes}\mv^1_tq_{\theta,t}^\dagger$ can be greater than $1-\hat\kappa\bar R^2$.
Since $B_R(\{x_i\})\subset\mathcal N_i$ for the choice $R=\bar R/\sqrt{1+\max_{t\in\goodtimes}t^2}$ we thus actually have
\begin{equation}\label{eqn:dualVariableMuDecrease}
((\amvp[1])^*q^\dagger)_\theta(x,v)\leq1-\kappa\min\{R,\dist((x,v),\supp{\liftVar^\dagger_\theta})\}^2.
\end{equation}
Summarizing, $\optimalDual$ satisfies the conditions of \cref{rem:nonnegativeMeasures}.
Finally, by the same argument,
\begin{gather*}
\textstyle((\amvp[1])^*q)_\theta\in\partial\|\cdot\|_{\M}\left(\sum_{(x,v)\in\supp\liftVar_\theta^\dagger}(\liftVar_\theta-\liftVar^\dagger_\theta)(B_R(\{(x,v)\}))\delta_{(x,v)}\right)\\
\text{and }((\amvp[1])^*q)_\theta(x,v)=s_i\quad\text{for }(x,v)\in B_R(\{(x_i\cdot\theta,v_i\cdot\theta)\})
\end{gather*}
so that also $w$ satisfies the conditions of \cref{rem:nonnegativeMeasures} and we obtain the estimate from \cref{thm:noisyReconstructionEstimate} with $\mu=0$.

To conclude, we note that by construction $1=|\card{\goodtimes}(\Rfp^*q)_t(x)|$ for $x\in B_{\bar R}(\supp u_t^\dagger)$ and any $t\in\goodtimes$. Thus, 
\begin{multline*}
\langle(K^*w)_2,\approximation_2-\groundTruth_2\rangle
=\sum_{t\in\goodtimes}\langle(\Rfp^*q)_t,u_t-u^\dagger_t\rangle
\leq\frac1{\card{\goodtimes}}\sum_{t\in\goodtimes}\bigg(|u_t-u^\dagger_t|(\domstat\setminus B_R(\supp u_t^\dagger))\\
+\sum_{x\in\supp u_t^\dagger}|(u_t-u_t^\dagger)(B_R(\{x\}))|\bigg)
\leq\frac1{2R^2\card{\goodtimes}}\sum_{t\in\goodtimes}\unbalancedWasserstein{R}(u_t,u^\dagger_t),
\end{multline*}
where we exploited that in the unbalanced Wasserstein divergence $\unbalancedWasserstein{R}$ mass is never transported further than $R$.
\end{proof}

\begin{rem}[Signed particles]\label{rem:signedMeasures}
For the above error estimates it is essential that all particles have nonnegative mass.
We exploited this explicitly by the fact that each time we can choose our regularizer to be the total variation of exactly that variable
for which we would like to derive an error estimate.
Furthermore, we exploit this in the fact that we only require $\fopdynp^*v^\dagger+\Rfp^*q^\dagger\geq0$ in the above construction of dual variables.
Had we considered signed measures, then this condition would turn into $\fopdynp^*v^\dagger+\Rfp^*q^\dagger=0$.
However, the range of the predual observation operator $\fopdynp^*$ and the predual Radon transform $\Rfp^*$ will typically have a very small intersection
which leaves too little freedom to construct appropriate dual variables.
For instance, if we want $q^\dagger_{\theta,t}$ to be nonzero only for one $\theta$ and if $\fopstat_t\added{=\ftrunc}$ measures the Fourier coefficients at a finite number of integer frequencies,
then for irrational $\theta$ the only possible $q^\dagger_{\theta,t}$ is a constant.
\end{rem}

Finally, we derive an error estimate for all those \added{``bad''} snapshots that are not \added{stably} reconstructible from $\fopstat_t$,
for instance because their particles are too close together.
Their reconstruction works nevertheless based solely on those snapshots that \emph{are} \added{stably} reconstructible from $\fopstat_t$.
\added{\label{txt:itemeProofLogic11}\lookUp{\ref{iteme}}Again in analogy to \cref{sec:exact_recon_dim_reduced},
the necessary additional information to reconstruct the bad snapshots comes from the already reconstructed position-velocity projections,
so the corresponding dual variables will be based on the dual variables from \cref{thm:estimateMu} for reconstructing $\liftVar$.
In addition, the} construction of the dual variables is analogous to the construction in \cref{thm:estimateMu},
only the roles of $u$ and $\liftVar$ as well as of $\goodtimes$ and $\gooddirs$ swap.
This also means that we require an analogue of $\Delta$-coincidences and $\Delta$-ghost particles.

\begin{defn}[$\Delta$-coincidence and $\Delta$-ghost particle with respect to directions]\label{def:DeltaGhostII}
Let $u=\sum_{i=1}^Nm_i\delta_{x_i}\in\Mp(\R^d)$ represent a set of particles and $\gooddirs\subset\sphere^{d-1}$ a set of directions.
\begin{enumerate}
\item
Configuration $u$ contains a \emph{$\Delta$-coincidence} with respect to $\gooddirs$
if two distinct particles $i\neq j$ seem $\Delta$-close in some projection $\theta\in\gooddirs$, that is, $\dist(\theta\cdot x_j,\theta\cdot x_i)\leq\Delta$ .
\item
Configuration $u$ contains a \emph{$\Delta$-ghost particle} at $x\in\R^d$ with respect to $\gooddirs$
if for each $\theta\in\gooddirs$ there is a particle $i$ (not the same for all $\theta\in\gooddirs$) with $\dist(\theta\cdot x,\theta\cdot x_i)\leq\Delta$.
\end{enumerate}
\end{defn}

\begin{thm}[Error estimates for bad snapshots]\label{thm:estimateBadU}
Assume there is $\goodtimes\subset\alltimes$
such that $u_t^\dagger$ is $(\hat\kappa,\hat\mu,\hat R)$-\added{stably} reconstructible via $\hat v_t^\dagger,\hat v^s_t$ from measurements with $\fopstat_t$ for all $t\in\goodtimes$.
Let $t\in\domt\setminus\goodtimes$ and pick any finite $\gooddirs\subset\alldirs$.
Then the solution $(u,\liftVar)$ to \PInoisyProblemTag[\sqrt\delta]{\data^\delta} satisfies
\begin{align*}
\unbalancedWasserstein[2]{\bar R}(u_t,u_t^\dagger)
&\leq\frac{12+7\bar\kappa\bar R^2}{4\bar\kappa}\left(1+4\max_{t\in\goodtimes}|\hat v_t^\dagger|^2\right)\sqrt\delta
+\frac1{2\bar R^2}\max_{\theta\in\gooddirs}\unbalancedWasserstein{\bar R}(\liftVar_\theta,\liftVar_\theta^\dagger)
\end{align*}
for $\bar\kappa=\sigma\kappa/(1+t^2)$ and $\bar R=\min\{(1+t^2)R,\minSeparation/(3\sqrt{1+t^2}),\ghostDistance\}$,
where $\sigma$ is the smallest eigenvalue of $\frac1{\card\gooddirs}\sum_{\theta\in\gooddirs}\theta\otimes\theta$,
$\minSeparation,\ghostDistance\geq0$ are such that $u^\dagger_t$ contains no $\minSeparation$-coincidence or $\ghostDistance$-ghost particle with respect to $\gooddirs$,
and $\kappa,R\geq0$ are the minimum values from the statement of \cref{thm:estimateMu}, applied to all $\theta\in\gooddirs$.
\end{thm}
\begin{proof}
This time we consider the splitting of $\approximation=(u,\liftVar)$ into $\approximation_1=u_t$ and $\approximation_2=((u_{\tilde t})_{\tilde t\neq t},\liftVar)$ as well as the regularizer choice
$$\regularizer(u,\liftVar)=\|u_t\|_++\iota_+((u_{\tilde t})_{\tilde t\neq t},\liftVar).$$
For $\tilde\kappa$ to be determined later, define $q^\dagger\in C(\projdomstat)^{\alldirs\times\domt}$ as
\begin{equation*}
q^\dagger_{\theta,t}(x)=-(1-\tilde\kappa\min\{\bar R,\dist(x,\supp\mv^1_t\liftVar^\dagger_\theta)\}^2)/\card{\gooddirs}
\quad\text{for }\theta\in\gooddirs
\end{equation*}
and $q^\dagger_{\tilde\theta,\tilde t}=0$ for $\tilde t\neq t$ or $\tilde\theta\notin\gooddirs$.
Due to $\bar R\leq\minSeparation/3$, each $q_{\theta,t}^\dagger$ contains $\card\goodtimes$ distinct parabolas of width $2\bar R$.
Also, let $s_i=\sgn(u_t-u_t^\dagger)(B_{\bar R}(\{x_i+tv_i\}))$, $i=1,\ldots,N$, and define $q\in C(\projdomstat)^{\alldirs\times\domt}$
as $q_{\tilde\theta,\tilde t}=0$ for $\tilde t\neq t$ or $\tilde\theta\notin\gooddirs$ and, for $\theta\in\gooddirs$, as
\begin{equation*}
q_{\theta,t}(x)=-s_i/\card{\gooddirs}
\quad\text{if }x\in B_{\minSeparation/3}(\theta\cdot(x_i+tv_i)),
\end{equation*}
continuously extended onto the rest of $\projdomstat$ with absolute value no larger than $1/\card{\gooddirs}$.
Let further $w^\theta\in Y^*$ for $\theta\in\gooddirs$ denote the first dual variable constructed in the proof of \cref{thm:estimateMu} (there denoted $w^\dagger$),
and define
\begin{equation*}
w^\dagger={0\choose q^\dagger}+\frac1{\card{\gooddirs}}\sum_{\theta\in\gooddirs}w^\theta,
\quad
w={0\choose q}.
\end{equation*}
We now pick $\tilde\kappa$ small enough such that the conditions of \cref{rem:nonnegativeMeasures} are satisfied with $\mu=0$.
First note that $((0,(\amvp[1])^*)w^\dagger)_\theta$ is zero for all $\theta\notin\gooddirs$.
Now recall from \eqref{eqn:dualVariableMuDecrease} that for $\theta\in\gooddirs$ we have
$(0,(\mv^1)^*)w^\theta(x,v)\leq1-\kappa\min\{R,\dist((x,v),\supp\liftVar_\theta^\dagger)\}^2$.
Thus, for the choice $\tilde\kappa=\kappa/(1+t^2)$ and due to $\bar R\leq(1+t^2)R$ and $\bar R\leq\minSeparation/3$ we have $(0,(\amvp[1])^*)w^\dagger\leq0$ everywhere
with equality on $\supp\liftVar^\dagger_\theta$.
Next recall from \eqref{eqn:propertyDualLiftVar} that $w^\theta$ satisfies $(\fopdynp^*,\Rfp^*)w^\theta\geq0$ (with equality on the support of $u^\dagger$),
thus by construction we have $(-(\fopdynp^*,\Rfp^*)w^\dagger)_{\tilde t}\leq0$ for $\tilde t\neq t$ with equality on $\supp u_{\tilde t}^\dagger$.
Furthermore, $(-(\fopdynp^*,\Rfp^*)w^\dagger)_{t}=-(\Rfp^*q^\dagger)_t\leq1$ with equality on $\supp u_t$.
Even more, by construction, in the neighbourhood
\begin{equation*}
\mathcal N_i=\bigcap_{\theta\in\gooddirs}\textstyle\left\{x\in\R^d\,\middle|\,-\bar R\leq\theta\cdot\left[x-(x_i+tv_i)\right]\leq\bar R\right\}
\end{equation*}
of $x_i+tv_i$ we have
\begin{equation*}
-(\Rfp^*q^\dagger)_t=-\sum_{\theta\in\alldirs}\Rs_\theta^*q^\dagger_{\theta,t}=1-\tilde\kappa(x-x_i-tv_i)^TQ(x-x_i-tv_i)
\end{equation*}
for the symmetric positive semi-definite $Q=\frac1{\card\gooddirs}\sum_{\theta\in\gooddirs}\theta\otimes\theta$
so that from $\bar\kappa\leq\sigma\tilde\kappa$ we obtain
$$-(\Rfp^*q^\dagger)_t(x)\leq1-\bar\kappa\min\{\bar R,\dist(x,\supp u_t)\}^2$$
(here we exploited $\bar R\leq\ghostDistance$).
Summarizing, $w^\dagger$ fulfills the conditions from \cref{rem:nonnegativeMeasures}.
By the same argument,
\begin{gather*}
\textstyle-(\Rfp^*q)_t\in\partial\|\cdot\|_{\M}\left(\sum_{x\in\supp u_t^\dagger}(u_t-u_t^\dagger)(B_{\bar R}(x))\delta_{x}\right)\\
\text{and }-(\Rfp^*q)_t(x)=s_i\text{ for }x\in B_{\bar R}(\{x_i+tv_i\})
\end{gather*}
so that also $w$ satisfies the conditions of \cref{rem:nonnegativeMeasures} and we obtain the estimate from \cref{thm:noisyReconstructionEstimate} with $\mu=0$
(note that $w^\dagger_1=v^\dagger=\frac1{\card\goodtimes}\hat v^\dagger$ using the notation from the proof of \cref{thm:estimateMu}).

To conclude, we note that, due to $\bar R\leq\minSeparation/(3\sqrt{1+t^2})$, by construction we have $1=|\card{\gooddirs}((\amvp[1])^*q)_\theta(x,v)|$ for $(x,v)\in B_{\bar R}(\supp\liftVar_\theta^\dagger)$ and any $\theta\in\gooddirs$. Thus, 
\begin{multline*}
\langle(K^*w)_2,\approximation_2-\groundTruth_2\rangle
=-\sum_{\theta\in\gooddirs}\langle(\mv^1)^*q_{\theta,t},\liftVar_\theta-\liftVar_\theta^\dagger\rangle
\leq\frac1{\card{\gooddirs}}\sum_{\theta\in\gooddirs}\bigg(|\liftVar_\theta-\liftVar_\theta^\dagger|(\projdomdyn\setminus B_{\bar R}(\supp\liftVar_\theta^\dagger))\\
+\sum_{(x,v)\in\supp\liftVar_\theta^\dagger}|(\liftVar_\theta-\liftVar_\theta^\dagger)(B_{\bar R}(\{(x,v)\}))|\bigg)
\leq\frac1{2\bar R^2\card{\gooddirs}}\sum_{\theta\in\gooddirs}\unbalancedWasserstein{\bar R}(\liftVar_\theta,\liftVar_\theta^\dagger),
\end{multline*}
where again we exploited that in the unbalanced Wasserstein divergence $\unbalancedWasserstein{\bar R}$ mass is never transported further than $\bar R$.
\end{proof}

\Cref{thm:errorEstimatesSummary} for \PInoisyProblemTag[\sqrt\delta]{\data^\delta} now follows from combining \cref{thm:errorGoodSnapshots,thm:estimateMu,thm:estimateBadU}.

\subsection{Construction of dual variables for \texorpdfstring{\ref{eq:dim_reduced_general}}{P\_alpha(f)}}\label{sec:dualConstructionII}
We now modify the error analysis from the previous \namecref{sec:dualVariables} for reconstruction via \PInoisyProblemTag[\sqrt\delta]{\data^\delta}
to an error analysis for the model \problemTag[\sqrt\delta]{\data^\delta}.
Since in this model the slices $\liftVar_\theta$ are only defined for $\mdirs$-almost every $\theta\in\alldirs$,
our error estimate will also just hold for $\mdirs$-almost every slice.
This time we employ the operators\phantomsection\label{eqn:mixedTopologyOperators}
\begin{align*}
\Rfm&:\M(\domstat)^\domt\to\M(\alldirs\times\projdomstat)^{\domt},&
\Rfm u&=(\Rs u_t)_{t\in\domt},&\\
\amvm[1]&:\M(\alldirs\times\projdomdyn)\to\M(\alldirs\times\projdomstat)^{\domt},&
\amvm[1]\liftVar&=(\amvt[1]{t}\liftVar)_{t\in\domt},&
\end{align*}
and choose
\begin{gather*}
\groundTruth\equiv(u^\dagger,\liftVar^\dagger),\,
\approximation\equiv(u^\delta,\liftVar^\delta)\in\M(\Omega)^{|\alltimes|}\times\M(\alldirs\times\projdomdyn)\equiv X,\\
K\equiv\left(\begin{smallmatrix}\fopdynp&0\\\Rfm&-\amvm[1]\end{smallmatrix}\right):X\to H^{|\alltimes|}\times\M(\alldirs\times\projdomstat)^\domt\equiv Y,\\
\fidelity{\data}(a,b)=\tfrac12\sum_{t\in\alltimes}\|a_t-\data_t\|_H^2+\iota_{0}(b),\,
\regularizer(u,\liftVar)=\iota_+(u,\liftVar)+\ast,
\end{gather*}
where $\ast$ will either be $\|u_t\|_{\M}$ for some $t\in\alltimes$ or a weighted total variation of $\liftVar$,
\begin{equation*}
\|\liftVar\|_\weight
=\int_{\alldirs\times\projdomdyn}\weight(\theta)\wrt|\liftVar|(\theta,(x,v))
=\int_\alldirs\|\liftVar_\theta\|_{\M}\weight(\theta)\wrt\mdirs
\end{equation*}
for some fixed weight $\weight:\alldirs\to[0,\infty)$ with $\int_\alldirs\weight\wrt\mdirs=1$
(the rightmost expression is well-defined for all $\liftVar$ admissible in \ref{eq:dim_reduced_general} by \cref{lem:projected_constraint_decomposition_mixed_model}).
As before, all these choices for $\ast$ are equivalent.

Essentially, the necessary changes to the analysis are the following, where we now work backwards through the previous \namecrefs{sec:unbalancedTransport}.%
\begin{itemize}
\item
For the estimates we needed to construct dual variables $\optimalDual$ (and similarly $w$) such that $K^*\optimalDual\in\partial\regularizer(\groundTruth)$.
Since in the previous \namecref{sec:dualVariables} the regularizer acted separately on each $u_t$ and $\liftVar_\theta$ for all $t\in\alltimes$ and $\theta\in\alldirs$,
this amounted to constructing continuous functions $q_{\theta,t}^\dagger$ separately for each $t,\theta$.
Now, however, that our variable $\liftVar$ stems from $\Mp(\alldirs\times\projdomdyn)$, these dual variables will have to be continuous in $\theta$,
thus we need new versions of \cref{thm:estimateMu,thm:estimateBadU}
(while the construction from \cref{thm:errorGoodSnapshots} stays unchanged apart from replacing the space $\Mp(\projdomdyn)^\alldirs$ by $\Mp(\alldirs\times\projdomdyn)$ everywhere).
Roughly, the strategy is to smooth $q_{\theta,t}^\dagger$ a little out in $\theta$.
\item
So far we only had to estimate the mass distribution of approximations to discrete measures,
for which we used Bregman distances with respect to the total variation. %
Now however, when we aim for estimating the mass distribution in $\liftVar$,
the corresponding ground truth measure $\liftVar^\dagger$ is not discrete, but rather a measure concentrated on codimension-2 manifolds
(each $\theta$-slice $\liftVar_\theta^\dagger$ is discrete, however, these are only well-defined $\mdirs$-almost everywhere so that our earlier estimates do not carry over).
Hence we will require a new version of \cref{thm:massFromBregman},
where weighted averages of mass distributions of $\liftVar_\theta$ are estimated based on Bregman distances with respect to weighted total variations.
\item
An additional argument is required to pass from mass distribution estimates of weighted averages of $\liftVar_\theta$ to estimates for $\mdirs$-almost all $\theta$.
\end{itemize}

We begin with a new version of \cref{thm:massFromBregman}, for simplicity directly adapted to our setting of estimating the mass distribution in $\liftVar$.

\begin{thm}[Weighted mass distribution from Bregman distance]\label{thm:massFromBregman2}
Let $\weight:\alldirs\to[0,\infty)$ be continuous with $\int_\alldirs\weight\wrt\mdirs=1$.
Assume $\liftVar^\dagger=\Rj\lambda^\dagger\in\Mp(\alldirs\times\projdomdyn)$ for a discrete measure $\lambda^\dagger$ with particles $(x_1,v_1),\ldots,(x_N,v_N)$,
and let $v^\dagger\in\partial\|\cdot\|_\weight(\liftVar^\dagger)\subset C(\alldirs\times\projdomdyn)$.
If $v^\dagger$ satisfies
\begin{equation*}
v^\dagger(\theta,(x,v))\leq\weight(\theta)(1-\kappa\min\{R,\dist((x,v),\supp\liftVar_\theta^\dagger)\}^2)
\end{equation*}
for some $\kappa,R>0$ with $2R$ smaller than the minimum distance between the points in the support of $\liftVar_\theta^\dagger$ for $\theta\in\supp\weight$,
then for any $\liftVar\in\Mp(\alldirs\times\projdomdyn)$ of the form $\liftVar=\mdirs \prodm \theta \liftVar_\theta$ we have
\begin{align*}
\int_\alldirs\weight(\theta)|\liftVar_\theta|(\projdomdyn\setminus B_R(\supp\liftVar_\theta^\dagger))\wrt\mdirs
&\leq\frac1{\kappa R^2}\BregmanDistance{\|\cdot\|_\weight}{v^\dagger}(\liftVar,\liftVar^\dagger),\\
\sum_{i=1}^N\int_\alldirs\weight(\theta)\int_{B_R(\{p_i^\theta\})}\dist(p,p_i^\theta)^2\wrt|\liftVar_\theta|(p)\wrt\mdirs
&\leq\frac1{\kappa}\BregmanDistance{\|\cdot\|_\weight}{v^\dagger}(\liftVar,\liftVar^\dagger),
\end{align*}
where we abbreviated $p_i^\theta=(\theta\cdot x_i,\theta\cdot v_i)$.
Furthermore, let $v\in\partial\|\cdot\|_{\weight}(\bar\liftVar)\subset C(\alldirs\times\projdomdyn)$ for
$$\bar\liftVar=\sum_{i=1}^Ns_i\hd^{d-1}\restr\{p_i^\theta\,|\,\theta\in\alldirs\}
\quad\text{ with }
s_i=\int_\alldirs\weight(\theta)(\liftVar_\theta-\liftVar^\dagger_\theta)(B_R(\{p_i^\theta\}))\wrt\mdirs.$$
If $v$ satisfies
\begin{equation*}
s_iv(\theta,p)\geq\weight(\theta)(1-\mu\dist(p,p_i^\theta)^2)\quad\text{for all }p\in B_R(\{p_i^\theta\}),\,i=1,\ldots,N,
\end{equation*}
for some $\mu>0$, then we additionally have
\begin{equation*}
\sum_{i=1}^N\left|\int_\alldirs\weight(\theta)(\liftVar_\theta-\liftVar^\dagger_\theta)(B_R(\{p_i^\theta\}))\wrt\mdirs\right|
\leq\frac{1+\mu R^2}{\kappa R^2}\BregmanDistance{\|\cdot\|_\weight}{v^\dagger}(\liftVar,\liftVar^\dagger)+\langle v,\liftVar-\liftVar^\dagger\rangle.
\end{equation*}
\end{thm}
\begin{proof}
Abbreviate $B_\theta^c=\projdomdyn\setminus B_R(\supp\liftVar_\theta^\dagger)$, then the first two inequalities follow from
\begin{multline*}
\int_\alldirs
\weight(\theta)\kappa R^2|\liftVar_\theta|(B_\theta^c)
+\weight(\theta)\kappa\sum_{i=1}^N\int_{B_R(\{p_i^\theta\})}\dist(p,p_i^\theta)^2\wrt|\liftVar_\theta|p
\wrt\mdirs(\theta)\\
\leq\int_\alldirs
\int_{B_\theta^c}\weight(\theta)-v^\dagger(\theta,p)\wrt\liftVar_\theta(p)
+\sum_{i=1}^N\int_{B_R(\{p_i^\theta\})}\weight(\theta)-v^\dagger(\theta,p)\wrt\liftVar_\theta(p)
\wrt\mdirs(\theta)\\
=\int_{\alldirs\times\projdomdyn}(\weight(\theta)-v^\dagger(\theta,p))\wrt\liftVar(\theta,p)
=\|\liftVar\|_\weight-\langle v^\dagger,\liftVar\rangle
=\|\liftVar\|_\weight-\|\liftVar^\dagger\|_\weight-\langle v^\dagger,\liftVar-\liftVar^\dagger\rangle
=\BregmanDistance{\|\cdot\|_\weight}{v^\dagger}(\liftVar,\liftVar^\dagger).
\end{multline*}
The third inequality is obtained as follows. We have
\begin{multline*}
\sum_{i=1}^N\left|\int_\alldirs\weight(\theta)(\liftVar_\theta-\liftVar^\dagger_\theta)(B_R(\{p_i^\theta\}))\wrt\mdirs\right|\\
=\int_{\{(\theta,p)\,|\,\dist(p,\supp\liftVar_\theta^\dagger)<R\}}v\wrt(\liftVar-\liftVar^\dagger)
+\sum_{i=1}^N\int_\alldirs\int_{B_R(\{p_i^\theta\})}\weight(\theta)s_i-v\wrt(\liftVar_\theta-\liftVar^\dagger_\theta)\wrt\mdirs,
\end{multline*}
where the first summand can be estimated as
\begin{multline*}
\int_{\{(\theta,p)\,|\,\dist(p,\supp\liftVar_\theta^\dagger)<R\}}v\wrt(\liftVar-\liftVar^\dagger)
=\langle v,\liftVar-\liftVar^\dagger\rangle-\int_\alldirs\int_{B_\theta^c}v\wrt(\liftVar_\theta-\liftVar_\theta^\dagger)\wrt\mdirs(\theta)\\
\leq\langle v,\liftVar-\liftVar^\dagger\rangle+\int_\alldirs\weight(\theta)|\liftVar_\theta-\liftVar_\theta^\dagger|(B_\theta^c)\wrt\mdirs
\leq\langle v,\liftVar-\liftVar^\dagger\rangle+\frac1{\kappa R^2}\BregmanDistance{\|\cdot\|_\weight}{v^\dagger}(\liftVar,\liftVar^\dagger),
\end{multline*}
while the second term satisfies
\begin{multline*}
\sum_{i=1}^N\int_\alldirs\int_{B_R(\{p_i^\theta\})}\weight(\theta)s_i-v\wrt(\liftVar_\theta-\liftVar^\dagger_\theta)\wrt\mdirs\\
\leq\sum_{i=1}^N\int_\alldirs\weight(\theta)\int_{B_R(\{p_i^\theta\})}\mu\dist(p,p_i^\theta)^2\wrt|\liftVar_\theta|(p)\wrt\mdirs(\theta)
\leq\frac\mu\kappa\BregmanDistance{\|\cdot\|_\weight}{v^\dagger}(\liftVar,\liftVar^\dagger).
\mbox{\qedhere}%
\end{multline*}
\end{proof}

The previous mass distribution estimate can now be used to obtain estimates on $\liftVar$
\added{\label{txt:itemeProofLogic12}\lookUp{\ref{iteme}}in the following, previously announced new version of \cref{thm:estimateMu}.
In contrast to \cref{thm:estimateMu} it does not give an estimate for single slices $\liftVar_\theta$ since now $\liftVar$ is viewed as a variable in $\Mp(\alldirs\times\projdomdyn)$.
For this reason the error estimate is not merely expressed in terms of the unbalanced Wasserstein divergence, which would be too weak an estimate:
The unbalanced Wasserstein divergence between the reconstruction $\liftVar$ and the ground truth $\liftVar^\dagger$ would also allow mass exchange between different slices $\liftVar_\theta$.
Instead we will turn the error estimate into an unbalanced Wasserstein divergence estimate for \emph{each separate} slice $\liftVar_\theta$ in the subsequent result.}

\begin{thm}[Error estimates for $\liftVar\in\Mp(\alldirs\times\projdomdyn)$]\label{thm:estimateMu2}
Let $\weight:\alldirs\to[0,\infty)$ be continuous with $\int_\alldirs\weight\wrt\mdirs=1$ and assume there is $\goodtimes\subset\alltimes$
such that $u_t^\dagger$ is $(\hat\kappa,\hat\mu,\hat R)$-\added{stably} reconstructible via $\hat v_t^\dagger,\hat v_t^s$ from measurements with $\fopstat_t$ for all $t\in\goodtimes$.
Then, abbreviating $p_i^\theta=(\theta\cdot x_i,\theta\cdot v_i)$, the solution $(u,\liftVar)$ to \problemTag[\sqrt\delta]{\data^\delta} satisfies
\begin{align*}
\int_\alldirs\weight(\theta)|\liftVar_\theta|(\projdomdyn\setminus B_R(\supp\liftVar_\theta^\dagger))\wrt\mdirs
&\leq\frac1{\kappa R^2}\left(\frac32+2\max_{t\in\goodtimes}|\hat v_t^\dagger|^2\right)\sqrt\delta,\\
\sum_{i=1}^N\int_\alldirs\!\!\weight(\theta)\!\!\int_{B_R(\{p_i^\theta\})}\!\!\!\!\dist(p,p_i^\theta)^2\wrt|\liftVar_\theta|(p)\wrt\mdirs
&\leq\frac1{\kappa}\left(\frac32+2\max_{t\in\goodtimes}|\hat v_t^\dagger|^2\right)\sqrt\delta,\\
\sum_{i=1}^N\left|\int_\alldirs\weight(\theta)(\liftVar_\theta-\liftVar^\dagger_\theta)(B_R(\{p_i^\theta\}))\wrt\mdirs\right|
&\leq\frac{1+\kappa R^2}{\kappa R^2}\left(2\!+\!2\max_{t\in\goodtimes}|\hat v_t^\dagger|^2\right)\sqrt\delta
\!+\!\frac1{2R^2}\max_{t\in\goodtimes}\unbalancedWasserstein{R}(u_t,u^\dagger_t)
\end{align*}
for $\kappa=\sigma\hat\kappa$ and any $R\leq\min\{\hat R,\minSeparation/3,\ghostDistance\}/\sqrt{1+\max_{t\in\goodtimes}t^2}$,
where $\sigma$ is the smallest eigenvalue of $\frac1{\card\goodtimes}\sum_{t\in\goodtimes}{1\choose-t}\otimes{1\choose-t}$
and $\minSeparation,\ghostDistance\geq0$ are such that $\liftVar_\theta^\dagger$ for any $\theta\in\supp\weight$ contains no $\minSeparation$-coincidence or $\ghostDistance$-ghost particle with respect to $\goodtimes$.
\end{thm}
\begin{proof}
We split $\approximation=(u,\liftVar)$ into $\approximation_1=\liftVar$ and $\approximation_2=u$ and consider the regularizer
$$\regularizer(u,\liftVar)=\|\liftVar\|_\weight+\iota_{+}(\liftVar)+\iota_{+}(u).$$
As in the proof of \cref{thm:estimateMu} define $v^\dagger\in(H^*)^{|\alltimes|}$
via $v^\dagger_t=-\hat v^\dagger_t/\card{\goodtimes}$ for all $t\in\goodtimes$ and $v^\dagger_{\tilde t}=0$ else.
The dual variables are now constructed as weighted averages (with weight $\weight$) of the dual variables from \cref{thm:estimateMu}.
Indeed, we define $q^\dagger\in C(\alldirs\times\projdomdyn)^\domt$ as
\begin{equation*}
q^\dagger_{t}(\theta,x)=\weight(\theta)(1-\hat\kappa\min\{\bar R,\dist(x,\supp\Rs_\theta u_t^\dagger)\}^2)/\card{\goodtimes}
\quad\text{for }t\in\goodtimes
\end{equation*}
and $q^\dagger_{t}=0$ for $t\notin\goodtimes$, where $\bar R=\min\{\hat R,\minSeparation/3,\ghostDistance\}$.
Also, let $s_i=\sgn\int_\alldirs\weight(\theta)(\liftVar_\theta-\liftVar^\dagger_\theta)(B_R(\{p_i^\theta\}))\wrt\mdirs$, $i=1,\ldots,N$,
where the slices $\liftVar_\theta$ are defined via the relation $\liftVar=\mdirs\prodm\theta\liftVar_\theta$
(which is possible by \cref{lem:projected_constraint_decomposition_mixed_model}).
With this $s_i$ we define $q\in C(\alldirs\times\projdomdyn)^\domt$
as $q_{t}=0$ for $t\notin\goodtimes$ and, for $t\in\goodtimes$, as
\begin{equation*}
q_{t}(\theta,x)=\weight(\theta)s_i/\card{\goodtimes}
\quad\text{if }x\in B_{\bar R}(\theta\cdot(x_i+tv_i)),
\end{equation*}
continuously extended onto the rest of $\alldirs\times\projdomdyn$ with absolute value no larger than $\weight(\theta)/\card{\goodtimes}$.
Now analogously to the proof of \cref{thm:estimateMu} it follows that $\optimalDual={v^\dagger\choose q^\dagger}$ and $w={0\choose q}$ satisfy
\begin{align*}
-K^*\optimalDual=-((K^*\optimalDual)_1,(K^*\optimalDual)_2)&\in\partial\regularizer(\groundTruth),\\
-(K^*w)_1&\in\partial\|\cdot\|_\weight\left(\sum_{i=1}^Ns_i\hd^{d-1}\restr\{p_i^\theta\,|\,\theta\in\alldirs\}\right),\\
-(K^*\optimalDual)_1(\theta,p)&\leq\weight(\theta)(1-\kappa\min\{R,\dist(p,\supp\liftVar_\theta^\dagger)\}^2),\\
-(K^*w)_1(\theta,p)&=s_i\weight(\theta)\quad\text{for all }p\in B_R(\{p_i^\theta\}).
\end{align*}
Thus, using (in this order) \cref{thm:massFromBregman2}, \cref{thm:BregmanEstimates} and \eqref{eqn:dualFidelity} we conclude
\begin{multline*}
\int_\alldirs\weight(\theta)|\liftVar_\theta|(\projdomdyn\setminus B_R(\supp\liftVar_\theta^\dagger))\wrt\mdirs
\leq\frac1{\kappa R^2}\BregmanDistance{\|\cdot\|_\weight}{-(K^*\optimalDual)_1}(\liftVar,\liftVar^\dagger)
\leq\frac1{\kappa R^2}\BregmanDistance{G}{-K^*\optimalDual}(\approximation,\groundTruth)\\^
\leq\left(3\delta+\fidelity{\noisyData}^*(2\sqrt\delta\optimalDual)+\fidelity{\noisyData}^*(-2\sqrt\delta\optimalDual)\right)/(2\kappa R^2\sqrt\delta)
=\frac{3+4|v^\dagger|^2}{2\kappa R^2}\sqrt\delta
\leq\frac{3+4\max_{t\in\goodtimes}|\hat v^\dagger_t|^2}{2\kappa R^2}\sqrt\delta,
\end{multline*}
proving the first inequality of the statement.
The second and third inequality follow in exactly the same way (using $\mu=0$ in \cref{thm:massFromBregman2}),
where the third needs one additional estimate for $\langle-(K^*w)_1,\liftVar-\liftVar^\dagger\rangle$:
We employ $\langle-(K^*w)_1,\liftVar-\liftVar^\dagger\rangle=\langle-K^*w,(u,\liftVar)-(u^\dagger,\liftVar^\dagger)\rangle+\langle(K^*w)_2,u-u^\dagger\rangle$,
where the first summand can be estimated via \cref{thm:BregmanEstimates} and \eqref{eqn:dualFidelity} as
\begin{equation*}
\langle-K^*w,(u,\liftVar)-(u^\dagger,\liftVar^\dagger)\rangle
\leq2(1+|v^\dagger|)\sqrt\delta
\leq2(1+\max_{t\in\goodtimes}|\hat v^\dagger_t|)\sqrt\delta,
\end{equation*}
and for the second summand we note that by construction $1=|\card{\goodtimes}(\Rfm^*q)_t(x)|$ for $x\in B_{\bar R}(\supp u_t^\dagger)$ and any $t\in\goodtimes$. Thus, 
\begin{multline*}
\langle(K^*w)_2,\approximation_2-\groundTruth_2\rangle
=\sum_{t\in\goodtimes}\langle(\Rfm^*q)_t,u_t-u^\dagger_t\rangle
\leq\frac1{\card{\goodtimes}}\sum_{t\in\goodtimes}\bigg(|u_t-u^\dagger_t|(\Omega\setminus B_R(\supp u_t^\dagger))\\
+\sum_{x\in\supp u_t^\dagger}|(u_t-u_t^\dagger)(B_R(x))|\bigg)
\leq\frac1{2R^2\card{\goodtimes}}\sum_{t\in\goodtimes}\unbalancedWasserstein{R}(u_t,u^\dagger_t),
\end{multline*}
where we exploited that in the unbalanced Wasserstein divergence $\unbalancedWasserstein{R}$ mass is never transported further than $R$.
\end{proof}

\added{\label{txt:itemeProofLogic13}\lookUp{\ref{iteme}}As promised,}
this result can now be turned into an unbalanced Wasserstein estimate for the slices of $\liftVar$,
\added{a yet different variant of \cref{thm:estimateMu}}.

\begin{thm}[Error estimates for slices of $\liftVar\in\Mp(\alldirs\times\projdomdyn)$]\label{thm:liftVarEstimateMeasureModel}
The statement of \cref{thm:estimateMu} holds for $\mdirs$-almost every $\theta$ if \PInoisyProblemTag[\sqrt\delta]{\data^\delta} is replaced with \problemTag[\sqrt\delta]{\data^\delta}.
\end{thm}
\begin{proof}
Based on \cref{thm:unbalancedWasserstein},
it suffices to show that one may replace $\weight\mdirs$ in \cref{thm:estimateMu2} with a Dirac measure $\delta_\theta$ for $\mdirs$-almost all $\theta$.
We only show this for the first inequality of \cref{thm:estimateMu2} (the others follow analogously).
Now let
\begin{align*}
\minSeparation(\theta)&=\sup\{\minSeparation\geq0\,|\,\liftVar_\theta^\dagger\text{ contains no $\minSeparation$-coincidence with respect to }\goodtimes'\},\\
\ghostDistance(\theta)&=\sup\{\ghostDistance\geq0\,|\,\liftVar_\theta^\dagger\text{ contains no $\ghostDistance$-ghost particle with respect to }\goodtimes'\}
\end{align*}
and note that both are continuous in $\theta$.
Thus, so is $R(\theta)=\min\{\hat R,\minSeparation(\theta)/3,\ghostDistance(\theta)\}/\sqrt{1+\max_{t\in\goodtimes}t^2}$.
We will show that for any nonnegative continuous function $\hat R(\theta)\leq R(\theta)$ and any $\rho<1$
there is a piecewise constant function $\bar R:\sphere^{d-1}\to[0,\infty)$ with $\rho^2\hat R(\theta)\leq\bar R(\theta)\leq\hat R(\theta)$ and such that
$|\liftVar_\theta|(\projdomdyn\setminus B_{\bar R(\theta)}(\supp\liftVar_\theta^\dagger))\leq\frac1{\kappa\bar R(\theta)^2}(\frac32+2\max_{t\in\goodtimes}|\hat v_t^\dagger|^2)\sqrt\delta$
for $\mdirs$-almost every $\theta$, which then implies the desired result due to the arbitrariness of $\rho$ and $\hat R$.
To this end, fix some $\rho\in(0,1)$ and let $\epsilon>0$ be arbitrary.
Abbreviate
\begin{equation*}
S=\{\theta\in\alldirs\,|\,\hat R(\theta)=0\},
\end{equation*}
then it suffices to prove the inequality for $\mdirs$-almost every $\theta$ in $\alldirs\setminus S$ (since on $S$ there is nothing to show).
Due to the regularity of $\mdirs$ we can pick a compact $C\subset\alldirs\setminus S$ such that $\mdirs((\alldirs\setminus S)\setminus C)<\epsilon$.
We will show the inequality for $\mdirs$-almost every $\theta$ in $C$, then the result for $\alldirs\setminus S$ follows from the arbitrariness of $\epsilon$.
For each $\theta\in C$ define
\begin{equation*}
r(\theta)=\sup\{r\in\R\,|\,\rho^2\hat R(\theta')<\rho\hat R(\theta)<\hat R(\theta')\text{ for all $\theta'\in\sphere^{d-1}$ with }\dist(\theta,\theta')<r\},
\end{equation*}
then $r(\theta)>0$ due to the continuity of $\hat R(\theta)$.
Consider the open $r(\hat\theta)$-ball $B\subset\sphere^{d-1}$ around $\hat\theta\in C$.
It suffices to prove the desired inequality for $\mdirs$-almost every $\theta\in B$ (choosing $\bar R(\theta)=\rho\hat R(\hat\theta)$),
since the compactum $C$ can be covered with finitely many of these balls and the function $\bar R(\theta)$ can then simply be chosen constant on each ball
(minus the finitely many balls where it was already defined).
Now on $B$, by \cref{thm:estimateMu2} we have
\begin{equation*}
\int_\alldirs\weight(\theta)\left[|\liftVar_\theta|(\projdomdyn\setminus B_{R}(\supp\liftVar_\theta^\dagger))
-\frac1{\kappa R^2}\left(\frac32+2\max_{t\in\goodtimes}|\hat v_t^\dagger|^2\right)\sqrt\delta\right]\wrt\mdirs
\leq0
\end{equation*}
for $R=\bar R(\theta)=\rho\hat R(\hat\theta)$
and for any nonnegative continuous $\weight:B\to[0,\infty)$ of compact support.
Therefore the square-bracketed term must be nonpositve for $\mdirs$-almost every $\theta\in B$, which is what we needed to prove.
\end{proof}

The error estimate for all remaining snapshots finally stays the same for model \problemTag[\sqrt\delta]{\data^\delta} as for model \PInoisyProblemTag[\sqrt\delta]{\data^\delta}.

\begin{thm}[Error estimates for bad snapshots]\label{thm:estimateBadUMeasureModel}
\Cref{thm:estimateBadU} still holds when replacing model \PInoisyProblemTag[\sqrt\delta]{\data^\delta} with \problemTag[\sqrt\delta]{\data^\delta}.
\end{thm}
\begin{proof}
The left- and right-hand side of the estimate depend continuously on $\bar R,\bar\kappa$
so that it suffices to prove the estimate for any $\bar\kappa,\bar R$ strictly smaller than their values from \cref{thm:estimateBadU}.
As in \cref{thm:estimateBadU} we consider the splitting of $\approximation=(u,\liftVar)$ into $\approximation_1=u_t$ and $\approximation_2=((u_{\tilde t})_{\tilde t\neq t},\liftVar)$ as well as the regularizer choice
$$\regularizer(u,\liftVar)=\|u_t\|_++\iota_+((u_{\tilde t})_{\tilde t\neq t},\liftVar).$$
We will also use almost exactly the same dual variables as in \cref{thm:estimateBadU}, only we will smooth them out a little in $\theta$-direction.
To this end, for a finite set $\gooddirs'\subset\sphere^{d-1}$ let
\begin{align*}
\minSeparation(\gooddirs')&=\sup\{\minSeparation\geq0\,|\,u_t^\dagger\text{ contains no $\minSeparation$-coincidence with respect to }\gooddirs'\},\\
\ghostDistance(\gooddirs')&=\sup\{\ghostDistance\geq0\,|\,u_t^\dagger\text{ contains no $\ghostDistance$-ghost particle with respect to }\gooddirs'\}
\end{align*}
and let $\sigma(\gooddirs')$ be the smallest eigenvalue of $\frac1{\card\gooddirs'}\sum_{\theta\in\gooddirs'}\theta\otimes\theta$.
Furthermore, let $\kappa(\gooddirs')$ and $R(\gooddirs')$ be the values of $\kappa$ and $R$ from the statement of \cref{thm:estimateBadU} for the choice $\gooddirs=\gooddirs'$.
Now let $\zeta>0$ be small enough and pick $\tilde\kappa,\tilde R$ such that the balls $B_\zeta(\{\theta\})$ are disjoint for all $\theta\in\gooddirs$
and such that $\bar\kappa\leq\sigma(\gooddirs')\tilde\kappa$, $\tilde\kappa\leq\kappa(\gooddirs')/(1+t^2)$,
$\bar R\leq\min\{(1+t^2)R(\gooddirs'),\tilde R/\sqrt{1+t^2},\ghostDistance(\gooddirs')\}$ and $\tilde R\leq\minSeparation(\gooddirs')/3$
for any perturbation $\gooddirs'$ of $\gooddirs$ that perturbs no $\theta\in\gooddirs$ by more than distance $\zeta$
(this is possible since we assume $\bar\kappa$ and $\bar R$ strictly smaller than their values in \cref{thm:estimateBadU}
and since $\minSeparation(\gooddirs')$, $\ghostDistance(\gooddirs')$, $\sigma(\gooddirs')$, $\kappa(\gooddirs')$ and $R(\gooddirs')$ depend continuously on $\gooddirs'$).
Furthermore, for $\bar\theta\in\gooddirs$ let $\weight_{\bar\theta}:\alldirs\to[0,\infty)$ be continuous with $\supp\weight_{\bar\theta}\subset B_\zeta(\{\bar\theta\})$ and $\int_\alldirs\weight_{\bar\theta}\mdirs=1$.
Now define $q^\dagger\in C(\alldirs\times\projdomstat)^{\domt}$ as
\begin{equation*}
q^\dagger_{t}(\theta,x)=-\weight_{\bar\theta}(\theta)(1-\tilde\kappa\min\{\bar R,\dist(x,\supp\mv^1_t\liftVar^\dagger_\theta)\}^2)/\card{\gooddirs}
\quad\text{for }\theta\in B_\zeta(\{\bar\theta\}),\bar\theta\in\gooddirs
\end{equation*}
and $q^\dagger_{\tilde t}=0$ for $\tilde t\neq t$.
Also, let $s_i=\sgn(u_t-u_t^\dagger)(B_{\bar R}(\{x_i+tv_i\}))$, $i=1,\ldots,N$, and define $q\in C(\alldirs\times\projdomstat)^\domt$
as $q_{\tilde t}=0$ for $\tilde t\neq t$ and, otherwise,
\begin{equation*}
q_{t}(\theta,x)=-\sum_{\bar\theta\in\gooddirs}\weight_{\bar\theta}(\theta)s_i/\card{\gooddirs}
\quad\text{if }x\in B_{\tilde R}(\theta\cdot(x_i+tv_i)),
\end{equation*}
continuously extended onto the rest of $\alldirs\times\projdomstat$ with absolute value no larger than $\weight_{\bar\theta}(\theta)/\card{\gooddirs}$.
Finally, let $w^\theta\in Y^*$ for $\theta\in\alldirs$ denote the first dual variables constructed in \cref{thm:estimateMu},
and define
\begin{equation*}
w^\dagger={0\choose q^\dagger}+\frac1{\card{\gooddirs}}\sum_{\bar\theta\in\gooddirs}\int_{\alldirs}\weight_{\bar\theta}(\theta)w^\theta\wrt\mdirs(\theta),
\quad
w={0\choose q}.
\end{equation*}
The remainder of the proof (checking the properties of the constructed dual variables and estimating one dual pairing) proceeds analogously to the proof of \cref{thm:estimateBadU}.
\end{proof}

\Cref{thm:errorEstimatesSummary} for \problemTag[\sqrt\delta]{\data^\delta} now follows from combining \cref{thm:errorGoodSnapshots,thm:liftVarEstimateMeasureModel,thm:estimateBadUMeasureModel}.

\section{Discretization and numerical validation}\label{sec:numerics}
Our goal is now to develop a numerical method in order to solve the dimension-reduced problem \ref{eq:dim_reduced_general}.\removed{ or \ref{eq:dim_reduced_product_only}.} Such a method has to translate the dimension reduction from the continuous formulation to a discrete formulation, which then is more efficient to solve than comparable discretizations of the full-dimensional problem. A straightforward discretization of \ref{eq:dim_reduced_general} is achieved by simply gridding the parameter space and replacing Radon measures on the continuous parameter space by vectors representing the weights of a discrete measure supported in (or around) the grid points.

In this section, we assume that the sets $\domt$ and $\alldirs$ are finite. Some choices of these parameter sets might be better suited than others, see \cref{sec:numDirections}.
To be able to compare different choices of $\domt$ for our method without changing datasets or domains, we define a slightly less restrictive domain in phase space:
We require particles to stay in $\domstat$ only during the measurement times $\alltimes$ instead of $\domt$, i.\,e., instead of $\domdyn$ as defined in \eqref{eq:domdyn_definitoin} we use
\begin{equation*}
    \domdynnum\coloneqq\set{(x,v)\in\R^d\times\R^d\given x+tv\in\domstat\text{ for all } t \in \alltimes}.
\end{equation*}

The code to reproduce the results of this paper will be made available at \url{https://github.com/alexschlueter/supermops}.

\subsection{Gridding for snapshots and position-velocity projections}
The first challenge is that problem \ref{eq:dim_reduced_general} contains multiple Radon measures as variables that should be discretized consistently, namely measures $\liftVar_\theta$ for directions $\theta\in\alldirs$ as well as snapshots $u_t$ for times $t\in\domt$. For each of these variables we need to generate a corresponding grid.

One option for $\liftVar$ would be to use a single grid on $\projdomdyn$ in order to discretize all $\liftVar_\theta$ equally, independent of the direction $\theta$.
This however would result in wasted degrees of freedom, because the $\liftVar_\theta$ occupy different parts of $\projdomdyn$.
For this reason we will generate individual grids for each direction $\theta$.
These should be compatible with the (non-discretized) phase space domain $\domdynnum$ in the following sense:
If $\lambda\in\M(\domdynnum)$ is a high-dimensional phase space measure, the grid for $\liftVar_\theta$ should cover enough of $\projdomdyn$ such that the joint Radon transform of $\lambda$ along direction $\theta$ can still be represented, thus at least $\Rj_\theta(\domdynnum)\subset\projdomdyn$.
Analogous considerations hold for $u$ so that we also use individual grids for all $u_t$.

In the following we will consider
\begin{equation*}
\domstat=[0,1]^d
\end{equation*}
and furthermore centre without loss of generality the times $\alltimes$ around $0$ so that there exists $\maxT>0$ with
\begin{equation*}
[\min\alltimes,\max\alltimes]=[-\maxT,\maxT].
\end{equation*}
Based on this we now compute the domains of all $\liftVar_\theta$ and $u_t$ in order to generate tailored grids on these.

\begin{prop}[Domain of $\liftVar_\theta$]\label{thm:dom_mu}
    For a given $\thins$ define $s_+\coloneqq\sum_{j:\theta_j>0}\theta_j$ and $s_-\coloneqq\sum_{j:\theta_j<0}\theta_j$.
    The set $\Rj_\theta(\domdynnum)\subset\projdomdyn$ is a parallelogram in $\R^2$ with vertices
    \begin{equation*}
        \begin{pmatrix}
            s_- \\ 0
        \end{pmatrix},
        \begin{pmatrix}
            s_+ \\ 0
        \end{pmatrix},
        \begin{pmatrix}
            \frac{s_++s_-}{2} \\ \frac{s_+-s_-}{2\maxT}
        \end{pmatrix},
        \begin{pmatrix}
            \frac{s_++s_-}{2} \\ \frac{s_--s_+}{2\maxT}
        \end{pmatrix}.
    \end{equation*}
\end{prop}
\begin{proof}
    Since we assume $\domstat=[0,1]^d$, we have
    \begin{align*}
        (x,v)\in\domdynnum&\iff\forall j=1,\dotsc,d:0\leq x_j+\maxT v_j\leq 1\land 0\leq x_j-\maxT v_j\leq 1 \\
        &\iff \forall j=1,\dotsc,d: (x_j,v_j)\in\conv\set*{
        \begin{pmatrix}
            0 \\ 0
        \end{pmatrix},
        \begin{pmatrix}
            1 \\ 0
        \end{pmatrix},
        \begin{pmatrix}
            1/2 \\ 1/(2\maxT)
        \end{pmatrix},
        \begin{pmatrix}
            1/2 \\ -1/(2\maxT)
        \end{pmatrix}
        }.
    \end{align*}
    Denote the convex hull above by $\domdynone\subset\R^2$. For a given $\thins$ we have, using the Minkowski sum of sets,
    \begin{align*}
        \Rj_\theta(\domdynnum)&=\set{(\theta\cdot x, \theta\cdot v)\given (x,v)\in\domdynnum}\\
        &=\bigg\{(y,w)\in\R^2\,\bigg| \exists(x_j,v_j)\in\domdynone,j=1,\dotsc,d:(y,w)=\sum_{j=1}^d\theta_j(x_j,v_j)\bigg\}\\
        &=\theta_1\domdynone +\dotsb + \theta_d\domdynone\\
        &=\sum_{j:\theta_j>0}(\theta_j\domdynone) + \sum_{j:\theta_j<0}(\theta_j\domdynone) .
    \end{align*}
    By convexity of $\domdynone$, we can use the distributivity for Minkowski sums after introducing a minus sign,
    \begin{equation*}
        \sum_{j:\theta_j>0}(\theta_j\domdynone) + \sum_{j:\theta_j<0}(\theta_j\domdynone)
        =\sum_{j:\theta_j>0}(\theta_j\domdynone) + \sum_{j:\theta_j<0}((-\theta_j)(-\domdynone))
        =s_+\domdynone + (-s_-)(-\domdynone).
    \end{equation*}
    Note that $-\domdynone=\domdynone-(1,0)$ so that
    \begin{align*}
        s_+\domdynone + (-s_-)(-\domdynone)
        =s_+\domdynone + (-s_-)\domdynone+(s_-,0)
        =(s_+-s_-)\domdynone+(s_-,0).
    \end{align*}
    Applying the scaling and translation in the last expression to the four points inside the definition of $\domdynone$, we obtain the stated result.
\end{proof}
\begin{prop}[Domain of $u_t$]\label{thm:dom_u}
     Let $t\in\R$ be given.
     \begin{enumerate}
         \item If $\abs{t}\leq\maxT$, we have $\mv^d_t(\domdynnum)=[0,1]^d$.
         \item If $\abs{t}>\maxT$, we have $\mv^d_t(\domdynnum)=\left[\frac{1}{2}-\frac{\abs{t}}{2\maxT},\frac{1}{2}+\frac{\abs{t}}{2\maxT}\right]^d$.
     \end{enumerate}
\end{prop}
\begin{proof}
    Write
    \begin{equation*}
        S\coloneqq\set*{
        \begin{pmatrix}
            0 \\ 0
        \end{pmatrix},
        \begin{pmatrix}
            1 \\ 0
        \end{pmatrix},
        \begin{pmatrix}
            1/2 \\ 1/(2\maxT)
        \end{pmatrix},
        \begin{pmatrix}
            1/2 \\ -1/(2\maxT)
        \end{pmatrix}
        }
    \end{equation*}
    and define again $\domdynone=\conv S$. We have
    \begin{equation*}
        \mv^d_t(\domdynnum)=\set{x+tv\given (x,v)\in\domdynnum}
        =\bigtimes_{j=1}^d\mv^1_t(\domdynone)
    \end{equation*}
    and by linearity of $\mv^1_t$ we know that
    \begin{equation*}
\mv^1_t(\domdynone)=\mv^1_t(\conv S)=\conv\mv^1_t(S)
=\conv\set*{0, 1, \frac{1}{2}+\frac{t}{2\maxT},\frac{1}{2}-\frac{t}{2\maxT}}\,,
    \end{equation*}
    from which the claim follows easily.
\end{proof}

Now knowing the exact domains, we generate regular grids for each variable, that is, grids $\gridu{t}$ for each $t\in\domt$ and grids $\gridmu{\theta}$ for each $\theta\in\alldirs$. Examples %
can be seen in \cref{fig:grid_proj}. The grids $\gridmu{\theta}$ are generated by applying an appropriate affine transform to a standard rectangular grid as suggested by \cref{thm:dom_mu}.
We employ the same number of $\nod$ grid cells along each dimension to obtain a more or less isotropic resolution (in our numerical experiments, $\nod=100$ or $200$).
Each variable $u_t$, $t\in\domt$, and $\liftVar_\theta$, $\theta\in\alldirs$, is now discretized as a piecewise constant function on the corresponding grid,
whose masses in the different grid cells (function value times cell volume) are collected in the vectors $\uvec{t}\in\R^{\nod^d}$, $t\in\domt$, and $\muvec{\theta}\in\R^{\nod^2}$, $\theta\in\alldirs$.

\subsection{Discretized operators and optimization problem}
\begin{figure}[ht]
    \centering
    \input{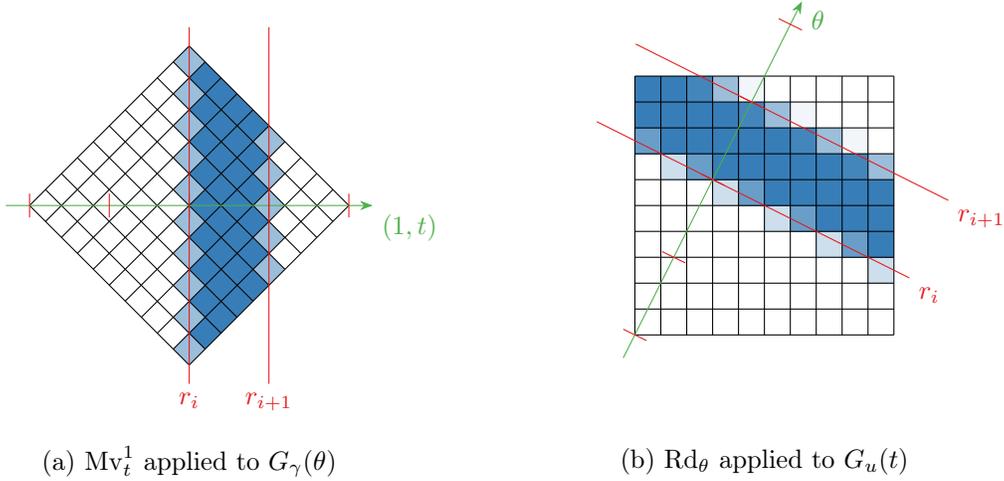}
    \caption{Example grids and projections. The green line in projection direction $\theta$ is subdivided into red bins corresponding to the grid $\gridr{\theta}{t}$. The contribution of a grid cell to the projection onto a red bin $[r_i, r_{i+1}]$ is determined by its relative area of intersection (visualized by opacity) with the strip orthogonal to $\theta$.}
    \label{fig:grid_proj}
\end{figure}

Next, the operators $\Rs_\theta$ and $\mv_t^1$ need to be discretized. We restrict our attention to the case $d=2$ for simplicity and note that, for two spatial dimensions, these operators are identical up to a rescaling: For every $t\in\R$, we have $\mv_t^1=\pf{[x\mapsto\sqrt{1+t^2}x]}\Rs_{\theta(t)}$, where $\theta(t)\coloneqq (1,t)/\sqrt{1+t^2}$. Thus it will be sufficient to only describe the implementation of the operator $\Rs_\theta$ in the following.

For every direction $\theta\in\alldirs$, we need to discretize the operation $\Rs_\theta u_t$, where $u_t$ is represented by the vector $\uvec{t}$ of function values on the grid $\gridu{t}$.
The result of the projection $\Rs_\theta u_t$ is a measure on $\projdomstat\subset\R$,
whose domain $\Rs_\theta(\mv^d_t(\domdynnum))=\mv^1_t(\Rj_\theta(\domdynnum))\subset\projdomstat$ depends on $t$ and $\theta$ and is again discretized via a one-dimensional equispaced grid $\gridr{\theta}{t}$.
To have a comparable resolution as for the variables $u_t$ and $\liftVar_\theta$ we pick $\gridr\theta t$ to have the same number $\nod$ of intervals as the other grids have along each direction.
Again, measures on $\projdomstat$ are discretized as piecewise constant functions on that grid.

After discretization, the Radon transform $\Rs_\theta u_t$ will be represented by $\matu{t}{\theta}\uvec{t}$ for a matrix $\matu{t}{\theta}\in\R^{\nod\times\nod^2}$.
(In a completely analogous manner the move operator $\mv_t^1$ in $\mv_t^1\gamma_\theta$ will be represented by a matrix $\matmu{\theta}{t}\in\R^{\nod\times\nod^2}$).
The entry of the $i$-th row and $j$-th column in $\matu{t}{\theta}$ is given by the relative area of intersection between the strip
\begin{equation*}
    \set{x\in\R^2\given r_i\leq\thdot{x}\leq r_{i+1}}
\end{equation*}
(where $[r_i, r_{i+1}]$ represents the $i$-th grid cell of the grid $\gridr{\theta}{t}$)
and the $j$-th cell of the grid $\gridu{t}$ (cf.\ \cref{fig:grid_proj}).

Finally, we take $\measmat{t}$ to be appropriate matrices discretizing the forward operators $\fopstat_t$ on the grids $\gridu{t}$ for all $t\in\alltimes$.
In our experiments we consider as forward operator the truncated discrete Fourier transform \added{$\ftrunc$} on $\Omega=[0,1]^2$ from \cref{ex:truncatedFourier}.
To compute the entries of $\measmat{t}$ we simply apply $\fopstat_t$ to Dirac measures located at the centres of the grid cells in $\gridu{t}$ for each time $t\in\alltimes$.

In principle one would now obtain a discretized version of \ref{eq:dim_reduced_product_only} by replacing all variables and operators by their discretized counterparts.
However, the constraints $\Rs_\theta u_t=\mv_t^1\liftVar_\theta$ should not be enforced exactly in the discretized setting
since this is likely to introduce discretization artefacts.
Therefore we replace them by a bound $\redconsbnd>0$ on the violation so that the discretized problem reads
\begin{equation}\label{eqn:discrprob}
\begin{split}
\min_{\substack{
\muvec{\theta}\in\R^{\nod^2}, \theta\in\alldirs\\
\uvec{t}\in\R^{\nod^d}, t\in\domt
}}\,
\sum_{\theta\in\alldirs}\norm{\muvec{\theta}}_1+\sum_{t\in\domt}\norm{\uvec{t}}_1
+\frac1{2\alpha}\sum_{t\in\alltimes}\norm{\measmat{t}\uvec{t}-\data_t}_2^2\\
\quad\text{ such that }
\left(\sum_{\theta\in\alldirs, t\in\domt}\norm{\matmu{\theta}{t}\muvec{\theta} - \matu{t}{\theta}\uvec{t}}_2^2\right)^{1/2}\leq\tau.
\end{split}
\end{equation}
A coarse parameter search has shown a value of $\redconsbnd=0.001$ to be a good choice, which we use in all our experiments.
The resulting second-order cone program is solved using the commercial MOSEK solver \cite{ApS2019} called through the Python library CVXPY \cite{diamond2016cvxpy,agrawal2018rewriting}.

\subsection{Application to truncated Fourier measurements}
As a model super-resolution problem, we consider the case of truncated Fourier measurements \added{$\ftrunc$} from \cref{ex:truncatedFourier} in $d=2$ dimensions, where we set the frequency cutoff to $\maxFrequency=2$.
This setting in particular serves to benchmark our dimension reduction against the full-dimensional dynamic reconstruction proposed by Alberti et al.

\paragraph{Measurement times.}
We are interested in evaluating the performance of the method for different numbers of measurement times $\alltimes$. For simplicity, we choose times $\alltimes\subset[-1,1]$ centred around 0 as follows: Given a $K\in\N$, choose $\alltimes\coloneqq\set{k/K\given k\in\set{-K,\dotsc,K}}$. The values of $K$ used in the following are $K=1,2,3$ resulting in 3, 5 and 7 times respectively.

\paragraph{Datasets.}
To compare different parameter choices for the algorithm and to benchmark it against other methods, we generate three datasets \dataset{3}, \dataset{5}, \dataset{7} consisting of 2000 particle configurations in $d=2$ dimensions, one dataset for each choice of times $\alltimes$.
Each configuration within a dataset consists of $4\leq N\leq 20$ particles with randomly chosen positions and velocities $(x_i,v_i)\in\domdynnum$ and weights $m_i$ drawn from the uniform distribution on $[0.9, 1.1]$.

We define the minimal dynamic separation of a particle configuration $\dsupp=\set{(x_i,v_i)\given i=1,\dotsc,N}$ as
\begin{equation*}
    \dynsep(\dsupp)\coloneqq\min_{t\in\alltimes}\min_{i\neq j}\norm{x_i+tv_i-(x_j+tv_j)}.
\end{equation*}
Since we are interested in testing the reconstruction methods for a range of separation values, we want configurations with low and high separations to be equally well represented in each dataset. To this end, we use a \emph{rejection-sampling} method adapted from \cite{AlbertiAmmariRomeroWintz2019}, which roughly works by drawing configurations from the uniform distribution on $\domdynnum$ and accepting or rejecting the new configuration based on an acceptance probability. For our datasets this probability is chosen as to make the dynamic separation $\dynsep$ uniformly distributed on $[0,0.1]$.

\paragraph{Postprocessing and error measures.}
Because static reconstruction will be one of the methods we compare against, we only consider the reconstruction results for the centre time $t=0$, that is, the reconstructed snapshot $u_0$.
We use two different error measures to evaluate the reconstruction results.
The first is simply the unbalanced Wasserstein divergence $\unbalancedWasserstein[2]{R}(u_0^\dagger,u_0)$ as defined in \cref{sec:unbalancedTransport}, where $u_0^\dagger$ is the ground truth as before and $u_0$ is the measure obtained by placing a Dirac measure at each cell centre of $\gridu{0}$ with weights as given in $\uvec{0}$. The radius is chosen as $R=0.05$ to balance transport and mass creation inside the unit square.
The second error measure uses a clustering method:
Since usually the ground truth particle is not exactly representable on the discrete grid, the solution to the discretized problem will typically contain a cluster of positive weights near a true particle position.
Therefore, in a post-processing step, we first set all weights with values below a threshold $\clusterthresh=0.1$ to zero. Afterwards, we cluster neighbouring non-zero weights and create a list of detected particles, where each cluster is assumed to be a detected particle with its position equal to the cluster's centre of mass. Now, a configuration of particles is seen as correctly reconstructed if there exists a pairing between true and reconstructed particles at time $t=0$ such that the paired particles have distance less than $0.01$ %
and there remain no unpaired particles.

\paragraph{Methods and parameters.}
Our goal is to compare the following methods:
\begin{itemize}
    \item \textbf{Static:} We reconstruct $u_0$ by solving
\begin{equation*}%
    \min_{u\in\Mp(\domstat)}\mnorm{u}+
    \frac1{2\alpha}\norm{\fopstat_0 u-\data_0}_2^2
\end{equation*}
    (the fidelity-term version of the static reconstruction problem \eqref{eqn:staticModel}),
    disregarding the data $\data_t$ for $t\in\alltimes\setminus\set{0}$.
    In contrast to \eqref{eqn:discrprob}, the discretization of this problem is straightforward.
    In practice, we reuse our implementation of the dimension-reduced problem \eqref{eqn:discrprob}, setting $\alldirs=\varnothing$.

    \item \textbf{ADCG:} We use the implementation of Alberti et al.\ in \cite{AlbertiAmmariRomeroWintz2019}, which applies the Alternating Descent Conditional Gradient method (ADCG) from \cite{Boyd2017alternating_mh} to the full-dimensional dynamic formulation
\begin{equation*}
\min_{\lambda \in \Mp(\domdynnum)} \| \lambda \|_\M \quad \text{such that } \fopstat_t\,\mv^d_t \lambda =  \data_t \quad \text{for all }t \in \alltimes.
\end{equation*}
In short, this method is an extension of the well-known Frank--Wolfe method and works by alternating between adding new source points and differentiable optimization of source position and weights.

\item \textbf{Dimred.\ low}: Our method with $\abs{\alldirs}=3$ and $\abs{\domt\setminus\alltimes}=0$.
\item \textbf{Dimred.\ mid}: Our method with $\abs{\alldirs}=5$ and $\abs{\domt\setminus\alltimes}=3$.
\item \textbf{Dimred.\ high}: Our method with $\abs{\alldirs}=10$ and $\abs{\domt\setminus\alltimes}=7$.
\end{itemize}

The directions $\alldirs$ for the dimension-reduced methods are chosen as equidistributed points on the right half of the unit circle $\sphere^{1}$ with maximum possible spacing (depending on the number of directions), $\alldirs=\{\exp(\ii\pi(\frac j{|\alldirs|}-\frac12))\,|\,j=0,\ldots,|\alldirs|-1\}$ when expressed in the complex plane. The additional times in $\domt\setminus\alltimes$ are similarly chosen after applying the transformation $t\mapsto(1, t) / \sqrt{1+t^2}$ to map times to points on $\sphere^{1}$.
For the ADCG method we chose the parameters according to the reference implementation as follows:
\begin{table}[H]
    \centering
    \noindent\makebox[\textwidth]{%
        \begin{tabular}{lc}\toprule
    Description & Value \\ \midrule
    Maximum iterations in outer loop & 100 \\
    Maximum iterations in coordinate descent & 200 \\
    Grid size to find initial guess for new source point & $20\times20$ \\
    Minimum optimality gap before termination & $10^{-5}$ \\
    Minimum objective progress before termination & $10^{-4}$ \\
 \bottomrule
\end{tabular}
    }
    \caption{Parameters for ADCG}\label{tab:adcgparams}
\end{table}

\paragraph{Experiments without noise.}
In order to evaluate the exact reconstruction properties, we run two experiments.
In both, rather than actually imposing the constraint that the measurements $\data_t$ are exactly matched, we pick a very small regularization parameter $\alpha=0.005$. %
The size of the grid is chosen as $\nod=100$.

First, we study the influence of the number of measurement times by
comparing the performance of our ``dimred.\ mid'' method for each number of measurement times $\abs{\alltimes}=3,5,7$ to the static reconstruction. To this end, we run both methods on truncated Fourier measurements of particle configurations from dataset \dataset{3}. Since here we are only interested in the effect of $\abs{\alltimes}$ independent of the dynamic separation values, we use the dataset \dataset{3} for all choices of $\abs{\alltimes}$. Each reconstruction result is evaluated with the clustering method as either correctly or incorrectly reconstructed. For each number of times, the results are split into four groups depending on whether both, one or neither of the methods was successful. The size of the groups is visualized by the area of the green blobs in \cref{fig:exact}, top.

We see a large increase in the number of correctly reconstructed samples by our (dimension-reduced) dynamic method as the number of measurement times increases from $\abs{\alltimes}=3$ to $\abs{\alltimes}=5$, as expected. Meanwhile, adding two more measurement times ($\abs{\alltimes}=7$) seems to have less of an impact.
As expected, we also see that it is very rare for the dynamic reconstruction to fail on a configuration which is correctly reconstructed by the static method. Such cases cannot be fully excluded however, since the additional constraints in our dynamic formulation introduce some numerical dissipation, which sometimes results in spurious  or smeared out non-zero weights. These might then be incorrectly identified as clusters by the clustering evaluation.
Since we chose the frequency cutoff of the measurement operator quite low at $\maxFrequency=2$, the reconstruction task is challenging so that the number of configurations on which both methods fail stays large even for $\abs{\alltimes}=7$.

\begin{figure}%
\centering
\pgfplotsset{
all axes style/.style={
    width=\linewidth/3,
    scale only axis,
    every axis plot/.append style={
        thick
    }
},
blob axis style/.style={
    height=\linewidth/3,
    title style={yshift=2.5em},
    all axes style,
    axis lines=box,
    axis line style={draw=none},
    xmin=-2, xmax=2,
    ymin=-2, ymax=2,
    xtick=data, ytick=\empty,
    xticklabels={correct recons., failed recons.},
    xtick pos=top,
    ytick pos=left,
    tick align=inside,
    yticklabel style={
        rotate=90
    },
    xlabel={dynamic},
    label style={font=\bfseries},
    title style={font=\bfseries},
    y dir=reverse
},
blob plot style/.style={
    mark=*,
    mark options={
        fill=YlGn-I,
    },
    scatter,
    only marks,
    visualization depends on={value \thisrow{count} \as \blobcount},
    visualization depends on={value \thisrow{style} \as \blobstyle},
    visualization depends on={\thisrow{size} \as \marksize},
    scatter/@pre marker code/.style={
        /tikz/mark size=\marksize
    },
    scatter/@post marker code/.style={},
    nodes near coords*={\blobcount},
    every node near coord/.code={
        \ifnum \blobstyle=0
            \tikzset{color=white}
        \else
            \tikzset{color=black}
        \fi
        },
    nodes near coords align={
        \ifnum \blobstyle=0
            anchor=center,xshift=0pt
        \else
            anchor=west,xshift=3pt
        \fi
        },
},
blob cross style/.style={
    all axes style,
    height=\linewidth/3,
    axis lines=center,
    xmin=-1, xmax=1,
    ymin=-1, ymax=1,
    ticks=none,
    axis line style={-}
},
lines axis style/.style={
    all axes style,
    height=6 cm,
    xlabel={Dynamic separation},
    xticklabel style={
        /pgf/number format/fixed,
        /pgf/number format/precision=2
    },
    grid=major,
    xmin=0, xmax=0.1,
    ymin=-0.02, ymax=1.05,
    legend pos=north west,
    reverse legend,
},
lines plot style/.style={
    mark=none
},
}
\centerline{%
\begin{tabular}{ccc}
\includegraphics{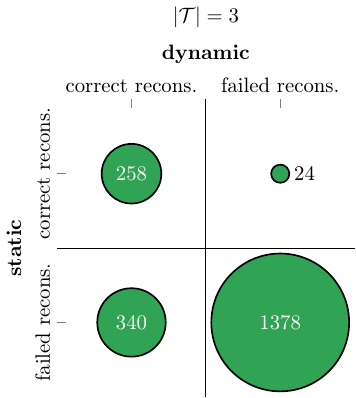}%
&
\includegraphics{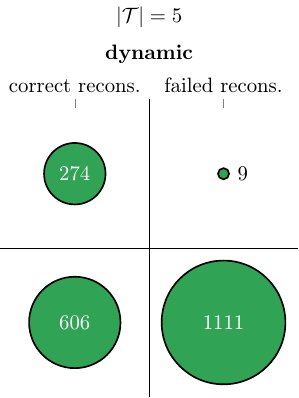}%
&
\includegraphics{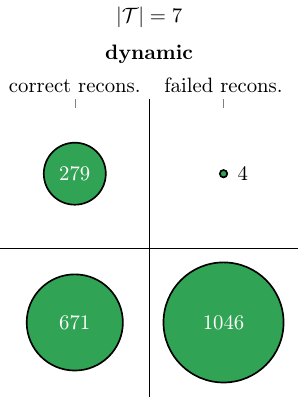}%

\\[0.5cm]
\includegraphics{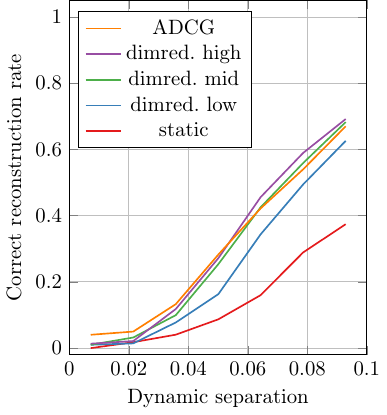}%
&
\includegraphics{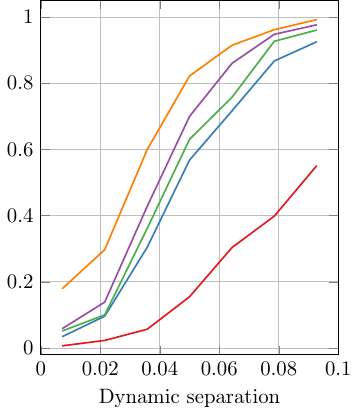}%
&
\includegraphics{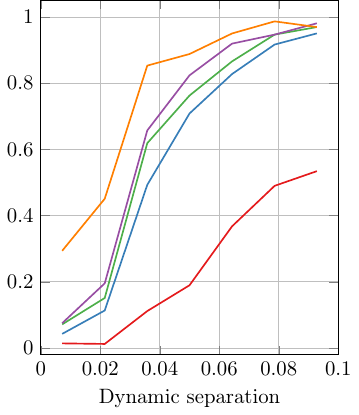}%

\end{tabular}%
}%

\caption{\textbf{Top row:} Comparing static reconstruction to our dynamic method with parameters as in dimred.\ mid, applied to the \dataset{3} dataset with 3, 5 and 7 measurement times (left to right). Results on each configuration evaluated by the clustering method as ``correct'' or ``failed'' and divided into four groups per number of measurement times. \textbf{Bottom row:} Comparing the correct reconstruction rate of our dimension-reduced methods against static reconstruction and ADCG applied to the full-dimensional model. Datasets used are \dataset{3}, \dataset{5}, \dataset{7} (left to right), particle configurations are binned by their dynamic separation, and the rate of correctly reconstructed configurations, as evaluated by the clustering method, is plotted for each bin.}\label{fig:exact}
\end{figure}

As a second experiment, we compare the performance of our dimension-reduced method for different parameter choices against the static as well as the full-dimensional reconstruction solved by ADCG. Here, we use the dataset \dataset{n}, $n=3,5,7$ for each respective number of measurement times in order to have a uniform sample of configurations across the whole range of dynamic separation values. The resulting reconstruction rates are displayed as a function of the dynamic separation in \cref{fig:exact}, bottom.

We notice that the dynamic reconstruction method ``dimred.\ low'', which uses only three projection directions to couple the information from multiple measurements, is able to improve upon the static reconstruction by a large margin. Unsurprisingly, all reconstruction methods improve with higher dynamic separation of the particle configurations, since our measurement operator filters out fine scale information.
For all values of $\abs{\alltimes}$, higher numbers of projection directions in our formulation (dimred.\ low to high) result in better reconstruction.
For $\abs{\alltimes}\in\set{5,7}$, the full-dimensional model solved by ADCG still beats our dimension-reduced method. Curiously, for only three measurement times, our dimension-reduced models mid and high slightly outperform ADCG for high separations, which might be explained by suboptimal parameter choices for ADCG (such as the grid size for the initial guess of a new source point).

Note that the performance of the static reconstruction method slightly improves with increasing $\abs{\alltimes}$ for each fixed dynamic separation $\dynsep$.
This is explained by the fact that, on average, for the same dynamic separation a larger $\abs{\alltimes}$ means a higher particle separation at the single time point $0$.

\paragraph{Experiments with noise.}
Finally, we intend to investigate the behaviour of our method with noise added to measurements. In this section, we fix the setting with five measurement times, $\alltimes=\set{-1,-0.5,0,0.5,1}$. For a given noise level $\delta$, Gaussian noise with standard deviation $\sqrt{2\delta / 250}$ is added to the vector resulting from truncated Fourier measurements \added{$\ftrunc$} with cutoff $\maxFrequency=2$. Since the size of this vector is 250, this is consistent in expectation with the definition of $\delta$ in \eqref{eqn:noiseStrength}.

We are especially interested in supporting our convergence analysis from \cref{sec:noisyReconstruction}. This analysis suggests the optimal regularization parameter $\alpha$ in \eqref{eqn:discrprob} to scale like the square root $\sqrt{\delta}$ of the noise level, thus $\alpha=\noiseconst\sqrt{\delta}$ for some constant $\noiseconst>0$. By a coarse parameter search, we have determined the value $\noiseconst=0.2$ to be roughly optimal, which we will use for the following experiments.

Two new datasets are generated for this setting: A dataset consisting only of a single particle in each configuration ($N=1$), as well as a second dataset consisting of two particles in each configuration which stay far apart ($N=2$ with $\dynsep>0.4$) over the five measurement times. Additionally, the dataset \dataset{3} from before (labelled ``$4\leq N\leq 20$'' in \cref{fig:noisy}) is filtered to contain only configurations which where correctly reconstructed in the noiseless experiments by the respective methods. All datasets are thinned to no less than 50 configurations per noise level.

We compare the dimension reduced method ``dimred.\ mid'' from before, but with grid size set to $200\times 200$, to the ADCG algorithm applied to the full-dimensional formulation. %
The error measure used is the mean of the unbalanced Wasserstein divergence per noise level.

\begin{figure}[ht]
\centering
\input{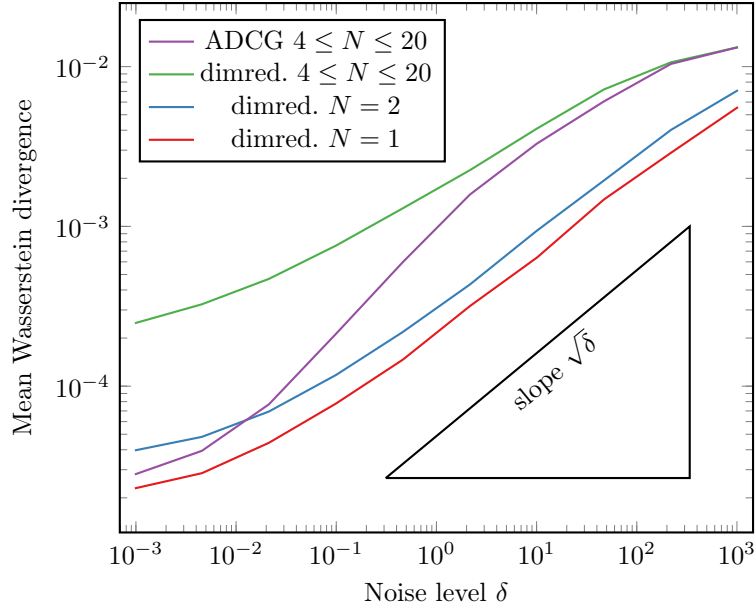}
\caption{Convergence properties for vanishing noise level for our method applied to datasets with one, two and $4-20$ particles in each configuration, as well as ADCG on the latter. Slope $\sqrt{\delta}$ provided for comparison.}\label{fig:noisy}
\end{figure}

From our analysis in \cref{thm:errorEstimatesSummary}, we expect to see a convergence rate of the Wasserstein divergence of $\sqrt{\delta}$ for our dimension-reduced method. \Cref{fig:noisy} shows this to be approximately true for the datasets with one and two particles per configuration in the regime $\delta\in[1,100]$. It is not surprising that the curves flatten in the regime of very low noise levels, since our grid-based implementation cannot surpass the limit given by the grid resolution.

We notice that the curve ``dimred.\ $4\leq N\leq 20$'', which corresponds to our method applied to the diverse \dataset{3} dataset, never quite reaches the expected slope of $\sqrt{\delta}$. Our hypothesis is that in this case, errors from the discretization start to degrade the reconstruction result already before the noise level gets into the regime in which the expected rate would be observed. Finally, the full-dimensional ADCG method seems to even surpass the rate $\sqrt{\delta}$ before flattening for very low noise levels.

\appendix
\section{}

\subsection{Alternative Formulations}\label{sec:alternativeFormulations}
\added{\lookUp{\ref{item3}}%
In the dimension-reduced problem \ref{eq:dim_reduced_general}, we employed the product space $\Mp(\domstat)^\domt$ for the snapshots and the measure space $\Mp(\alldirs \times \projdomdyn) $ for the position-velocity projections. As discussed in \cref{rem:equivalent_formulations} there are, however,  formulations alternative to \eqref{eq:lifted_exact} that work with either only product spaces or only measure spaces. The purpose of this section is to introduce and analyse these variants, in particular with respect to their equivalence to \ref{eq:dim_reduced_general}.
}

\removed{This will be the main setting in our work, because using $\Mp(\domstat)^\domt$ (as opposed to $\Mp(\domstat \times \domt)$) allows to apply the observation operator $\fopstat_t$ at every $t\in\domt$ and because using $\Mp(\alldirs \times \projdomdyn)$ (as opposed to $\Mp( \projdomdyn)^\alldirs $) already provides some regularity of the measures $\liftVar$ with respect to $\theta \in \alldirs$.
There are, however, formulations alternative to \eqref{eq:lifted_exact}.
}

One alternative is to use the product topology both for the temporal dimension of the snapshots and the angular dimension of the position-velocity projections as follows,
\leqnomode
\begin{equation*} \label{eq:dim_reduced_product_only} \tag*{{$\Pi$-\problemTagNoLink[\alpha]{\data}}}
\min_{ \substack{ \liftVar \in \Mp( \projdomdyn)^{\alldirs}  \\ u\in \Mp(\domstat)^{\domt} }}\!
\| \liftVar_{\hat\theta} \|_\M
\!+\!\frac1{2\alpha}\!\sum_{t\in\alltimes}\|\fopstat_tu_t\!-\!\data_t\|_H^2
\quad \text{ such that }
\mv_t \liftVar_\theta = \Rs_\theta u_t \text{ for all }t\!\in\!\domt, \, \theta\!\in\!\alldirs,
\end{equation*}
\reqnomode
where $\hat\theta\in\alldirs$ is arbitrary, but fixed.

Another alternative would be to consider both the snapshots as well as the position-velocity projections as Radon measures on a product space. To achieve this, we need to adapt the measurement operators $(\fopstat_t)_{t \in \alltimes}$ such that an evaluation on $u \in \Mp(\domstat \times \domt)$ is well-defined. For the moment, let $\fopdyn:\Mp(\domstat \times \domt) \rightarrow H^{|\alltimes|}$ denote such a generalization (which will be specified in \cref{defn:radon_space_measurement_operator} below).
Then the second variant of the dimension-reduced problem reads
\leqnomode
\begin{equation*} \label{eq:dim_reduced_measure_only}\tag*{{$\M$-\problemTagNoLink[\alpha]{\data}}}
\min_{ \substack{ \liftVar \in \Mp(\alldirs \times \projdomdyn)  \\ u \in \Mp(\domstat \times \domt) }}\!
\| \liftVar \|_\M
\!+\!\frac1{2\alpha}\!\sum_{t\in\alltimes}\|(\fopdyn u)_t\!-\!\data_t\|_H^2
\quad \text{ such that }
\amv \liftVar = \Rf u.
\end{equation*}
\reqnomode

\removed{The goal of this section is to investigate when the three dimension-reduced settings \ref{eq:dim_reduced_general}, \ref{eq:dim_reduced_product_only} and \ref{eq:dim_reduced_measure_only} are actually equivalent. %
}
\added{

The main results of this section are the existence of solutions to \ref{eq:dim_reduced_measure_only} in \cref{prop:existence_reduced_measure_only} and to \ref{eq:dim_reduced_product_only} in \cref{prop:existence_dim_reduced_product_only}, where for the former we require that the Radon operator corresponding to $(\alldirs, \mdirs)$ is injective. The equivalence of \ref{eq:dim_reduced_measure_only} and \ref{eq:dim_reduced_general} will be shown under the same condition in \cref{prop:equivalence_measure}. Finally, the equivalence of \ref{eq:dim_reduced_product_only} and \ref{eq:dim_reduced_general} will be shown in \cref{prop:equivalence_product_only}, where we require that the move operator corresponding to $(\domt, \mtime)$ is injective.
}

At first, we elaborate on the relation between \ref{eq:dim_reduced_general} and \ref{eq:dim_reduced_measure_only}. To this end, we will need to decompose $u \in \Mp(\domstat \times \domt) $ as  $u_t \prodm t  \wrt \mtime $. As the following lemma shows, this is always possible when $u$ corresponds to the projection of a full-dimensional measure $\lambda  \in \Mp(\Omega)$. 
\begin{lem}[Measure decomposition with full-dimensional lifting] Let $u \in \Mp(\R^d \times \domt)$ and $\lambda  \in \Mp(\R^d \times \R^d)$ be such that 
\[ u = \mv^d \lambda .\]
Then there is a unique family of measures $(u_t)_{t\in \domt}$ in $\Mp(\R^d)$ such that
\[ u = u_t \prodm t \mtime \quad \text{ and } \quad t \mapsto \int_{\R^d} \varphi \wrt u_t \text{ is continuous for every } \varphi \in C_c(\R^d).
\]
\begin{proof}
Regarding existence, we can select $u_t = \mv_t^d \lambda$ for each $t \in \domt$ such that the decomposition follows since by definition $\mv^d \lambda = \mv_t^d \lambda \prodm t \mtime$.
With this choice it then follows from uniform continuity of any  $\varphi \in C_c(\R^d)$ that the mapping
\[ t \mapsto \int_{\R^d} \varphi \wrt u_t = \int_{\R^d \times \R^d} \varphi (x + tv) \wrt \lambda (x,v)
\] is (even uniformly) continuous. Given any other $\tilde{u}_t$ with $u = \tilde{u}_t \prodm t \mtime$ we must have $\tilde{u}_t = u_t$ $\mtime$-almost everywhere. If $\tilde u$ in addition satisfies the continuity requirement it follows that $\tilde{u}_t = u_t$ for all $t \in \domt$. 
\end{proof}
\end{lem}

Now we aim to obtain a similar regularity result for a measure $u \in \Mp(\domstat \times \domt)$ with $ \Rf u = \amv \liftVar $ for some $\liftVar \in \Mp(\alldirs \times \projdomdyn)$.
To this end, we first need a lemma about disintegration \added{\lookUp{\ref{item3}}that is a direct extension of \cref{lem:projected_constraint_decomposition_mixed_model} to the setting where  $u \in \Mp(\domstat \times \domt)$.}
\begin{lem}[Measure decomposition with projected lifting on measure spaces] \label{lem:projected_constraint_decomposition}
Let $u \in \Mp(\R^d \times \domt )$ and $\liftVar  \in \Mp(\alldirs \times  \R^2)$ be such that
\[ \Rf u = \amv \liftVar .\]
Then $u$ and $\liftVar$ can be decomposed as
\[ u = u_t \prodm t \mtime \quad\text{ and }\quad \liftVar = \mdirs \prodm \theta \liftVar_\theta \]
with $(u_t)_{t\in \domt}$ a family of measures in $\Mp(\R^d)$ such that $\|u_t\| _\M = \|u\|_\M$ and 
$(\liftVar_\theta)_{\theta\in \alldirs}$ a family of measures in $\Mp(\R^2)$ such that $\|\liftVar_\theta\| _\M = \|\liftVar \|_\M$. Further, $\|u\|_\M = \|\liftVar\|_\M$ and 
\[ \Rs _\theta u_t = \mv_t \liftVar_\theta, \quad
\Rs u_t = \amv_t \liftVar,\quad
\Rf _\theta u = \mv \liftVar_\theta
\]
for $\mtime$-almost every $t$ and $\mdirs$-almost every $\theta$.
\begin{proof}
At first, note that $\Rf u = \amv \liftVar$ implies that
\[ \mdirs \prodm \theta \Rf_\theta u = \amv_t \liftVar \prodm t \mtime.
\] 
Now using the disintegration theorem \cite[Thm.\,5.3.1]{ambrosio2008gradient}, we decompose $u =  u_t \prodm t \nu_u $ and $\liftVar = \nu_\liftVar  \prodm \theta \liftVar_\theta$  with $\nu_u=\frac1{\|u\|_\M}\pf{(\pi^\domt)}u$, $\nu_{\liftVar}=\frac1{\|\liftVar\|_\M}\pf{(\pi^\alldirs)}\liftVar$ being probability measures and $\|u_t\|_\M =\| u\|_\M $ and $\|\liftVar_\theta\|_\M = \|\liftVar\|_{\M}$ for almost every $t$ and $\theta$. Here, $\pi^\domt : \R^d \times \domt \rightarrow \domt$ and $\pi^\alldirs :  \alldirs  \times \R^2  \rightarrow \alldirs$ denote the canonical projections.
From this decomposition it follows by a direct computation that
\[
\Rf_\theta u = \Rs _\theta u_t \prodm t \nu_u 
\quad \text{and} \quad 
\amv_t \liftVar = \nu_\liftVar  \prodm  \theta \mv_t \liftVar_\theta .
\]
Hence we obtain
\[\mdirs \prodm \theta  (\Rs _\theta u_t \prodm t  \nu_u )= 
(\nu_\liftVar  \prodm \theta \mv_t \liftVar_\theta) \prodm t \mtime .\]
Evaluating the measures on both sides at $\Theta \times \R \times A $ with $A \subset \domt$ an arbitrary Borel set, we obtain $\|u\|_\M\nu_u(A) = \|\liftVar\|_\M\mtime(A)$. The choice $A=\domt$ yields $\|u\|_\M=\|\liftVar\|_\M$ and thus $\nu_u=\mtime$. Likewise, evaluation at $\omega \times \R \times \domt$ for $\omega \subset \Theta$ an arbitrary Borel set leads to $ \|\liftVar\|_\M\nu_\liftVar (\omega) = \|u\|_\M\mdirs (\omega) $ and thus $\nu_\liftVar=\mdirs$. This proves the decomposition of $u$ and $\liftVar$, and the claimed equalities follow 
\added{\lookUp{\ref{item3}} similarly as in the proof of \cref{lem:projected_constraint_decomposition_mixed_model}}.
\end{proof}
\end{lem}

Now if $(\alldirs,\mdirs)$ is such that the Radon transform is injective (see \cref{prop:radon_injective}), we can obtain a certain uniqueness of the decomposition of $u$ in the above \namecref{lem:projected_constraint_decomposition}, even for \emph{every} (rather than $\mtime$-almost every) $t$. As our argument requires weak-$\ast$ continuity of the Radon transform, we formulate it for the bounded support setting so that \cref{prop:supports} can be applied.
\begin{prop}[Unique decomposition with Radon operators] \label{prop:decomposition_measure_u_unique}
Assume $(\alldirs,\mdirs)$ to be such that $\Rs:\M(\Omega)\to\M(\alldirs\times\projdomstat)$ is injective,
and let $u \in \Mp(\domstat \times \domt)$ and $\liftVar  \in \Mp(\alldirs \times  \projdomdyn)$ satisfy
\[ \Rf u = \amv \liftVar. \]
Then there is a unique family of measures $(u_t)_{t\in\domt}$ such that
\[ u = u_t \prodm t  \mtime \quad \text{ and } \quad t \mapsto \int_{\alldirs \times \projdomstat} \varphi \wrt \Rs  u_t \text{ is continuous for every } \varphi \in C(\alldirs \times \projdomstat).
\]
\begin{proof}
By \cref{lem:projected_constraint_decomposition}, there is  a family of measures $u$ such that
\[ u = u_t \prodm t  \mtime  \quad \text{and} \quad \Rs u_t = \amv_t \liftVar \quad \text{for } \mtime\text{-almost every }t . \]
Since $\Rs$ is injective, this implies $u_t=\Rs^{-1}\amv_t\liftVar$ for $\mtime$-almost every $t$, where $\Rs^{-1}$ denotes the left-inverse of $\Rs$.
To show that this formula can be used to specify $u_t$ for \emph{every} $t\in\domt$, we show that $\amv_t \liftVar$ is in the range of $\Rs$ for every $t \in \domt$. Indeed, for an arbitrary $t \in \domt$ we can find a sequence $t_1,t_2,\ldots$ converging to $t$ such that $\Rs u_{t_n} = \amv_{t_n}\liftVar$. Since $\|u_{t_n}\|_\M = \|u\|_\M$ for all $n$, the sequence $u_{t_n}$ admits a (non-relabelled) subsequence converging to some $v \in \Mp(\domstat)$. Weak-$\ast$ continuity of $\Rs$ implies convergence of $\Rs u_{t_n}$ to $\Rs v$. On the other hand, it is easy to see that also $\amv_{t_n} \liftVar$ weakly-$\ast$ converges to $\amv_t \liftVar$ as $n \rightarrow \infty$ such that $\amv_t \liftVar$ is indeed in the range of $\Rs$ as claimed. 
Now it follows from uniform continuity of $\varphi \in C(\alldirs \times \projdomstat)$ that the mapping
\begin{align*}
 t \mapsto  \int_{\alldirs \times \projdomstat} \varphi \wrt \Rs u_t 
& =\int_{\alldirs \times \projdomstat} \varphi \wrt \amv_t \liftVar  
 =\int_{\alldirs \times \projdomdyn} \varphi(\theta,s+tr) \liftVar(\theta,s,r)  
\end{align*}
is continuous (again even uniformly) with respect to $t$.

As for uniqueness, given any other family $(\tilde{u}_t)_{t\in\domt}$ satisfying $u =\tilde u_t \prodm t  \mtime$ and such that $t\mapsto\int_{\alldirs \times \projdomstat} \varphi \wrt \Rs\tilde u_t$ is continuous, it follows that $\Rs u_t = \Rs \tilde{u}_t$ for all $t$ and, from injectivity of $\Rs$, that $ u_t = \tilde{u}_t$ for all $t$ as claimed.
\end{proof}
\end{prop}

For any $u \in \Mp(\domstat\times\domt)$ admitting a decomposition $u = u_t \prodm t  \mtime$ with $u_t$ uniquely defined in the above sense for every $t \in \domt$, we can now define an observation operator at all observation times.
\begin{defn}[Dynamic observation operator] \label{defn:radon_space_measurement_operator}
If $(\alldirs,\mdirs)$ is such that $\Rs$ is injective, abbreviate
\begin{equation*}
\Cc:=\{u \in \Mp(\domstat \times \domt)\,|\,\Rf u = \amv \liftVar\text{ for some }\liftVar\in\M(\alldirs\times\projdomdyn)\}
\end{equation*}
and denote by $u = u_t \prodm t  \mtime$ the unique decomposition of a $u\in\Cc$ from \cref{prop:decomposition_measure_u_unique}.
With $(\fopstat_t)_{t \in \alltimes}$ the family of observation operators from \cref{defn:pointwise_measurement_operator}  we define the \emph{dynamic observation operator} $\fopdyn:\Cc \subset  \Mp (\domstat \times \domt) \rightarrow H^{|\alltimes|}$ as
\[
\fopdyn u:=  (\fopstat _tu_t)_{t \in \alltimes}.
\]
\end{defn}

With this definition of $\fopdyn$, the minimization problem \ref{eq:dim_reduced_measure_only} is well-defined and also admits a solution.
\begin{prop}[Existence for \ref{eq:dim_reduced_measure_only}] \label{prop:existence_reduced_measure_only} Let $(\alldirs,\mdirs)$ be such that $\Rs:\M(\domstat) \rightarrow \M(\alldirs \times \projdomstat) $ is injective and, if $\alpha=0$, assume that there exists $\lambda \in \Mp(\domdyn)$ with $\fopstat_t\mv^d_t\lambda = \data^\dagger _t$ for all $t \in \alltimes$. Then the minimization problem \ref{eq:dim_reduced_measure_only} is well-defined and admits a solution.
\end{prop}
This result on existence is a direct consequence of the existence result \cref{prop:well_posedness_dim_reduced_noisy} for \ref{eq:dim_reduced_general} and the following equivalence of the two problems.

\begin{prop}[Equivalence of \ref{eq:dim_reduced_general} and \ref{eq:dim_reduced_measure_only}] \label{prop:equivalence_measure} If $(\alldirs,\mdirs)$ is such that $\Rs:\M(\domstat) \rightarrow \M(\alldirs \times \projdomstat) $ is injective, the minimization problems \ref{eq:dim_reduced_general} and \ref{eq:dim_reduced_measure_only} are equivalent in the following sense.
\begin{enumerate}
\item If $((u_t)_{t\in\domt},\liftVar)$ solves \ref{eq:dim_reduced_general}, then $(u,\liftVar)$ solves \ref{eq:dim_reduced_measure_only} with $u:= u_t \prodm t  \mtime \in \Mp(\domstat \times \domt)$.
\item If $(u,\liftVar)$ solves \ref{eq:dim_reduced_measure_only}, then $((u_t)_{t\in\domt},\liftVar)$ solves \ref{eq:dim_reduced_general}, where $u=u_t\prodm t\mtime$ is the unique decomposition of $u$ from \cref{prop:decomposition_measure_u_unique}.
\end{enumerate}
\begin{proof}
First we show that if $((u_t)_{t\in\domt},\liftVar)$ is admissible for \ref{eq:dim_reduced_general}, then $(u,\liftVar)$ with $u:= u_t \prodm t\mtime$ is admissible for \ref{eq:dim_reduced_measure_only} and has same cost:
Indeed, as long as $u=u_t\prodm t\mtime$ is actually well-defined and lies in $\Mp(\domstat \times \domt)$,
then the condition $\Rs u_t = \amv_t \liftVar$ for all $t \in \domt$ automatically implies $\Rf u=\amv\liftVar$.
Furthermore, since $t\mapsto\int_{\alldirs\times\projdomstat}\varphi\wrt\Rs u_t
=\int_{\alldirs \times \projdomstat} \varphi \wrt \amv_t \liftVar  
=\int_{\alldirs \times \projdomdyn} \varphi(\theta,s+tr) \liftVar(\theta,s,r)$
is continuous in $t$, the unique decomposition of $u$ from \cref{prop:decomposition_measure_u_unique} is in fact given by $u=u_t\prodm t\mtime$
so that $\fopdyn u = (\fopstat_t u_t)_{t \in \alltimes}$.
It still remains to show that $u_t\prodm t\mtime\in\Mp(\domstat \times \domt)$ is well-defined.
To this end, first note that by \cref{prop:properties_product_measure} the product measure $\amv_t \liftVar \prodm t  \mtime =\Rs u_t \prodm t  \mtime \in \Mp(\alldirs \times \projdomstat \times \domt)$ is well-defined.
Next we show existence of $\tilde u \in \Mp(\domstat \times \domt)$ such that $\Rf\tilde u = \Rs u_t \prodm t  \mtime$. For this purpose we note that
\[ \Rf^*:C_0(\alldirs \times \projdomstat \times \domt) \rightarrow C_0(\domstat \times \domt),\quad
\Rf^* \varphi (x,t):=\int_\alldirs \varphi( \theta,\theta \cdot x,t) \wrt \mdirs(\theta) \]
is well-defined and, as a direct computation shows, has $\Rf$ as its dual. Then, existence of $\tilde u \in \Mp(\domstat \times \domt)$ such that $\Rf\tilde u = \Rs u_t \prodm t  \mtime$ follows from \cref{prop:abstractRangeCharacterization} in the appendix since
\begin{multline*}
 \sup_{\substack{
 \varphi \in C_0(\alldirs \times \projdomstat \times \domt) \\\sup_{x \in \domstat,t \in \domt}(\Rf^* \varphi)(x,t) \leq 1
 }} \int _ \domt \int _{\alldirs \times \projdomstat}  \varphi(\theta,s,t) \wrt \Rs u_t(\theta,s) \wrt \mtime(t) \\  
  \leq \sup_{\substack{
  \varphi \in C_0(\alldirs \times \projdomstat \times \domt)  \\
  \sup_{x \in \domstat,t \in \domt} \int_\alldirs \varphi( \theta,\theta \cdot x,t) \wrt \mdirs(\theta) \leq 1 }} \sup_{t \in \domt} \int _{\alldirs \times \projdomstat } \varphi(\theta,s,t) \wrt \Rs u_t(\theta,s) \\
  \leq \sup_{t \in \domt} \sup_{\substack{
   \varphi \in C_0(\alldirs \times \projdomstat) \\
    \sup_{x \in \domstat } \int_\alldirs \varphi( \theta,\theta \cdot x) \wrt \mdirs(\theta) \leq 1 }}
     \int _{\alldirs \times \projdomstat}\varphi (\theta,s) \wrt  \Rs u_t(\theta,s) \leq  \sup_{t \in \domt} \|u_t\|_\M<  \infty.
\end{multline*}
Now by \cref{prop:decomposition_measure_u_unique} we can uniquely decompose $\tilde u=\tilde u_t \prodm t\mtime$ so that
\[ \Rs u_t \prodm t \mtime = \Rf\tilde u = \Rs \tilde u_t \prodm t \mtime \]
and thus $\Rs u_t = \Rs \tilde u_t$ for $\mtime$-almost every $t \in \domt$. Injectivity of the Radon transform then implies $\tilde u = u_t \prodm t \mtime \in \Mp(\domstat \times \domt)$.

Now we consider the reverse situation and show that if $(u,\liftVar)$ is admissible for \ref{eq:dim_reduced_measure_only} then $((u_t)_{t\in\domt},\liftVar)$ is so for \ref{eq:dim_reduced_general} with same cost (where $u=u_t\prodm t\mtime$ is the unique decomposition of $u$ from \cref{prop:decomposition_measure_u_unique}).
Indeed, by definition of the dynamic observation operator we have $\fopdyn u = (\fopstat_t u_t)_{t \in \alltimes}$ for all $t\in\alltimes$.
Furthermore, as in the proof of \cref{prop:decomposition_measure_u_unique} it follows that $\Rs u_t = \amv_t \liftVar $ for every $t \in \domt$.

Since an admissible point for one problem induces an admissible point for the other problem with exactly same cost,
minimizers of one problem must induce minimizers of the other.
\end{proof}
\end{prop}

Before considering in a second step the equivalence of \ref{eq:dim_reduced_general} and \ref{eq:dim_reduced_product_only}, let us briefly show well-posedness of \ref{eq:dim_reduced_product_only}.

\begin{prop}[Existence for \ref{eq:dim_reduced_product_only}]\label{prop:existence_dim_reduced_product_only}
Let $\data = (\data_t)_{t\in\alltimes} \in H^{|\alltimes|}$ and $\alpha \in [0,\infty)$.
If $\alpha=0$, assume there exists $\lambda \in \Mp(\domdyn)$ with $\fopstat_t\mv^d_t\lambda = \data _t$ for all $t \in \alltimes$.
Then \ref{eq:dim_reduced_product_only} has a minimizer.
\end{prop}
\begin{proof}
We rewrite \ref{eq:dim_reduced_product_only} as minimization of the energy
\[
E(u,\liftVar)=\|\liftVar_{\hat\theta}\|_\M+\frac1{2\alpha}\sum_{t\in\alltimes}\|\fopstat_tu_t-\data_t\|_H^2+\iota_{\{\Rs_\theta u_t=\mv_t\liftVar_\theta\,\forall t,\theta\}}(u,\liftVar)
\]
over $(u,\liftVar)\in\Mp(\Omega)^\domt\times\Mp(\projdomdyn)^\alldirs$, where $\iota_A$ is the indicator function of a constraint set $A$.
There exists $C<\infty$ with $\inf E<C$ since $E$ is finite at $(u,\liftVar)=(0,0)$ when $\alpha>0$
and at $(u,\liftVar)=((\mv^d_t\lambda)_{t\in\domt},(\Rj_\theta\lambda)_{\theta\in\alldirs})$ when $\alpha=0$.
We now endow $\Mp(\Omega)^\domt\times\Mp(\projdomdyn)^\alldirs$ with the product topology of the weak-$\ast$ topologies on $\Mp(\Omega)$ and $\Mp(\projdomdyn)$, respectively.
Then the set
$B=\{(u,\liftVar)\in\Mp(\Omega)^\domt\times\Mp(\projdomdyn)^\alldirs\,|\,\|u_t\|_\M,\|\liftVar_\theta\|_\M\leq C\text{ for all }t\in\domt,\theta\in\alldirs\}$
is compact by Tychonoff's theorem, since $\|\cdot\|_\M$-balls are compact on $\Mp(\Omega)$ and $\Mp(\projdomdyn)$, respectively, by the Banach--Alaoglu theorem.
By \cref{prop:radon_properties,prop:move_properties}, any $(u,\liftVar)$ with finite energy satisfies
$\|u_t\|_\M=\|\liftVar_\theta\|_\M=\|\liftVar_{\hat\theta}\|_\M$ for all $t\in\domt$ and $\theta\in\alldirs$
so that $\inf E=\inf_BE$.
Existence of a minimizer on the compact set $B$ now follows from the lower semi-continuity of $E$ with respect to the chosen topology,
which in turn follows from the weak-$\ast$-to-weak-$\ast$ continuity of $\fopstat_t$, $\mv_t$, and $\Rs_\theta$
as well as the weak-$\ast$ lower semi-continuity of norms.
\end{proof}

Now we aim to show also equivalence of \ref{eq:dim_reduced_general} and \ref{eq:dim_reduced_product_only}, which, similarly to the previous case, requires injectivity of the move operator $\mv$.
First we obtain for $ \liftVar \in \Mp(\alldirs \times \projdomstat)$ with $\amv _t \liftVar = \Rs u_t$ a decomposition $\liftVar = \mdirs  \prodm \theta \liftVar_\theta$ with $\liftVar_\theta$ being uniquely defined for every $\theta \in \alldirs$ in a certain sense.
\begin{prop}[Unique decomposition with move operators]\label{prop:unique_decomp_move} Assume $(\domt,\mtime)$ to be such that $\mv:\M(\projdomdyn) \rightarrow \M(\projdomstat \times \domt) $ is injective, and let  $ \liftVar \in \Mp(\alldirs \times \projdomdyn)$ and $u \in \Mp(\domstat)^\domt$ satisfy
\[ \amv_t \liftVar = \Rs u_t \quad \text{for every }t \in \domt.\] 
Then there exists a unique family of measures $(\liftVar_\theta)_{\theta \in \alldirs}$ such that
\[ \liftVar = \mdirs \prodm \theta \liftVar_\theta \quad \text{ and } \quad \theta \mapsto \int_{\projdomstat \times \domt} \varphi \wrt \mv \liftVar_\theta \text{ is continuous for each } \varphi \in C(\projdomstat \times \domt).
\]
\begin{proof}
First we can argue as in the proof of \cref{prop:equivalence_measure} to obtain, from $\amv_t \liftVar = \Rs u_t$, existence of $\tilde u \in \M(\domstat \times \domt)$ such that
\[ \Rf\tilde u = \amv \liftVar ,\]
noting that injectivity of the Radon transform in \cref{prop:equivalence_measure} was only required to obtain uniqueness of $\tilde u$, which is not necessary here.
Now \cref{lem:projected_constraint_decomposition} implies that there is  a family of measures $(\liftVar_\theta)_{\theta \in \alldirs}$ such that
\[ \liftVar =  \mdirs \prodm \theta \liftVar_\theta \quad \text{and} \quad \mv \liftVar_\theta = \Rf_\theta\tilde u  \quad \text{ for $\mdirs$-almost every }\theta \in \alldirs . \]
Similarly as in the proof of \cref{prop:decomposition_measure_u_unique}, this implies $\liftVar_\theta=\mv^{-1}\Rf_\theta\tilde u$ for $\mdirs$-almost every $\theta \in \alldirs$, and our goal is to define $\liftVar_\theta$ this way for \emph{every} $\theta \in \alldirs$.
 To this end, we show that $\Rf_\theta\tilde u$ is in the range of $\mv$ for every $\theta \in \alldirs$: For $\theta \in \alldirs$ arbitrary, we can find a sequence $\theta_1,\theta_2,\ldots$ converging to $\theta$ such that $\Rf_{\theta_n}\tilde u = \mv \liftVar_{\theta_n}$.
  Since $\|\liftVar_{\theta_n}\|_\M = \|\liftVar\|_\M$ for all $n$, the sequence $\liftVar_{\theta_n}$ admits a (non-relabelled) subsequence converging weakly-$\ast$ to some $\tilde{\liftVar} \in \Mp(\projdomdyn)$. Weak-$\ast$ continuity of $\mv$ implies convergence of $\mv \liftVar_{\theta_n}$ to $\mv \tilde{\liftVar}$. On the other hand, it is easy to see that also $\Rf_{\theta_n}\tilde u$ weakly-$\ast$ converges to $\Rf_\theta\tilde u$ as $n \rightarrow \infty$ such that $\Rf_\theta\tilde u$ is in the range of $\mv$ as claimed. 
Hence we have $\liftVar = \mdirs \prodm \theta \liftVar_\theta$ with $\liftVar_\theta = \mv^{-1}\Rf_\theta\tilde u$ for every $\theta  \in \alldirs$.
Now it follows from uniform continuity of $\varphi \in C(\projdomstat \times \domt)$ that the mapping
\begin{align*}
 \theta \mapsto  \int_{\projdomstat \times \domt} \varphi \wrt \mv \liftVar_\theta 
& =\int_{\projdomstat \times \domt} \varphi \wrt \Rf_\theta\tilde u  
 =\int_{\domstat \times \domt} \varphi(\theta \cdot x,t) \wrt\tilde u(x,t)  
\end{align*}
is continuous (again even uniformly) with respect to $\theta$ as claimed. Given any other family $(\tilde{\liftVar}_\theta)_{\theta\in\alldirs}$ satisfying $\liftVar  = \mdirs (\theta) \times \tilde{\liftVar}_\theta$ and continuity of $\theta\mapsto\int_{\projdomstat \times \domt} \varphi \wrt \mv\tilde\liftVar_\theta$, it follows that $\mv \liftVar_\theta = \mv \tilde{\liftVar}_\theta$ for all $\theta$ and, from injectivity of $\mv$, that $ \liftVar_\theta = \tilde{\liftVar}_\theta$ for all $\theta$ as claimed.
\end{proof}
\end{prop}

As before, this unique decomposition allows to show equivalence between \ref{eq:dim_reduced_general} and \ref{eq:dim_reduced_product_only}.

\begin{prop}[Equivalence of \ref{eq:dim_reduced_general} and \ref{eq:dim_reduced_product_only}] \label{prop:equivalence_product_only}
If $(\domt,\mtime)$ is such that $\mv:\M(\projdomdyn) \rightarrow \M(\projdomstat \times \domt) $ is injective, the minimization problems \ref{eq:dim_reduced_general} and \ref{eq:dim_reduced_product_only} are equivalent in the following sense.
\begin{enumerate}
\item If $(u,\liftVar)$ solves \ref{eq:dim_reduced_general}, then $(u,(\liftVar_\theta)_{\theta\in\alldirs})$ solves \ref{eq:dim_reduced_product_only},
where $\liftVar = \mdirs \prodm \theta \liftVar_\theta$ is the unique decomposition of $\liftVar$ from \cref{prop:unique_decomp_move}.
\item If $(u,(\liftVar_\theta)_{\theta \in \alldirs})$ solves \ref{eq:dim_reduced_product_only}, then $(u,\liftVar)$ solves \ref{eq:dim_reduced_general} with $\liftVar:= \mdirs \prodm \theta \liftVar_\theta \in \Mp(\alldirs \times \projdomdyn)$.
\end{enumerate}
\begin{proof}
First we show that if $(u,(\liftVar_\theta)_{\theta \in \alldirs})$ is admissible for \ref{eq:dim_reduced_product_only}, then $(u,\liftVar)$ with $\liftVar:= \mdirs \prodm \theta \liftVar_\theta$ is admissible for \ref{eq:dim_reduced_general} and has same cost:
Indeed, as long as $\liftVar= \mdirs \prodm \theta \liftVar_\theta$ is actually well-defined and lies in $\Mp(\alldirs \times \projdomdyn)$,
the condition $\Rs_\theta u_t=\mv_t\liftVar_\theta$ for all $t\in\domt$ and $\theta\in\alldirs$ automatically implies $\Rs u_t=\amv_t\liftVar$ for all $t\in\domt$.
Furthermore, by \cref{prop:radon_properties,prop:move_properties} we have $\|\liftVar_\theta\|_\M=\|u_t\|_\M$ for all $t\in\domt$ and $\theta\in\alldirs$
and thus also $\|\liftVar\|_\M=\|\mdirs \prodm \theta \liftVar_{\theta}\|_\M=\|\liftVar_{\hat\theta}\|_\M$.
It still remains to show that $\mdirs \prodm \theta \liftVar_\theta\in\Mp(\alldirs \times \projdomdyn)$ is well-defined.
To this end, first note that by \cref{prop:properties_product_measure} the product measure
$\mdirs \prodm \theta \Rs_\theta u_t  = \mdirs \prodm \theta \mv_t \liftVar_\theta \in \Mp(\alldirs \times \projdomstat)$ is well-defined for each $t \in \domt$. Further observe that for any  $ \varphi \in C_0(\alldirs \times \projdomstat \times \domt)$ the mapping
\[ t \mapsto  
\int_\alldirs \int_\projdomstat \varphi(\theta,s,t) \wrt \mv_t \liftVar_\theta(s) \wrt  \mdirs (\theta)
= \int_\alldirs \int_\projdomdyn \varphi(\theta,z_1+tz_2,t) \wrt\liftVar_\theta(z_1,z_2) \wrt  \mdirs (\theta)
\]
is continuous in $t$. Hence we can define the mapping 
\[ C_0(\alldirs \times \projdomstat \times \domt) \ni \varphi   \mapsto \left( (\mdirs \prodm \theta \mv_t \liftVar_\theta) \prodm t \mtime  \right) (\varphi): =
\int_\domt \int_\alldirs \int_\projdomstat \varphi(\theta,s,t) \wrt \mv_t \liftVar_\theta(s) \wrt  \mdirs (\theta) \wrt \mtime(t)
 \]
 and, since $ |\left( (\mdirs \prodm \theta \mv_t \liftVar_\theta) \prodm t \mtime  \right) (\varphi)| \leq C \|\varphi\|_\infty $  with $C>0$ by a direct estimate, obtain that 
$  (\mdirs \prodm \theta \mv_t \liftVar_\theta) \prodm t \mtime \in \Mp(\alldirs \times \projdomstat \times \domt) $.
We next show existence of some $\tilde\liftVar \in \Mp(\alldirs \times \projdomdyn)$ such that $\amv \tilde\liftVar  =(\mdirs \prodm \theta \mv_t \liftVar_\theta) \prodm t \mtime $. To this end, first note that
\[ \amv^*:C_0(\alldirs \times \projdomstat \times \domt) \rightarrow C_0(\alldirs \times \projdomdyn),\quad
\amv^* \varphi (\theta,(z_1,z_2)):= \int_\domt \varphi( \theta,z_1 + tz_2,t) \wrt \mtime(t) \]
is well-defined and, as a direct computation shows, has dual $\amv$. Then, existence of $\tilde\liftVar \in \Mp(\alldirs \times \projdomdyn)$ such that $\amv \tilde\liftVar  =(\mdirs \prodm \theta \mv_t \liftVar_\theta) \prodm t \mtime $ again follows from \cref{prop:abstractRangeCharacterization} since
\begin{multline*}
 \sup_{\substack{
 \varphi \in C_0(\alldirs \times \projdomstat \times \domt) \\\sup_{\theta \in \alldirs,z \in \projdomdyn}(\amv[]^* \varphi)(\theta,z) \leq 1
 }} \int _\domt \int_\alldirs \int_{\projdomstat} \varphi(\theta,s,t) \wrt \mv _t \liftVar_\theta(s) \wrt \mdirs(\theta) \wrt \mtime (t)\\  
  \leq \sup_{\substack{
  \varphi \in C_0(\alldirs \times \projdomstat \times \domt)  \\
  \sup_{\theta \in \alldirs,(z_1,z_2) \in \projdomdyn} \int_\domt \varphi( \theta,z_1 + tz_2,t) \wrt \mtime(t) \leq 1 }} \sup_{\theta \in \alldirs} \int_\domt \int_{\projdomstat} \varphi(\theta,s,t) \wrt \mv _t\liftVar_\theta(s)\wrt \mtime (t) \\
  \leq \sup_{\theta \in \alldirs} \sup_{\substack{
   \varphi \in C_0(\projdomstat \times \domt) \\
    \sup_{(z_1,z_2) \in \projdomdyn } \int_\domt \varphi( z_1+tz_2,t) \wrt \mtime(t) \leq 1 }}
     \int_\domt \int_{\projdomstat} \varphi(s,t) \wrt \mv _t\liftVar_\theta(s)\wrt \mtime (t) \leq  \sup_{\theta \in \alldirs} \|\liftVar_\theta\|_\M<  \infty.
\end{multline*}
Now by \cref{prop:unique_decomp_move} we can uniquely decompose $\tilde\liftVar=\mdirs \prodm \theta \tilde\liftVar_\theta$ so that
\[\mdirs\prodm\theta\mv\liftVar_\theta=\amv\tilde\liftVar=\mdirs\prodm\theta\mv\tilde\liftVar_\theta\]
and thus $\mv\liftVar_\theta=\mv\tilde\liftVar_\theta$ for $\mdirs$-almost every $\theta \in \alldirs$.
Injectivity of the move operator then implies $\tilde\liftVar=\mdirs\prodm\theta\liftVar_\theta\in\Mp(\alldirs\times\projdomdyn)$.

Now we consider the reverse situation and show that if $(u,\liftVar)$ is admissible for \ref{eq:dim_reduced_general}, then $(u,(\liftVar_\theta)_{\theta\in\alldirs})$ is so for \ref{eq:dim_reduced_product_only} with same cost (where $\liftVar=\mdirs\prodm\theta\liftVar_\theta$ is the unique decomposition of $\liftVar$ from \cref{prop:unique_decomp_move}).
Indeed, as in the proof of \cref{prop:unique_decomp_move} it follows that $\mv_t\liftVar_\theta=\Rs_\theta u_t$ for all $\theta\in\alldirs$ and $t\in\domt$,
and \cref{prop:radon_properties,prop:move_properties} then imply $\|\liftVar\|_\M=\|\mdirs \prodm \theta \liftVar_{\theta}\|_\M=\|\liftVar_{\hat\theta}\|_\M$.

Again, since each admissible point for one problem induces an admissible point of the other with same cost, the statement on the minimizers follows.
\end{proof}
\end{prop}

\removed{For the sake of completeness, let us state here that \cref{thm:modelEquivalences} is nothing more than a summary of \cref{prop:well_posedness_dim_reduced_noisy,prop:existence_reduced_measure_only,prop:equivalence_measure,prop:existence_dim_reduced_product_only,prop:equivalence_product_only}.}

\subsection{Auxiliary Results}

\begin{prop}[Duality-based characterization of operator range]\label{prop:abstractRangeCharacterization}
Let $A:X\to Y$ be a bounded linear operator such that $X$ and $Y$ admit predual spaces $X^\#$ and $Y^\#$ and such that $A^\#:Y^\# \rightarrow X^\#$ is a bounded linear operator whose dual operator is $A$. Then, with $K\subset X^\#$ a cone and $K^\circ\subset X$ the polar cone, we have
\begin{equation*}
y\in A(K^\circ)
\qquad\Leftrightarrow\qquad
\sup_{y^\#\in Y^\#,A^\#y^\#\in B_1^{X^\#}+K}\langle y^\#,y\rangle<\infty\,,
\end{equation*}
where $B_1^{X^\#}$ denotes the closed unit ball in $X^\#$.
Further, in the supremum, $Y^\#$ may also be replaced by a dense subset of $Y^\#$.
\begin{proof}
Given $y\in Y$, define
\begin{equation*}
f:Y^\#\to\R,\,f(y^\#)=\langle y^\#,-y\rangle\,,
\qquad
g:X^\#\to[0,\infty],\,g(x^\#)=\iota_{B_1^{X^\#}+K}(x^\#)
\end{equation*}
where $\iota_S$ denotes the indicator function of a set $S$.0
Then $g$ is continuous in $0$ and $0\in A^\#\domain f=\range A^\#$.
Thus, by Fenchel--Rockafellar duality we have
\begin{equation*}
\inf_{y^\#\in Y^\#}f(y^\#)+g(A^\#y^\#)
=\sup_{x\in X}-f^*(Ax)-g^*(-x)\,.
\end{equation*}
Noting $g^*(x)=\|x\|_{X}+\iota_{K^\circ}(x)$ and $f^*(\tilde y)=\iota_{-y}(\tilde y)=0$ if $\tilde y=-y$ and $\infty$ else, the above equality turns into
\begin{equation*}
\inf_{y\in Y^\#,A^\#y^\#\in B_1^{X^\#}+K}\langle y^\#,-y\rangle
=\sup_{x\in K^\circ,Ax=y}-\|x\|_{x}
\begin{cases}
=-\infty&\text{if }y\notin A(K^\circ)\,,\\
>-\infty&\text{else.}
\end{cases}
\end{equation*}
This directly implies the desired equivalence.
Since $y^\#\mapsto\langle y^\#,-y\rangle$ is continuous on $Y^\#$ and $A^\#(\lambda y^\#)\in\mathrm{int}(B_1^{X^\#}+K)$ for any $A^\#y^\#\in B_1^{X^\#}+K$ and $\lambda\in[0,1)$,
the supremum stays the same when only taken over a dense subset of $Y^\#$.
\end{proof}
\end{prop}

\subsection*{Acknowledgements}
The work is supported by the European Fund for Regional Development (ERDF) within the project TRACAR: Non-invasive Imaging  of Pharmacokinetics and Optimization of CAR T-Cell Therapies for Solid Tumors.
It was further supported by the Deutsche Forschungsgemeinschaft (DFG) under Germany's Excellence Strategy EXC 2044 -- 390685587, Mathematics M\"unster: Dynamics-Geometry-Structure,
and by the Alfried  Krupp  Prize  for  Young  University  Teachers  awarded  by  the  Alfried  Krupp  von  Bohlen  und Halbach-Stiftung. MH acknowledges support by the Austrian Science Fund (FWF) (Grant J 4112).

\bibliography{lit_dat,mh_lit_dat,as_lit_dat}
\bibliographystyle{abbrv}

\end{document}